\documentclass[reqno]{amsart}
\usepackage{enumitem}
\usepackage{amssymb,mathtools}
\usepackage{hyperref}
\usepackage{graphicx} 
\usepackage{xcolor}
\usepackage[margin=1in]{geometry}
\usepackage{cite}

 \usepackage{mathtools}
 \mathtoolsset{showonlyrefs=true}

\newcommand{\cA}{{\mathcal{A}}}
\newcommand{\cB}{{\mathcal{B}}}

\newcommand{\cI}{{\mathcal I}}
\newcommand{\cH}{{\mathcal H}}
\newcommand{\sF}{{\mathsf{F}}}
\newcommand{\cL}{{\mathcal{L}}}
\newcommand{\G}{{\mathcal{G}}}
\newcommand{\N}{{\mathbb N}}
\newcommand{\R}{{\mathbb R}}

\newcommand{\T}{{\mathbb T}}

\newcommand{\cV}{{\mathcal{V}}}
\newcommand{\cW}{{\mathcal{W}}}
\newcommand{\cR}{{\mathcal{R}}}

\renewcommand{\d}{\mathrm{d}}
\newcommand{\ep}{\varepsilon}

\renewcommand{\u}{\boldsymbol{u}}
\newcommand{\bv}{\boldsymbol{v}}

\newcommand{\n}{\mathbf{n}}

\newtheorem*{theorem*}{Theorem}
\DeclareMathOperator\Div{div}
\DeclareMathOperator{\vspan}{span}
\DeclareMathOperator{\Id}{Id}

\newcommand{\norm}[1]{\left\|#1\right\|}
\newcommand{\abs}[1]{\left\vert#1\right\vert}
\newcommand{\avg}[1]{\left\langle #1\right\rangle}

\DeclareSymbolFont{bbold}{U}{bbold}{m}{n}
\DeclareSymbolFontAlphabet{\mathbbold}{bbold}
\newcommand{\ind}{\mathbbold{1}}

\newtheorem{proposition}{Proposition}[section]
\newtheorem{theorem}[proposition]{Theorem}
\newtheorem{lemma}[proposition]{Lemma}
\newtheorem{corollary}[proposition]{Corollary}

\theoremstyle{definition}
\newtheorem{definition}[proposition]{Definition}

\newtheorem*{question}{Question} 
\newtheorem{example}[proposition]{Example} 
\newtheorem{remark}[proposition]{Remark}

\theoremstyle{plain}
\newtheorem{informaltheorem}{Theorem}

\numberwithin{equation}{section}

\title{Stationary Structures of the Boussinesq equations}

\author{Marc Nualart}
\address{Instituto de ciencias matem\'{a}ticas, consejo superior de investigaciones cient\'{i}ficas, 28049 Madrid, Spain}
\email{marc.nualart@icmat.es}

\author{Jaemin Park}
\address{Department of Mathematics, Yonsei University, Yonseiro 50, Seodaemungu, Seoul 03722, South Korea}
\email{jpark776@yonsei.ac.kr}

\author{Javier Pe\~{n}afiel-Tom\'as}
\address{Institute of Mathematics, 
\'Ecole Polytechnique F\'ed\'erale
de Lausanne (EPFL), 1015 Lausanne, Switzerland}
\email{javier.penafieltomas@epfl.ch}

\begin{document}

\begin{abstract}
    We investigate the steady configurations supported by the inviscid two-dimensional Boussinesq equations in the periodic channel. We show steady laminar velocity fields (closed non-contractible fluid trajectories) coupled to non-decreasing density stratifications are necessarily stratified shear flows. The direction of density monotonicity along level sets is in fact fundamental: in striking contrast with the Euler equations, we exhibit laminar non-stagnant non-shear stationary solutions in the periodic channel. We also provide distinct flexibility and rigidity results (laminar persistence or appearance of islands for nearby non-shear solutions) close to stratified shear flows according to their stagnation points. As a by-product of our analysis, we establish a uniform Crandall-Rabinowitz bifurcation theorem that might be of independent interest.
\end{abstract}

\maketitle

\setcounter{tocdepth}{1}

\tableofcontents

\section{Introduction}
\label{intro:section}
In this paper we consider the stationary two-dimensional Boussinesq
equations posed in the periodic channel
$M\coloneqq \mathbb{T}\times(-1,1)$:
\begin{equation}\label{eq:BoussinesqVel}
\begin{cases}
\u\cdot\nabla\u+\nabla p=-\mathfrak{g}\rho e_2,\\
\u\cdot\nabla\rho=0,\\
\nabla\cdot\u=0,
\end{cases}
\end{equation}
where $e_2=(0,1){\in \R^2}$, together with the boundary condition
\[
\u\cdot n=0\quad\text{on }\partial M.
\]
Here $\u=(u,v)$ is the velocity, $\rho$ is the density variable,
and $p$ is the pressure. {For simplicity, we set the gravity constant $\mathfrak{g}=1$ throughout the manuscript.}

The Boussinesq system is a classical model for incompressible flows
in which variations of density are neglected in the inertial terms
but retained in the buoyancy force; see, for instance,
\cite{Pedlosky1987,Vallis2017}. Originating from the Boussinesq
approximation for stratified fluids, it provides one of the simplest
models in which fluid motion is coupled to an advected scalar through
gravity. The equations arise naturally in the study of atmospheric
and oceanic flows, thermal convection, and density-stratified fluids
\cite{Chandrasekhar1961,Majda2003}. In the inviscid setting, they may
be viewed as the incompressible Euler equations coupled to a
transported density field.
\newpage
Every pair of one-dimensional profiles $u$ and $\rho$ determines
a stratified shear equilibrium
\begin{equation}\label{intro:shear}
\u(y)=(u(y),0),
\qquad
\rho=\rho(y).
\end{equation}
Our aim is to understand the rigidity and flexibility of stratified
shear equilibria by providing answers to the following question:
\begin{question}\label{main_question}
\textit{Under what conditions must a stationary solution be a stratified shear flow,
and when can equilibria with nontrivial horizontal dependence occur?}
\end{question}

{For the Euler equations {($\rho\equiv 1$)}, the above question is posed for shear flows and has been recently answered in different directions. Regarding geometric characterizations, laminar steady Euler flow (closed, non-contractible fluid trajectories) in the channel/annulus is necessarily shear/radial, see \cite{drivas2024geometric}, and non-vanishing stationary solutions in the infinite channel, plane and annulus inherit the domain symmetries as well, \cite{hamel2017shear, hamel2019liouville, HM23, drivas2024geometric}. On the other hand, studying the stationary structures nearby specific steady states sheds light on the potential long-time limits of the Euler equations nearby them. In \cite{LinZeng}, Lin and Zeng established a regularity dichotomy for the steady states nearby Couette flow $u(y)=y$ in the periodic channel $M$: for all $\tau>0$, all steady states sufficiently close to Couette in $H^{\frac52+\tau}$ must be shear flow (local rigidity), and there exist non-shear steady solutions (local flexibility) arbitrarily close to Couette in $H^{\frac52-\tau}$.}

 \subsection{Main results}
We begin with rigidity results, giving conditions on steady Boussinesq flows that constrain them to be stratified shears. We later show the existence of {sufficiently smooth} non-shear solutions close to some stratified shear flows.

\subsubsection{\textbf{Global rigidity under laminar structure}}

{Inspired by the laminar definition in \cite{drivas2024geometric},} we first fix the geometric terminology used below.

\begin{definition}\label{def_laminar_shear}
An autonomous incompressible vector field $\u$ is called
\emph{laminar} if there exist a function
$\overline{\psi}\in C^1([-1,1])$, independent of the horizontal
variable, and a $C^1$ orientation-preserving
diffeomorphism $X:M\to M$ satisfying
\[
X(\mathbb T\times\{1\})=\mathbb T\times\{1\},\quad X(\mathbb T\times\{-1\})=\mathbb T\times\{-1\},
\qquad
\u=\nabla^\perp\bigl(\overline{\psi}\circ X^{-1}\bigr).
\]
We say that $\u$ admits a \emph{$C^k$ laminar representation} if both $X$ and $\overline{\psi}$ can be chosen to be of class $C^k$. Note that in this case $\u \in C^{k-1}(\overline{M})$. {Moreover, if $\u\in C^{k-1}(\overline{M})$, $|\u|>0$ in $\overline{M}$ and $\u\cdot n = 0$ on $\partial M$ then $\u$ admits a $C^k$ laminar representation.} A velocity field whose stream function is independent of the horizontal variable is called a \emph{shear flow}. See Figure~\ref{fig1}.

\begin{figure}[h]
\includegraphics[scale=1.1]{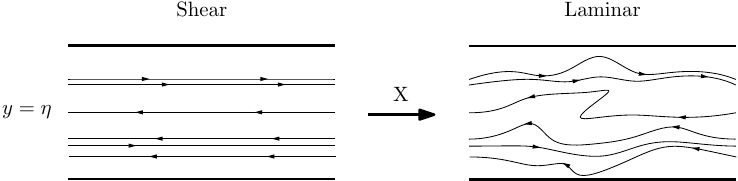}
\caption{Shear and laminar flows}
\label{fig1}
\end{figure}

For a fixed laminar representation, we call a leaf
$\Gamma_\eta:=X(\mathbb T\times\{\eta\})$
\emph{stagnant} if $\u$ vanishes identically on it,
and \emph{non-stagnant} otherwise.

A scalar function $\rho$ is called \emph{nondecreasing
across the leaves} if it is constant on every non-stagnant
leaf and, for any two non-stagnant leaves $\Gamma_{\eta_1},\Gamma_{\eta_2}$ with
$\eta_1<\eta_2$,
\[
\rho|_{\Gamma_{\eta_1}}
\leq \rho|_{\Gamma_{\eta_2}}.
\]
\end{definition}

For a laminar velocity field, the curves
\[
\mathbb T\ni\xi\longmapsto X(\xi,\eta),
\qquad \eta\in[-1,1],
\]
form a foliation of $M$ by simple closed
\emph{non-contractible} curves, that is, curves not homotopic
to a point in $M$. The velocity is tangent to each leaf,
and every non-stagnant leaf is a streamline.
Laminarity itself imposes no condition on the density
and allows entire open regions of stagnation.
Every shear flow is laminar, with $X$ equal to the identity.

Our first rigidity result concerns laminar stationary solutions
whose density is nondecreasing across the non-stagnant leaves. 

\begin{informaltheorem}[\textbf{=Theorem~\ref{laminar_monotone_rigidity_formal}}]\label{informal_rigidity_1}
Let $(\u,\rho)\in {C^1(\overline{M})}$ be a stationary
solution to \eqref{eq:BoussinesqVel}. Assume that
\begin{enumerate}[label=\textbf{(A\arabic*)}]
\item \label{laminar}
\textnormal{(Laminar structure)}
$\u$ admits a {$C^2$} laminar representation.
\item \label{density_monT}
\textnormal{(Density monotonicity)}
$\rho$ is nondecreasing across the leaves.
\end{enumerate}
Then $(\u,\rho)$ is a shear equilibrium, that is, $\partial_x\u=0$ and $\partial_x\rho=0$.
\end{informaltheorem}

    Note that the density transport equation in \eqref{eq:BoussinesqVel}
already ensures that $\rho$ is constant on each non-stagnant
leaf. Thus, the additional restriction in
assumption~\ref{density_monT} concerns the monotonicity of density
across distinct leaves and the resulting buoyancy forces, {and in general it cannot be dropped in view of the example below.}

\begin{example}\label{rmk:nontrivialsteady}
Fix $k_1,k_2\in\mathbb Z\setminus\{0\}$, set
$\mu=k_1^2+4\pi^2k_2^2$, and define
\[
\psi_A(x,y):=-\frac{y}{\mu}
+A\sin(k_1x)\sin(2\pi k_2y),
\qquad A\in\mathbb R.
\]
A direct computation shows that
\[
\u_A=\nabla^\perp\psi_A:=(-\partial_y\psi_A,\partial_x\psi_A),\qquad
\rho_A=\psi_A
\]
form a stationary solution. If $A\ne0$, this state has nontrivial horizontal dependence.

For $2\pi|k_2A|<1/\mu$, we have $\partial_y\psi_A<0$ throughout $M$.
The flow is therefore non-stagnant and laminar, with streamlines
given by periodic graphs winding around the channel.
As $A\to0$, these nontrivial states converge in $C^k$
for every finite $k$ to
\[
\u_0=(1/\mu,0),\qquad \rho_0=-y/\mu.
\]
For $2\pi|k_2A|\geq1/\mu$, the same family has stagnation points
on $y=0$.
\end{example}

{The example illustrates that the stationary structures in the Boussinesq equations are radically different from those of the Euler equations. In stark contrast with the Eulerian setting, where non-vanishing steady flows in the channel must be shear flows, \cite{drivas2024geometric, CDG21, hamel2017shear}, the example shows the existence of $x$-dependent stationary $(\u,\rho)$ solutions with non-vanishing velocity $|\u|>0$. Gravity effects are thus responsible for the Boussinesq-steadiness of configurations that are otherwise unsteady for the Euler equations.}

 \begin{remark}
The theorem allows degeneracy in both the velocity and the density, {and it is the most general Boussinesq analogue to the geometric characterization of steady Euler flow in a channel \cite{drivas2024geometric}.}
The velocity may vanish on individual leaves or throughout open
regions, while the density need not be strictly increasing
across the non-stagnant leaves, in particular, density level sets do not even have to be non-contractible curves in stagnant regions. No nonvanishing assumption on either
$\u$ or $\nabla\rho$ is imposed. See Appendix \ref{app:rigidity} for shorter, simpler proofs for Theorem \ref{informal_rigidity_1} under the additional assumptions of $\u\neq 0$ and $\rho$ non-decreasing across the leaves of $\u$ (Theorem \ref{thm:mainnonzerovelocityrigidity}); and $\nabla \rho \neq 0$ with $\rho(x,-1)<\rho(x,1)$ (Theorem \ref{thm:maindensityrigidity}). In view of Example \ref{rmk:nontrivialsteady}, the assumed monotonicity of $\rho$ in both theorems cannot be dispensed with.
\end{remark}

\begin{remark}
{Density stratifications $\rho$ with $\partial_y \rho\geq 0$ are usually called unstably stratified, since heavier fluid is placed above lighter fluid and buoyancy is expected to trigger unstable dynamics such as growth of vorticity and finite-time blow-up \cite{elgindi2020finite, elgindi2023instability,kiselev2025small,chen2021finite}. Theorem~\ref{informal_rigidity_1}, and also Theorem~\ref{thm:mainnonzerovelocityrigidity}, Theorem~\ref{thm:maindensityrigidity} can be understood in this broader picture: such configurations are not steady unless $\rho=\rho(y)$, so that the buoyancy force $-\rho e_2$ is constant along each horizontal layer of fluid and can be absorbed into an internal pressure. On the other hand, Example~\ref{rmk:nontrivialsteady} shows the existence of stratified flows $\rho$ with $\partial_y\rho \leq 0$, usually called stably stratified, such that they are stationary and $\partial_x\rho\neq 0$: while the stable stratification may accommodate steady vertical motion, they may also develop growth \cite{bianchini2026nonexistence}}.
\end{remark}

 \color{black}Concerning the proof of
Theorem~\ref{informal_rigidity_1}, we first reduce the problem
to one in which both the stream function and the density
are nondecreasing across the leaves.
More precisely, under assumptions~\ref{laminar}
and~\ref{density_monT}, we modify these quantities while
preserving the stationary Boussinesq equations.
The modified pair is a weak stationary solution whose
stream function and density are nondecreasing functions
of the leaf label
(Proposition~\ref{monotonicity_lemma}).
We then prove that this modified solution is a shear
equilibrium (Theorem~\ref{Rigidity_modified_solutions}),
which implies the same conclusion for the original solution.

The second observation concerns the variational structure
of the stationary equations; a stationary solution
$(\u=\nabla^\perp\psi,\rho)$ satisfies
\begin{align}\label{energy_variation_reformulation}
\int_M
\bigl(\Delta\psi\,\nabla\psi+y\nabla\rho\bigr)
\cdot\boldsymbol{v}\,{\d x}\,{\d y}=0
\end{align}
for every $\boldsymbol{v}\in C^1(\overline M;\mathbb R^2)$
such that
\[
\nabla\cdot\boldsymbol{v}=0
\quad\text{in }M,
\qquad
\boldsymbol{v}\cdot n=0
\quad\text{on }\partial M.
\]
This identity formally expresses stationarity of the functional
\[
\mathcal A(\u,\rho)
:=\frac12\int_M|\u|^2\,{\d x}\,{\d y}
-\int_M y\rho\,{\d x}\,{\d y}
\]
under a measure-preserving transport of the stream
function and density.
Indeed, let
\[
\partial_t\widetilde\psi
+\boldsymbol{v}\cdot\nabla\widetilde\psi=0,
\qquad
\partial_t\widetilde\rho
+\boldsymbol{v}\cdot\nabla\widetilde\rho=0,
\qquad
(\widetilde\psi,\widetilde\rho)|_{t=0}=(\psi,\rho).
\]
Then, formally,
\[
\left.\frac{\d}{\d t}
\mathcal A(\nabla^\perp\widetilde\psi,\widetilde\rho)
\right|_{t=0}
=
\int_M
\bigl(\Delta\psi\,\nabla\psi+y\nabla\rho\bigr)
\cdot\boldsymbol{v}\,{\d x}\,{\d y}
=0.
\]

With this observation, we design a suitable perturbation by means of the continuous
Steiner symmetrization \cite{brock1995continuous,Brock,Brock2}, resulting in a  local symmetry
of the stream function in the sense of Brock
(Definition~\ref{local_symmetry_def}).

Local symmetry alone, however, does not ensure that a solution
must be a shear flow.
To obtain the desired conclusion, we apply the symmetrization
argument to reflected extensions with arbitrary horizontal
shifts motivated by \cite{drivas2024geometric} (Subsection~\ref{sliding_reflection}).
The local symmetries obtained for these shifts then yield that
the stream function must be independent of the horizontal
variable.\color{black}

\subsubsection{\textbf{Global rigidity based on Richardson-number nonresonance}}
  As shown in the example, neither laminar geometry nor the absence of stagnation alone forces a stationary state to be a shear. We investigate how additional conditions on the density and velocity ensure rigidity and how they interact with the topology of approximation to determine the stationary structures nearby.
A different rigidity mechanism
arises from the local relation between the velocity shear
and the density gradient where the horizontal velocity vanishes.

For a shear equilibrium
$((u(y),0),\rho(y))$,
the Richardson number is defined by
\[
{\mathcal J}(y)
:=-\frac{\rho'(y)}{| u'(y)|^2}
\]
whenever $u'(y)\neq 0$. {It quantifies the competing stabilizing effects of density stratification and shear stresses. The celebrated Miles-Howard stability theorem \cite{miles1961stability, howard1961note} ensures spectral stability of stratified shear flows with $\mathcal{J}(y)\geq \frac14$, and $\mathcal{J}(y)$ has been recently seen to govern linear inviscid damping decay rates for a large family of stably stratified strictly monotone shear flows \cite{enciso2025linear}.}
At simple stagnation points, certain discrete values of the Richardson number are associated with possible finite orders of vanishing
of the vertical velocity. Excluding these values yields
the following rigidity result.

\begin{informaltheorem}[\textbf{Theorem~\ref{thm:rigidity_noncrit_J}}]\label{informal_rigidity_2}
Let $(\u,\rho)$ be a smooth stationary Boussinesq solution,
with $\u=(u,v)$ and a nonempty zero set $\Gamma_u:=\{u=0\}$.
Suppose that $|\partial_yu|>0$ on $\Gamma_u$ and that
$\partial_y\rho$ is not identically zero on any of its
connected components. If
\[
\mathcal J:=-\frac{\partial_y\rho}{|\partial_yu|^2}
\notin\{-m(m-1):m=2,3,\ldots\}
\qquad\text{on }\Gamma_u,
\]
then $\u$ is a shear flow and $\rho$ is stratified.
\end{informaltheorem}
Note that no laminarity or global density-monotonicity assumption is imposed in the above theorem. As a corollary, we obtain local rigidity near shear equilibria with non-resonant simple stagnation points. {The proof uses the monotonicity of $u$ near its stagnation curve and the steady Boussinesq equations to show that any critical curve is in fact a flat horizontal curve, where the vertical velocity $v$ vanishes to infinite order and $\partial_y\rho$ is a non-zero constant. The vertical velocity $v$ is then seen to satisfy a degenerate elliptic PDE near the flat critical line, and a suitable Carleman estimate for $v$ propagates its boundary vanishing to the whole domain, obtaining a shear flow.}

{As a direct corollary of Theorem \ref{informal_rigidity_2} we obtain a local rigidity theorem for stratified shear flows.}

\begin{informaltheorem}[\textbf{Theorem~\ref{thm:rigidity_noncrit_shear_dens}}]\label{informal_rigidity_3}
Let $((u(y),0),\rho(y))$
be a shear equilibrium whose zeros form a smooth nonempty finite
set of simple interior stagnation points $y_i$. If
\[
{\mathcal J}(y_i)
\notin\{-m(m-1):m=1,2,\ldots\}
\]
at every stagnation point, then every smooth stationary
Boussinesq solution sufficiently close to this equilibrium
in \textcolor{black}{$C^2$} is a shear equilibrium.
\end{informaltheorem}

The density monotonicity in Theorem~\ref{informal_rigidity_1}
requires heavier fluid to lie above lighter fluid.
Although this arrangement can cause disturbances to grow
in time, it need not allow them to persist as stationary
states. A nontrivial stationary state requires fluid motion
and buoyancy to balance exactly. Within the laminar class,
the density monotonicity prevents this balance: the
symmetrization argument shows that the energy variations
can cancel only when the flow is a shear.
The theorem therefore allows dynamical instability
while ruling out nontrivial stationary configurations.

 \subsubsection{\textbf{Flexibility and local rigidity with non-stagnant velocity field}} We now turn to flexibility and ask whether nontrivial stationary
states can be found near a given shear while retaining its
laminar structure. We begin with shears without stagnation.

\begin{informaltheorem}
[\textbf{=Theorem~\ref{thm:nonzeroflexibility}}]
\label{informal_nonstagnant_flexibility}
Let $(( u_0(y),0),\rho_0(y))\in {C^4(\overline{M})}$ be a shear equilibrium with $u_0(y)\neq 0$ for all $y\in[-1,1]$.
For every $1\leq q<\infty$ and every $\varepsilon>0$,
there exists a nontrivial  laminar stationary
Boussinesq solution $(\u,\rho)$ satisfying
\[
\|\u-(u_0,0)\|_{L^q}
+\|\rho-\rho_0\|_{W^{1,q}}
<\varepsilon.
\]
\end{informaltheorem}

The construction first modifies the background profiles
in a thin region to obtain a nearby smooth shear flow and density stratification such that the shear velocity is non-zero but very small (thus the $L^q$ regularity) and the density derivative is strictly negative (hence the $W^{1,q}$ regularity). Smooth ($\u,\rho\in H^3(M)$) nontrivial
laminar states bifurcate from these modified profiles by means of a uniform Crandall-Rabinowitz theorem. 

\begin{remark}\label{remark no stagnations non-laminar}
The constructed steady states bifurcate from a modified non-zero velocity field and are thus \emph{laminar}. Without this requirement, constructing non-shear stationary solutions close to
$((u_0,0),\rho_0)$ in
$L^q\times W^{1,q}$ is much simpler.
One first modifies the background profiles so that the
density is constant and the velocity vanishes in a thin
horizontal strip. Within this strip, the constant buoyancy
force can be absorbed into the pressure, leaving the
stationary Euler equations. One can then insert a smooth,
compactly supported non-shear Euler flow inside the strip as in \cite{drivas2024geometric}. See Figure~\ref{figEuler}. The resulting stationary Boussinesq solution is non-laminar.
By taking the strip sufficiently thin, these modifications
can be made arbitrarily small in $L^q\times W^{1,q}$
for every finite $q$.
\end{remark}

\vfill
\begin{figure}[h]
\includegraphics[scale=1]{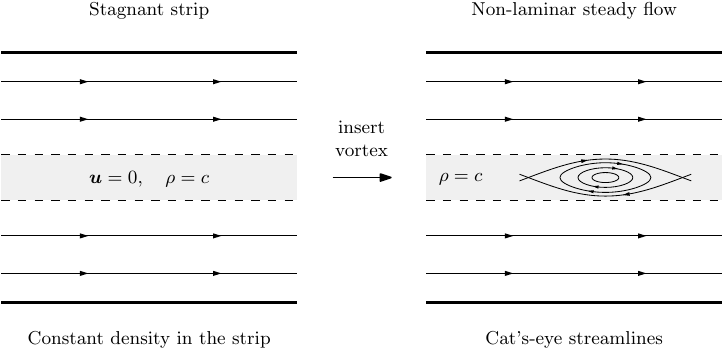}
\caption{Euler-like non-laminar stationary solutions}
\label{figEuler}
\end{figure}
\newpage
Theorem~\ref{informal_nonstagnant_flexibility} applies even
when the background density $\rho_0$ is strictly
increasing. There is no contradiction with
Theorem~\ref{informal_rigidity_1}: small perturbations in
$W^{1,q}$ can destroy the density monotonicity across the leaves.
By contrast, if $\rho_0'>0$ throughout $M$,
sufficiently small perturbations in $W^{1,\infty}$
preserve strict vertical monotonicity, and the global
rigidity theorem excludes nontrivial laminar states.
Thus, the restriction to finite $q$ in
Theorem~\ref{informal_nonstagnant_flexibility}
cannot in general be removed.

While the preceding theorem establishes flexibility
in $L^q\times W^{1,q}$ for finite $q$, our next result
shows that stronger control of the velocity perturbation
restores local rigidity under a coercivity condition
on the background shear, with the density still measured
in $W^{1,q}$. To state the coercivity condition, define
\begin{equation}\label{quadratic_coercivity}
Q_{\mathrm{sh}}(w):=
\int_M\left(
 u_0^2|\nabla w|^2
+\rho_0'|w|^2
\right)\,{\d x}\,{\d y},
\qquad w\in H_0^1(M).
\end{equation}
See Remark~\ref{motivation_quadform_intro} for a motivation of the quadratic form $Q_{\mathrm{sh}}$.

\begin{informaltheorem}
[\textbf{=Theorems~\ref{thm rigidity non-zero Linfty}
and~\ref{thm rigidity non-zero Sobolev}}]
\label{informal_nonstagnant_rigidity}
Let $(( u_0(y),0),\rho_0(y))$ be a smooth
shear equilibrium. Assume that there exist $m,c_0>0$ such that
\[
 u_0\geq m,
\qquad
Q_{\mathrm{sh}}(w)\geq c_0\|w\|_{H^1(M)}^2
\quad\text{for all }w\in H_0^1(M).
\]
For every $1<q<\infty$, the following  hold.
\begin{enumerate}[label=(\arabic*)]
\item \textbf{($L^\infty$ velocity proximity)}
There exists $\varepsilon_0=\varepsilon_0(q,m,c_0)>0$ such that any stationary
Boussinesq solution $(\u,\rho)$ satisfying
\[
\|\u-(u_0,0)\|_{L^\infty}
+\|\rho-\rho_0\|_{W^{1,q}}
<\varepsilon_0
\]
is a shear equilibrium.

\item \textbf{(Sobolev velocity proximity)}
For every $s>\frac12$, there exists $\varepsilon_1=\varepsilon_1(s,q,m,c_0)>0$
such that any stationary Boussinesq solution
$(\u,\rho)$, with $\u=(u,v)$, satisfying
\[
\inf_M u>0,
\qquad
\|\u-(u_0,0)\|_{H^s}
+\|\rho-\rho_0\|_{W^{1,q}}
<\varepsilon_1
\]
is a shear equilibrium.
\end{enumerate}
In both cases, $\partial_x\u=0$ and $\partial_x\rho=0$.
\end{informaltheorem}

\begin{remark}
For $s>1$, assertion~(2) follows from assertion~(1)
by the Sobolev embedding
$H^s(M)\hookrightarrow L^\infty(M)$.
The main difficulty lies in the range $\frac12<s\le1$:
although $\inf_M u>0$ is assumed, the $H^s$ proximity
does not provide a uniform positive lower bound for $u$.
The proof overcomes this difficulty through a different
energy argument in streamline coordinates.
\end{remark}

\begin{remark}\label{motivation_quadform_intro}
The coercivity condition in
Theorem~\ref{informal_nonstagnant_rigidity}
is motivated by an equation for the streamline slope.
For a stationary solution $\u=(u,v)$ with $u>0$,
a direct computation shows (see \eqref{vorticity_eq_rigidt2}) that the ratio $w:=v/u$ satisfies
\[
-\nabla\cdot(u^2\nabla w)+\partial_y\rho\,w=0,
\qquad w|_{\partial M}=0.
\]
Linearizing this equation about the background shear
gives the operator associated with $Q_{\mathrm{sh}}$:
\[
\mathcal L_{\mathrm{sh}}w
:=-\nabla\cdot( u_0^2\nabla w)
+\rho_0' w.
\]
Thus, coercivity rules out nonzero solutions of the
linearized stationary equation. For the full nonlinear equation, testing against $w$
gives
\[
\int_M\left(
u^2|\nabla w|^2+\partial_y\rho\,|w|^2
\right)\,{\d x}\,{\d y}=0.
\]
For the $L^\infty$ assertion of
Theorem~\ref{informal_nonstagnant_rigidity},
sufficiently small perturbations in
$L^\infty\times W^{1,q}$ preserve this coercivity,
forcing $w=0$ and hence $v=0$.

In particular, coercivity holds whenever
$u_0\ge m>0$ and $\rho_0'\ge0$.
For a shear flow, this density condition is precisely
the monotonicity across the leaves assumed in
Theorem~\ref{informal_rigidity_1}.
The present local result also allows negative density
gradients, provided their contribution is sufficiently
small relative to the positive gradient term.
\end{remark}

The preceding local rigidity result relies on the background
velocity being bounded away from zero. We next examine
what happens when the background shear has stagnation
points, beginning with bifurcations that retain
the laminar structure.

\subsubsection{\textbf{Flexibility and rigidity with stagnant velocity field}}

We next consider background shears with stagnation points.
In this case, laminar flexibility persists even when
the velocity perturbation is small in any prescribed
finite $C^k$ norm, while the density is still measured
in $W^{1,q}$.

\begin{informaltheorem}
[\textbf{=Theorem~\ref{thm:zeroflexibility}}]
\label{informal_stagnant_laminar_flexibility}
Let $((u_0(y),0),\rho_0(y))$ be a smooth
shear equilibrium with $u_0(y_0)=0$ for some
$y_0\in[-1,1]$.
For every integer $k\ge1$, every $1\le q<\infty$,
and every $\varepsilon>0$, there exists a nontrivial
laminar stationary Boussinesq solution $(\u,\rho)$
satisfying
\[\|\u-(u_0,0)\|_{{C^{k}}}+\|\rho-\rho_0\|_{W^{1,q}}<\varepsilon.\]
\end{informaltheorem}
\begin{remark}
In contrast with the non-stagnant setting of
Theorem~\ref{informal_nonstagnant_rigidity},
stronger control of the velocity alone does not restore
local rigidity near a stagnant shear when the density
is measured in $W^{1,q}$. The leaves of the laminar foliation in the constructed
solutions are periodic graphs over the horizontal variable.
This geometry also shows why the density proximity
cannot, in general, be strengthened to $W^{1,\infty}$.
Indeed, if $\rho_0'>0$ throughout $M$, sufficiently
small $W^{1,\infty}$ perturbations preserve
$\partial_y\rho>0$. Since the density is constant on
each non-stagnant leaf and the leaves are vertically
ordered graphs, it is nondecreasing across those leaves.
Theorem~\ref{informal_rigidity_1} then forces the solution
to be a shear equilibrium.
\end{remark}

The construction deforms the stream function and density
through a common change of the vertical coordinate,
preserving the laminar foliation.
The equation governing this deformation degenerates
where the background velocity vanishes.
By modifying the density profile near a stagnation point,
we arrange for nontrivial deformations with the prescribed
finite regularity to bifurcate.
These modifications are small in $W^{1,q}$ for finite $q$,
but need not be small in $C^1$.
In particular, when the background shear satisfies
the assumptions of Theorem~\ref{informal_rigidity_3},
the density perturbation must keep the constructed
states outside the joint $C^1$ neighborhood
in which that theorem guarantees rigidity.

We next ask whether nearby stationary states can leave
the laminar class itself.
When the density gradient also vanishes at a simple
stagnation point, we obtain non-laminar approximations
with both fields close in $C^{1,\alpha}$ in the spirit of \cite{LinZeng}.

\begin{informaltheorem}
[\textbf{=Theorem~\ref{theorem:nonlaminar_flexibility}}]
\label{informal_stagnant_nonlaminar_flexibility}
Let $((u_0(y),0),\rho_0(y))\in C^4(\overline{M})$ be a stratified shear equilibrium with a unique stagnation point
$y_0\in(-1,1)$. Assume that
\begin{equation}\label{eq:monotone_vel_zero_density_gradient}
 u_0'(y_0)\neq0,
\qquad
\rho_0'(y_0)=0.
\end{equation}
For every $0<\alpha<1$ and every $\varepsilon>0$,
there exist  non-laminar stationary Boussinesq
solutions $(\u,\rho)$ satisfying
\[
\|\u-(u_0,0)\|_{C^{1,\alpha}}
+\|\rho-\rho_0\|_{C^{1,\alpha}}
<\varepsilon.
\]
\end{informaltheorem}

The construction first makes the density constant
in a thin region around the stagnation point.
The condition $\rho_0'(y_0)=0$ allows this
modification to be small in $C^{1,\alpha}$ for any fixed $0\leq\alpha<1$.
Within that region, the buoyancy force can be absorbed
into the pressure, reducing the equations to the stationary
Euler equations. We then obtain non-laminar states
by bifurcating from a modified shear, following the
approach of
\cite{LinZeng,sinambela2023transition,zhao2026traveling}. Reducing to Euler in a constant-density region also
underlies the simpler construction discussed in
Remark~\ref{remark no stagnations non-laminar};
see also \cite[Theorem 1.3]{drivas2024geometric}.
The distinction here is the control of the velocity:
the bifurcation construction gives proximity in
$C^{1,\alpha}$, including the first derivatives,
rather than merely $C^{0,\alpha}$.

{Let us remark that the appearance of islands in Theorem \ref{informal_stagnant_nonlaminar_flexibility} takes place in the region of constant density. As such, the density stratification is still laminar. It is then an open question to exhibit other steady states with non-laminar density other than the simple example shown in Example \ref{rmk:nontrivialsteady}.}

\begin{remark}
Under the same assumptions, non-laminar stationary
solutions can also approximate the background shear
in Sobolev spaces. More precisely, for every $\delta>0$
and every $\varepsilon>0$, such a solution can be chosen
to satisfy
\[
\|\u-(u_0,0)\|_{H^{\frac52-\delta}}
+\|\rho-\rho_0\|_{H^{\frac52-\delta}}
<\varepsilon.
\]
\end{remark}

These results describe two ways in which flexibility
occurs near stagnant shears.
Laminar deformations are possible with arbitrarily
high finite control of the velocity when the density
is measured in $W^{1,q}$.
If the density gradient also vanishes at a simple
stagnation point, non-laminar states can occur with
both fields close in $C^{1,\alpha}$.
The latter result permits background densities with
$\rho_0'\ge0$: even near such backgrounds,
the global rigidity in
Theorem~\ref{informal_rigidity_1} does not exclude
approximations outside the laminar class. 

{We finish our discussion showing the sharpness of the regularity threshold in Theorem \ref{informal_stagnant_nonlaminar_flexibility} under a spectral assumption, completing the regularity dichotomy of Lin and Zeng in \cite{LinZeng} in the Boussinesq equations. Requiring spectral control for a local rigidity result has been necessary also in the in-homogeneous Euler equations, see \cite{zhao2026traveling}. For $u_0,\rho_0$ as in Theorem \ref{informal_stagnant_nonlaminar_flexibility}, let
\begin{align}
    q_0(y)\coloneqq  \frac{\partial_y\rho_0(y)}{u_0(y)}, \quad Q_0(y) \coloneqq u_0''(y) + q_0(y)
\end{align}
and define the Rayleigh operator
\begin{align}
    \mathcal{R}_0\phi\coloneqq  u_0(y)\Delta\phi - Q_0(y)\phi
\end{align}
for $\phi\in H^2(M)\cap H_0^1(M)$.}
\begin{informaltheorem}
[\textbf{=Theorem~\ref{thm:nonlaminar_rigidity}}]
\label{informal_stagnant_nonlaminar_rigidity}
Let $((u_0(y),0),\rho_0(y))\in C^4(\overline{M})$ be a smooth
shear equilibrium with a unique stagnation point
$y_0\in(-1,1)$ satisfying the assumptions of Theorem \ref{informal_stagnant_nonlaminar_flexibility}. Assume further that \textcolor{black}{$\inf_M u_0'>0$ and }$\mathcal{R}_0$ has trivial kernel in $H^2(M)\cap H_0^1(M)$. For every $0<\alpha<1$ there exists $\varepsilon>0$ such that
any stationary Boussinesq
solution $(\u,\rho)$ satisfying
\[
\|\u-(u_0,0)\|_{C^{2,\alpha}}
+\|\rho-\rho_0\|_{C^{2,\alpha}}
<\varepsilon
\]
must be a stratified shear flow, that is, $\partial_x\u=\partial_x\rho = 0$.
\end{informaltheorem}

\subsection{Related works}
Describing the stationary structures of inviscid fluid equations has greatly attracted the attention of the community these past few years. Undoubtedly, most efforts have focused on the two-dimensional Euler equations, yielding a fairly complete picture of their steady states in several settings. One main avenue of research investigates how the geometry of steady solutions is determined by the symmetries of the domains they occupy. {Apart from the previously mentioned rigidity results \cite{hamel2017shear, hamel2019liouville, HM23, drivas2024geometric}, we refer to \cite{de2026rigidity, cao2027free, gui2024classification, huang2026flexibility, gomez2021remarks, gomez2021symmetry, fan2024radial, fan2026rigidity, gui2026pressure, alonso2026rigidity, tione2026rigidity} for other recent developments in this direction.}

While rigidity for compactly supported velocity fields fails in general, see \cite{enciso2024smooth, enciso2026newapproachcompactlysupported} where non-radial compactly supported smooth Euler steady flows are constructed, one still recovers radial flows if the support of the velocity field is either an open annulus \cite{Ruiz} or a punctured simply connected open domain \cite{WangZhan}.

Some of the above rigidity statements rely on finding a semi-linear elliptic equation $\Delta\psi=F(\psi)$ for some $F$ and studying associated over-determined boundary value problems: both Dirichlet and Neumann conditions usually conspire with the semilinear equation to only support one-dimensional solutions, thus recovering shear flows or radial flows according to the domain. 

Despite the existence of smooth Morse steady states with no global functional relation, see \cite{elgindi2026flexibility}, recently $\omega=F(\psi)$ was derived in \cite{elgindi2024classification} for all analytic steady states $\omega$ on analytic simply connected domains that are not the ball. Once $\omega=F(\psi)$ is established and the linearised operator $\Delta - F'(\psi)$ is invertible, one can investigate the correspondence between steady states and co-adjoint orbits \cite{choffrut2012local}, and the deformation of streamlines along fluid vessel variations \cite{CDG21, drivas2024islands, drivas2025fragility}.

The study of stationary structures near Couette of \cite{LinZeng} has recently been extended in \cite{guo2026flexibility} to the infinite channel. The flexibility/rigidity regularity dichotomy for Poiseuille flow $u_0(y)=y^2$ was settled in \cite{drivas2024geometric}, while \cite{coti2023stationary} found non-trivial non-shear steady states close to Kolmogorov flow in $\T^2$ $u_0(y)=\sin(y)$ in analytic regularity. See \cite{nualart2023zonal, elgindi2025regular, constantin2025classification, sinambela2023transition} for other results nearby specific steady states, and \cite{CastroLear23,CastroLear24,CastroLear24b, castro2025time, enciso2025uniformly} for traveling-wave structures. We refer the interested reader to the survey \cite{hmidi2026rotating} and the references therein for recent rigidity and flexibility concerning vortex patches.

Steady states of the three-dimensional Euler equations have also seen a remarkable development. Steady 3D Euler flows with compact support are known to exist \cite{gavrilov2019steady, constantin2019remarks}, they also enjoy symmetric rigidity in some settings \cite{CDG21, peralta2026symmetry}, there exist steady states that are isolated from other stationary solutions \cite{enciso2025isolated}, and they can approximate any prescribed divergence-free flow topology \cite{enciso2025steady}. 

The steady states of the Boussinesq equations are far less understood. We mention \cite{CDG21}, where rigidity of non-stagnant flows in the periodic channel was established for the first time under the condition $\Theta'(\psi_0)\leq 0$ for decreasing $\psi_0$. We extend this result in Theorem \ref{thm:mainnonzerovelocityrigidity}. Stationary self-similar profiles for the Boussinesq equations were later studied in \cite{abe2026stationary}. See also \cite{zhao2026traveling} for flexibility and rigidity results close to Couette flow in the inhomogeneous Euler equations.

\color{black}
\subsection*{Organization of the paper}
Section~\ref{sec:steady_reformulation} presents the stream function formulation and
the semilinear elliptic equations used in our constructions.
Section~\ref{sec:Rigidity_l} proves global rigidity under laminar geometry
and density ordering, using a monotone modification of
the solution and continuous Steiner symmetrization (Theorem~\ref{informal_rigidity_1}).
Section~\ref{sec:rigidity_richardson} establishes rigidity under Richardson-number
nonresonance through an analysis near stagnation curves
and a unique continuation argument (Theorem~\ref{informal_rigidity_2}-\ref{informal_rigidity_3}).

Sections~\ref{sec: bifurcation non-zero}-\ref{sec:local_rigidity non-zero_2} concern non-stagnant shear equilibria (Theorem~\ref{informal_nonstagnant_flexibility}-\ref{informal_nonstagnant_rigidity}).
We first construct nearby nontrivial laminar states
by bifurcation, then establish local rigidity under
a coercivity condition in stronger velocity topologies.
Section~\ref{sec:bifurcation_-zero_01} treats laminar bifurcation near shears
with stagnation points (Theorem~\ref{informal_stagnant_laminar_flexibility}).
Section~\ref{sec:local_rigidity-zero_012} constructs non-laminar states near shears
whose density gradient vanishes at stagnation and
proves further local rigidity results (Theorem~\ref{informal_stagnant_nonlaminar_flexibility}-\ref{informal_stagnant_nonlaminar_rigidity}).

Appendix~\ref{app:rigidity} contains rigidity results under nonvanishing
velocity or density gradient assumptions.
Appendix~\ref{CR_conti_appendix} provides the parameter-dependent
Crandall--Rabinowitz theorem used in the bifurcation
constructions, {which may be of independent interest}\color{black}
\subsection*{Acknowledgements}
The research of MN was partially supported by the European Research Council (ERC) under the European Union's Horizon 2020 research and innovation programme through the grant agreement~862342. JP was partially supported by the Yonsei University Research Fund of 2025-22-0437, the
POSCO Science Fellowship of POSCO TJ Park Foundation, HYUNSONG Educational and Cultural Foundation, and  the National Research Foundation of Korea (NRF) grant funded by the Korea government (MEST) No. RS-2026-25476891. JPT was supported by the Swiss State Secretariat for Education, Research and Innovation (SERI) under contract number MB22.00034 through the project TENSE.

\subsection*{AI declaration}
This work began in May 2025, before the widespread use of LLMs became a practice, and by August 2026 the main structure and results of the manuscript were laid out by the authors. In September 2026, in view of the current capabilities of LLMs, we used several models as a guide to complete and simplify technical arguments, such as the absence of kernel in Lemma \ref{lemma: control eigenvalues no stag}, the Carleman estimates of Section \ref{sec:rigidity_richardson} and the Sobolev rigidity Theorem \ref{thm rigidity non-zero Sobolev}. All AI-generated suggestions, computations, and references were independently verified and, where necessary, corrected or reformulated by the authors. The resulting document is entirely human written.

 \section{Reformulation of stationary solutions into a semilinear elliptic equation}\label{sec:steady_reformulation}
  In this section, we formulate the stationary equations into a semilinear elliptic equation for a stream function and the density. This formulation identifies quantities that are constant along streamlines and allows us to construct a class of stationary solutions through a semilinear elliptic equation. It reveals the geometric structure of stationary solutions and plays a crucial role throughout the paper.
  
     Let $(\u=(u,v),\rho)$ be a stationary solution and let $\psi$ be a stream function for the incompressible vector field $\u$, that is, $\u=\nabla^\perp \psi:=(-\partial_y\psi,\partial_x\psi)$. The momentum equation in \eqref{eq:BoussinesqVel} takes the form
\begin{align}\label{vorticity_velocity_1rigid}
\Delta\psi\,\nabla\psi=\nabla P+\rho e_2,
\qquad P=p+\frac12|\u|^2.
\end{align}
Using $\rho e_2=\nabla(y\rho)-y\nabla\rho$, we rewrite it as
\begin{equation*}
\Delta\psi\,\nabla\psi+y\nabla\rho=\nabla B,
\qquad B=P+y\rho.
\end{equation*}
Taking the inner product with $\nabla^\perp\psi$ and using
$\nabla^\perp\psi\cdot\nabla\rho=0$, we obtain
\[
\nabla^\perp\psi\cdot\nabla B=0.
\]
Thus, both the density $\rho$ and  $B$ (referred to as the Bernoulli function) are first integrals of $\u$, that is, they are constant along
streamlines. Together with the density transport equation, the reformulated momentum equation is equivalent to the original stationary system. The following proposition records this equivalence and a sufficient condition for constructing stationary solutions.
\begin{proposition}\label{Bernoullie_function_2}
Let $(\u,\rho)\in C^1(\overline{M})$ be a stationary solution to the Boussinesq equations \eqref{eq:BoussinesqVel} and let 
\[
B:=p+\frac{1}{2}|\u|^2+y\rho.
\] Then, for a stream function $\psi$ such that $\u=\nabla^\perp \psi$, it holds that
\begin{align}
\Delta\psi\,\nabla\psi+y\nabla\rho&=\nabla B,\label{first_firstintegral1}\\
\nabla^\perp\psi \cdot\nabla \rho&=0.\label{second_firstintegral1}
\end{align}
Conversely, if $\psi\in C^2(\overline{M})$ and $\rho\in C^1(\overline{M})$ satisfy
\eqref{first_firstintegral1}--\eqref{second_firstintegral1} for a scalar function $B\in C^1(\overline{M})$, then $(\u:=\nabla^\perp\psi,\rho)$ satisfies  \eqref{eq:BoussinesqVel}, with
\[
p:=B-\frac12|\nabla\psi|^2-y\rho.
\]
Consequently, if $\psi\in C^2(\overline{M})$ satisfies
\begin{align}\label{semi_elliptic_stream}
\Delta\psi+y\Theta'(\psi)+G'(\psi)=0, \text{ for some $\Theta,G\in C^1$,}
\end{align} then
\[
\u=\nabla^\perp\psi,\qquad
\rho=\Theta(\psi),\qquad
p=-G(\psi)-\frac12|\nabla\psi|^2-y\Theta(\psi)
\]
satisfy \eqref{eq:BoussinesqVel}.
In both cases, the impermeability condition holds if $\psi$
is constant on each component of $\partial M$.
\end{proposition}
\begin{proof}
The forward implication that a stationary solution $(\u,\rho)$ satisfies \eqref{first_firstintegral1}--\eqref{second_firstintegral1} has already been proved above. For the converse, we notice that $\u=\nabla^\perp\psi$ and
\eqref{second_firstintegral1} give $\nabla\cdot\u=0$ and  $\u\cdot\nabla\rho=0$, respectively. Moreover, using \eqref{first_firstintegral1}, we obtain
\begin{align*}
(\u\cdot\nabla)\u+\nabla p
&=-\Delta\psi\nabla\psi + \nabla\left(\frac12|\nabla\psi|^2 + p \right)\\
&=-\nabla B + y\nabla\rho + \nabla \left(\frac12|\nabla\psi|^2 + p \right)\\
&=\nabla\left(-B+y\rho + \frac12|\nabla\psi|^2 + p \right)-\rho e_2\\
&=-\rho e_2,
\end{align*}
where the second equality follows from \eqref{first_firstintegral1}. Thus, $(\u,\rho)$ satisfies \eqref{eq:BoussinesqVel}.
Finally, if
$\Delta\psi+y\Theta'(\psi)+G'(\psi)=0$,
set $\rho:=\Theta(\psi)$ and $B:=-G(\psi)$. Then
\[
\nabla^\perp\psi\cdot\nabla\rho=0,
\qquad
(\Delta\psi)\nabla\psi+y\nabla\rho
=-G'(\psi)\nabla\psi=\nabla B,
\]
so the conclusion follows from the converse.
\end{proof}

Before we close this section, we recall the vorticity formulation of the momentum equation. For $\u=(u,v)$, the vorticity $\omega:=\nabla^\perp\cdot \u  =\partial_x v-\partial_yu$ satisfies
\begin{align}\label{vorticity_eq_rigidt1}
\u\cdot\nabla \omega = -\partial_x\rho, \quad \text{in $M$}.
\end{align}
We will use the following form of the equation frequently
\begin{align}\label{vorticity_eq_rigidt2}
u^2\Delta v - u(\Delta u) v - \partial_y \rho v = 0 \qquad \text{in $M$}.
\end{align}
Indeed, noting that $\partial_x\omega = \Delta v$, $\partial_y \omega = - \Delta u$ and combining \eqref{vorticity_eq_rigidt1} with the density transport equation $u\partial_x\rho+v\partial_y\rho=0$, we have
\[
u\left( u\Delta v -v\Delta u\right) = u\left( \u\cdot\nabla \omega \right)= -u\partial_x\rho =v\partial_y\rho,
\]
which gives \eqref{vorticity_eq_rigidt2}.

\begin{remark}\label{rmk:rhoctnpartialM}
The density $\rho$ is constant on each connected component of $\partial M$, regardless of the elliptic formulation of the steady equations. Indeed, for $\u\in C^2(\overline{M})$ and $\rho\in C^1(\overline{M})$ on $\partial M$ we have $v=0$ so that the vorticity equation and the density  transport equation read
\begin{align*}
    u\partial_x\omega = - \partial_x \rho, \quad u\partial_x \rho = 0.
\end{align*}
If $\partial_x\rho\neq 0$ somewhere on $\partial M$, the transport equation forces $u=0$ there. The vorticity equation then implies $\partial_x \rho=0$, reaching a contradiction.
\end{remark}


\section{Rigidity of steady states}\label{sec:Rigidity_l}
In this section we prove the main rigidity characterizations for stationary Boussinesq equations with laminar velocity and monotone density. For notational purposes of the proof, we pose the equations in a smaller domain $M_-:=\mathbb{T}\times (-1,0)$. That is, we consider
\begin{align}\label{Boussinesq_P_1}
\begin{cases}
\u\cdot\nabla\u+\nabla p=-\rho e_2,\\
\u\cdot\nabla\rho=0,\\
\nabla\cdot\u=0,
\end{cases} \text{ for } (x,y)\in M_-,\quad \text{ and }\quad \u\cdot n=0\text{ on $\partial M_-$.}
\end{align}
Let us recast the laminarity  in the domain $M_-$ as
\begin{enumerate}[label=\textbf{(H\arabic*)},ref=\textbf{(H\arabic*)}]
\item\label{density_mon_r12}  There exists $\overline{\psi}\in C^2([-1,0])$ and a $C^2$ orientation-preserving diffeomorphism $X=(X_1,X_2):\overline{M_-}\to \overline{M_-}$ satisfying
\[
X(\mathbb T\times\{-1\})=\mathbb T\times\{-1\},\quad X(\mathbb T\times\{0\})=\mathbb T\times\{0\},\quad
\qquad
\u=\nabla^\perp\bigl(\overline{\psi}\circ X^{-1}\bigr).
\]
\end{enumerate}
See Figure~\ref{fig1} for an illustration. For a laminar velocity, we denote the stream function by $\psi$, that is,
\begin{align}\label{streaM_-aminar_rigiditiy_capsule}
\psi:=\overline{\psi}\circ X^{-1}.
\end{align}
Since $X$ is a diffeomorphism, it is invertible and its inverse is 
\begin{align}\label{inverse_mapping_rg}
X^{-1}=:(Z,Y),\text{ that is, }X(Z(x,y),Y(x,y))=(x,y)\text{ for all $(x,y)\in M_-$.}
\end{align}
The boundary condition on $X$ ensures that
\begin{align}\label{inverse_mapping_rg_Y}
Y(\mathbb{T}\times \left\{-1\right\})=-1 \text{ and }Y(\mathbb{T}\times \left\{0\right\})=0.
\end{align}
For $\eta\in (-1,0)$, denote the corresponding leaf by
\[
\Gamma_\eta:=X(\mathbb{T}\times \left\{ \eta\right\})=\left\{ Y=\eta\right\}.
\]
Note that every $\eta\in (-1,0)$ is a regular value of $Y$. Indeed, since 
 \begin{align}\label{nondegenerate_DY}
 0 \neq \text{det}(DX^{-1})=\nabla Y\cdot \nabla^\perp Z,
 \end{align}
 we have $\nabla Y\ne 0$ everywhere in the domain. Therefore, 
 \[
 \u=\nabla^\perp (\overline{\psi}(Y))=\overline{\psi}'(Y)\nabla^\perp Y\ne 0,\text{ if $\overline{\psi}'(Y)\ne 0$.}
\]
We notice that under the property \ref{density_mon_r12}, $\inf_{\Gamma_\eta}|\u|>0$, if and only if $\overline{\psi}'(\eta)\ne 0$.
Moreover, the density transport equation in \eqref{Boussinesq_P_1} implies $\rho$ is constant on $\Gamma_\eta$, if $\overline{\psi}'(\eta)\ne 0$. For non-stagnant leaves, we state the monotonicity of density across the leaves as
\begin{enumerate}[start=2,label=\textbf{(H\arabic*)},ref=\textbf{(H\arabic*)}]
\item\label{density_mon_r11} $\rho\in C^1(\overline{M_-})$ and for every $\eta_1,\eta_2\in (-1,0)$ such that $\overline{\psi}'(\eta_1),\overline{\psi}'(\eta_2)\ne 0$, it holds that
\[
\rho(\Gamma_{\eta_1})\ge \rho(\Gamma_{\eta_2}),\text{ if $\eta_1\ge \eta_2$.}
\]
\end{enumerate}

The main goal of this section is to establish the following theorem, whose proof will be given at the end of the next subsection.
\begin{theorem}\label{laminar_monotone_rigidity_formal}
Let $(\u,\rho)$ be a solution to \eqref{Boussinesq_P_1}  satisfying \ref{density_mon_r12} and \ref{density_mon_r11}.  Then $(\u,\rho)$ is a shear equilibrium, that is, $\partial_x\u=0$ and $\partial_x\rho=0$.
\end{theorem}

 In the rest of the section, we will assume $(\u,\rho)$ satisfies \ref{density_mon_r12} and \ref{density_mon_r11}.

\subsection{Monotone modification of solution and the proof of Theorem~\ref{laminar_monotone_rigidity_formal}} In this subsection, we modify the stream function and the density so that they are monotone increasing in a stronger sense and still solve the stationary Boussinesq equations. At the end of this subsection, the proof of Theorem~\ref{laminar_monotone_rigidity_formal} will be given, assuming that such modified solution is a shear equilibrium.

 We first investigate the structure of stagnant sets.  Let
 \begin{align}\label{stagnant_sets_rigidity_1}
 \mathcal{R}:=\left\{ \eta\in [-1,0]: \overline{\psi}'(\eta)\ne 0\right\}.
 \end{align}
 For $\eta\in [-1,0]$, set
 \begin{align}\label{strongly_monotone_density}
 \overline{\rho}(\eta):=\rho(X(0,\eta)),\quad \overline{\theta}(\eta):=\int_{-1}^\eta \ind_{\mathcal{R}}(s)\overline{\rho}'(s){\d s}. 
 \end{align} 
 \begin{lemma}\label{density_indeed_monotone}
 It holds that   $\overline{\theta}\in W^{1,\infty}([-1,0])$ and $\overline{\theta}'\ge 0$ almost everywhere.
   \end{lemma}
 \begin{proof}
 For $\eta_1,\eta_2\in [-1,0]$ with $\eta_1\neq \eta_2$, it is clear that
 \begin{align}\label{theta_lipschitz_conti_1}
|\overline{\theta}(\eta_1)-\overline{\theta}(\eta_2)|\leq \rVert \overline{\rho}'\rVert_{L^\infty}|\eta_1-\eta_2|,
 \end{align}
 where $\overline{\rho}'$ is bounded since $X\in C^2$ and $\rho\in C^1$. Therefore $\overline{\theta}\in W^{1,\infty}$. By Rademacher's theorem, $\overline{\theta}$ is differentiable almost everywhere.
 
Since $\overline{\psi}'$ is continuous, $\mathcal{R}$
is relatively open in $[-1,0]$.
For each $\eta\in\mathcal{R}$, a sufficiently small
neighborhood of $\eta$ consists of non-stagnant labels.
Assumption~\ref{density_mon_r11} therefore implies that
$\overline{\rho}$ is nondecreasing in this neighborhood.
Consequently,
\[
\overline{\rho}'(\eta)\ge0
\qquad\text{for every }\eta\in\mathcal{R}.
\]
Thus, the integrand
$\mathbf{1}_{\mathcal{R}}\overline{\rho}'$ is nonnegative. This implies that  $\overline{\theta}$ is nondecreasing on $(-1,0)$. Hence $\overline{\theta}'\ge0$  almost everywhere.
\end{proof}
Next, let us define
\begin{align}\label{def_monotone_streamfunction1}
\overline{\phi}(\eta)=\int_{-1}^\eta |\overline{\psi}'(s)|{\d s}.
\end{align}
It is clear that $\overline{\phi}$ is monotone increasing, that is
\[
\overline{\phi}'(\eta)\ge 0\text{ for all $\eta\in (-1,0)$.}
\]
With the monotone nondecreasing functions $\overline{\theta}$ and $\overline{\phi}$, we have the following proposition.
 \begin{proposition}\label{monotonicity_lemma}
Let $\overline{\phi}$ and $\overline{\theta}$ be as in \eqref{def_monotone_streamfunction1} and \eqref{strongly_monotone_density}, respectively. Set
\begin{align}\label{defofphiandtheta}
\phi:= \overline{\phi}\circ Y,\quad \bv:=\nabla^\perp \phi,\quad \theta:=\overline{\theta}\circ Y.
\end{align}
Then, $(\bv,\theta)\in W^{1,\infty}$ is a stationary weak solution to the Boussinesq equations, that is,
\begin{align}\label{Boussinesq_P}
\begin{cases}
\bv\cdot\nabla\bv+\nabla q=-\theta e_2,\\
\bv\cdot\nabla\theta=0,\\
\nabla\cdot\bv=0,
\end{cases} \text{ for } (x,y)\in M_-,\quad \text{ and }\quad \bv\cdot n=0\text{ on $\partial M_-$.}
\end{align}
 \end{proposition}
\begin{proof}
The incompressibility of $\bv$ and the boundary condition for $\bv$ follow from the definition of $\bv$ and the fact that $Y$ keeps the boundary (see \eqref{inverse_mapping_rg_Y}). Also, we have
\[
\bv\cdot\nabla \theta = \overline{\phi}'\nabla^\perp Y \cdot (\overline{\theta}'\nabla Y)=0\text{ a.e. in $M_-$.}
\]

By construction,
\begin{align}\label{uandv}
\u=\overline{\psi}'(Y)\nabla^\perp Y,
\qquad
\bv=|\overline{\psi}'(Y)|\nabla^\perp Y.
\end{align}
Since $\overline{\psi}'\in C^1$ and $Y\in C^2$,
we have $\bv\in W^{1,\infty}(M_-)$.
Moreover,
\[
\bv\otimes\bv=\u\otimes\u.
\]
Both vector fields are divergence-free, thus
\begin{align}\label{advection_boussinesq_reform}
(\bv\cdot\nabla)\bv
=\nabla\cdot(\bv\otimes\bv)
=\nabla\cdot(\u\otimes\u)
=(\u\cdot\nabla)\u
\end{align}
in the distributional sense.

We next show that $\rho-\theta$ depends only on $y$.
On the non-stagnant region $Y^{-1}(\mathcal R)$,
the density $\rho$ is constant on each streamline $X(\mathbb{T}\times \left\{\eta\right\})$, therefore
\[
\rho(X(\xi,\eta))=\rho(X(0,\eta))=\overline{\rho}(\eta),\text{ for all $\xi\in\mathbb{T}$ and $\eta\in \mathcal{R}$.}
\]
Since $\overline{\theta}'=\overline{\rho}'$
on $\mathcal R$ (see \eqref{strongly_monotone_density}), it follows that
\begin{align}\label{nonstag_points_samedensity}
\nabla\theta=\overline{\theta}'\nabla Y = \overline{\rho}'\nabla Y = \nabla\rho
\quad\text{on }Y^{-1}(\mathcal R).
\end{align}
On the other hand, the definition of $\overline{\theta}$ in \eqref{strongly_monotone_density}
gives
\[
\overline{\theta}'=0
\quad\text{a.e. on }\mathcal R^c.
\]
Since $X$ is a $C^2$ diffeomorphism of the compact closure,
it is Lipschitz and therefore maps sets of measure zero
to sets of measure zero. Thus, for every null set
$N\subset(-1,0)$,
\begin{align}\label{nulset_estimate_YX}
Y^{-1}(N)=X(\mathbb T\times N),
\qquad |Y^{-1}(N)|=0.
\end{align}
In particular, we have
\begin{align}\label{stag_xindep}
\nabla\theta=0
\qquad\text{a.e. on }Y^{-1}(\mathcal R^c).
\end{align}

On the other hand, the vorticity equation \eqref{vorticity_eq_rigidt1} reads
\begin{align}\label{vorticity_eq_rigidt}
\u\cdot\nabla \omega = - \partial_x\rho.
\end{align}
with $\omega:=\Delta \psi$. For $z_0:=(x_0,y_0)\in Y^{-1}(\mathcal{R}^c)$, we test the vorticity equation with $\frac{1}{\delta^2}\ind_{B_{\delta}(z_0)}$, where $B_{\delta}(z_0)$ is the open ball centered at $z_0$ with radius $\delta>0$, to see that
\begin{align*}
\frac{1}{\delta^2}\int_{B_\delta(z_0)}\partial_x\rho(x,y)\d x\d y &=-\frac{1}{\delta^2}\int_{\partial B_{\delta}(z_0)}\omega \u\cdot\n d\sigma\\
& = -\frac{1}{\delta^2}\int_{\partial B_{\delta }(z_0)} \left(\omega(x,y)-\omega(z_0)\right)\u\cdot \n d\sigma\\
&  =-\frac{1}{\delta^2}\int_{\partial B_{\delta}(z_0)} \left(\omega(x,y)-\omega(z_0)\right) (\u-\u(z_0))\cdot\n d\sigma\\
& =o(1),
\end{align*}
where the second equality follows from the incompressibility of $\u$ and the third equality follows from $\u(z_0)=0$, since $z_0$ is a stagnation point. The last equality follows from \ref{density_mon_r12}, which guarantees that $\psi\in C^2$. Taking $\delta\to 0$, we arrive at
\[
\partial_x\rho(z_0)=0.
\]
Combining this with \ref{stag_xindep}, we have $\partial_x(\rho - \theta)=0\text{ a.e. on $Y^{-1}(\mathcal{R}^c)$}$. Therefore, together with \eqref{nonstag_points_samedensity}, we conclude that 
\[
\partial_x(\rho-\theta)=0\qquad\text{a.e. on $M_-$.}
\] Since both $\rho$ and $\theta$ are $W^{1,\infty}$,  there exists $f\in W^{1,\infty}([-1,0])$ such that 
\[
\rho-\theta=f(y).
\]
Finally, plugging this and \eqref{advection_boussinesq_reform} into the momentum equation in \eqref{Boussinesq_P_1}, we have
\[
(\bv\cdot\nabla)\bv =(\u\cdot\nabla)\u =-\nabla p - \rho e_2= -\nabla p -f e_2 - \theta e_2
\]
Denoting $q(x,y):=p(x,y)+\int_{-1}^yf(s){\d s}$, we obtain the momentum equation in \eqref{Boussinesq_P}.
\end{proof}

 The equations \eqref{Boussinesq_P} are our main object that we will further investigate in order to establish rigidity of laminar stationary solutions. We state our goal to prove about $(\bv,\theta)$ as a theorem, under which we can give a proof for Theorem~\ref{laminar_monotone_rigidity_formal}.
 
 \begin{theorem}\label{Rigidity_modified_solutions}
 $(\bv,\theta)$ in Proposition~\ref{monotonicity_lemma} is a shear equilibrium.
 \end{theorem}
 The proof of this theorem is lengthy. For the reader's convenience, the proof will be carried in the following subsections and given at the end of the section, see Subsection~\ref{prooftheoremRigidity_modified_solutions}.
  Now, we are ready to prove Theorem~\ref{laminar_monotone_rigidity_formal}.
\begin{proof}[Proof of Theorem~\ref{laminar_monotone_rigidity_formal}]
By Theorem~\ref{Rigidity_modified_solutions}, $\bv$ is
a shear flow. In particular, its vertical component
vanishes. Hence from \eqref{uandv} it follows that 
the vertical component of $\u$ also vanishes.
Incompressibility then implies that $\u$ is also a shear.

 Since $\partial_x\u =0$, in the vorticity equation \eqref{vorticity_eq_rigidt}, we have $\u\cdot\nabla \omega =0$ in the distribution sense. Therefore, $\partial_x \rho=0$ as well. 
\end{proof}
In the rest of the section, we will carry out the proof of Theorem~\ref{Rigidity_modified_solutions}.

\subsection{Additional first integral $\mathcal{P}$}
Thanks to Proposition~\ref{monotonicity_lemma}, we know that $(\bv,\theta)$ is a stationary solution. As we observed in Proposition~\ref{Bernoullie_function_2}, the so-called Bernoulli function is constant along each streamline. However, due to the lack of regularity of $(\bv,\theta)$, we need to be more careful to justify such a function indeed satisfies such a property. In this subsection, we establish this rigorously.

 Let us define
  \begin{align}\label{generalized_pressure_P}
\mathcal{P}(x,y):=\Delta \phi(x,y) \overline{\phi}'(Y(x,y)) + y\overline{\theta}'(Y(x,y)),\text{ for $(x,y)\in M_-$.}
\end{align}
Note that from the definition of $\phi$ in \eqref{def_monotone_streamfunction1} and Proposition \ref{monotonicity_lemma}, it follows that $\phi\in W^{2,\infty}$. And we already know that $\theta\in W^{1,\infty}$ from  \eqref{theta_lipschitz_conti_1}. Therefore, $\mathcal{P}\in L^\infty(M_-)$.

 Recall that  $\phi$ and $\theta$ are constant on each level set of $Y=\eta$ by their definitions. In the next lemma, we show that $\mathcal{P}$ is also constant on each level set of $Y$. 
 \begin{lemma}\label{B_condition}
 Let $(\phi,\theta)$ be a steady state.  Then $\mathcal{P}$ defined in \eqref{generalized_pressure_P} is constant on each level set $Y=\eta$. More precisely, there exists a $W^{1,\infty}$ function  $\overline{Q}:[-1,0]\to \mathbb{R}$ such that
 \[
 \mathcal{P}(X(\xi,\eta))=\overline{Q}'(\eta) \text{ for $a.e. \ (\xi,\eta)\in M_-$}.
 \]
  \end{lemma}
 \begin{proof}
As we observed in \eqref{vorticity_velocity_1rigid}, the momentum equation in \eqref{Boussinesq_P} can be written as
 \[
-\Delta\phi \nabla \phi=-\nabla \left(q+\frac{1}{2}|\bv|^2\right) -\theta e_2 = -\nabla \underbrace{\left(q+\frac{1}{2}|\bv|^2+y\theta\right)}_{=:\tilde{B}} + y\nabla \theta.
 \]
 Using $\phi(x,y)=\overline{\phi}(Y(x,y))$ and $\theta(x,y)=\overline{\theta}(Y(x,y))$, we have
 \[
 \left(\Delta \phi \overline{\phi}'(Y(x,y))+y\overline{\theta}'(Y(x,y))\right)\nabla Y(x,y) = \nabla \tilde{B}(x,y),
 \]
 in other words,
 \begin{align}\label{P_condition_1}
 \mathcal{P}(x,y)\nabla Y(x,y)=\nabla \tilde{B}(x,y).
 \end{align}
 Now, define
 \[
 \overline{Q}(\xi,\eta):=\tilde{B}(X(\xi,\eta)).
 \]
 Since $X,\phi\in W^{2,\infty}$ and $\theta\in W^{1,\infty}$, we have $\tilde{B}\in W^{1,\infty}$, consequently, $\overline{Q}\in W^{1,\infty}$.
 
 We claim that $\overline{Q}$ is independent of $\xi$. Indeed, taking the gradient of $\tilde{B}(x,y)=\overline{Q}(Z(x,y),Y(x,y))$, we have
 \[
 \nabla \tilde{B} = \partial_{\xi}\overline{Q}\nabla Z + \partial_\eta \overline{Q}\nabla Y.
 \]
 Taking the inner product with $\nabla^\perp Y$ on both sides and using \eqref{P_condition_1}, we obtain
 \begin{align*}
 \nabla \tilde{B} \cdot  \nabla^\perp Y&=\mathcal{P}(x,y)\nabla Y(x,y)\cdot\nabla^\perp Y=0,\\
 \left( \partial_{\xi}\overline{Q}\nabla Z + \partial_\eta \overline{Q}\nabla Y\right)\cdot\nabla^\perp Y&=\partial_\xi\overline{Q}\nabla Z\cdot\nabla^\perp Y=-\partial_\xi\overline{Q}\operatorname{det}(DX^{-1}),
 \end{align*}
 where the last equality follows from \eqref{nondegenerate_DY}. Therefore, $\overline{Q}$ is independent of $\xi$ and consequently, $\tilde{B}(x,y)=\overline{Q}(Y(x,y))$. Moreover, since $\overline{Q}$ is $W^{1,\infty}$ as a function on $M_-$, it is $W^{1,\infty}$ as a function on $[-1,0]$.  Plugging this into \eqref{P_condition_1}, we obtain
 \[
 \mathcal{P}\nabla Y = \partial_\eta \overline{Q}(Y)\nabla Y.
 \]
 Since $\nabla Y\neq 0$,  we obtain $\mathcal{P} = \partial_\eta \overline{Q}(Y)$. Thus $\mathcal{P}$ is constant on each level set of $Y$, and $\mathcal{P}(X(\xi,\eta))=\partial_\eta\overline{Q}(\eta)$ for almost every $(\xi,\eta)\in M_-$
   \end{proof}
  
  \subsection{Sliding parameter $\lambda$ and extension of the domain}\label{sliding_reflection} In this section, we introduce a parameter $\lambda\in\mathbb{R}$ and extend the equations to a larger domain $M=\mathbb{T}\times (-1,1)$. This approach is motivated by \cite{drivas2024geometric}, where a sliding parameter was combined with the usual continuous Steiner symmetrization. In our setting, we carry out the symmetrization argument in more detail.
  
    For a given $\lambda$, we define
  \begin{align}\label{def_Y_extended}
  Y_\lambda(x,y):=\begin{cases}
  Y(x,y)&\text{ if $(x,y)\in \mathbb{T}\times [-1,0]$,}\\
  Y(x-\lambda,-y)&\text{ if $(x,y)\in\mathbb{T}\times (0,1]$,}\\
  -1 & \text{ if $|y|>1$.}
  \end{cases}
  \end{align} By the boundary conditions on $Y$ in \eqref{inverse_mapping_rg_Y}, $Y_\lambda$ is $C^2$ in $M\backslash \left(\mathbb{T}\times \left\{ 0 \right\} \right)$ and Lipschitz continuous in $\mathbb{T}\times \mathbb{R}$ (see Figure~\ref{Slided_symmetrization} for an illustration). In view of the representations of $\phi$ and $\theta$ in \eqref{defofphiandtheta}, we define
 \begin{align}\label{def_psi_rho_B}
 \phi_\lambda(x,y):=\overline{\phi}(Y_\lambda(x,y)),\quad \theta_\lambda(x,y):=\overline{\theta}(Y_\lambda(x,y)), \text{ for $(x,y)\in \mathbb{T}\times\mathbb{R}$}.
 \end{align}
 According to our definition of $\overline{\theta}$ and $\overline{\phi}$, we  have $\overline{\phi}(-1)=\overline{\theta}(-1)=0$ so that their monotonicity  implies that $\phi_\lambda$ and $\theta_\lambda$ are nonnegative in $M$.
  \begin{figure}[h!]
\includegraphics[scale=1.1]{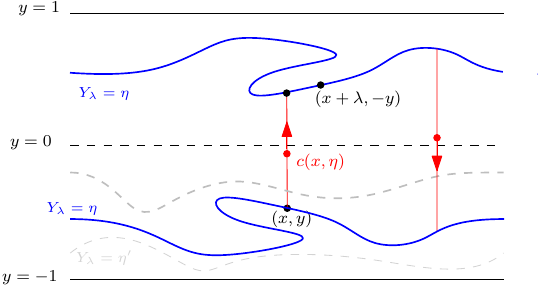}
\caption{Illustration of continuous Steiner symmetrization after $\lambda$-sliding}
\label{Slided_symmetrization}
\end{figure}
 Next, for $\overline{Q}$ in Lemma~\ref{B_condition}, we define
 \begin{align}\label{blambda_def}
 \mathcal{P}_\lambda(x,y):=\overline{Q}'(Y_\lambda(x,y)),\text{ for $(x,y)\in M$.}
 \end{align}
In the lower half of the domain, the quantities $Y,\phi,\theta,$ and $\mathcal{P}$ remain unchanged, and therefore \eqref{generalized_pressure_P} still holds,
\[
\mathcal{P}_\lambda(x,y)=\Delta \phi_\lambda(x,y) \overline{\phi}'(Y_\lambda(x,y)) + y\overline{\theta}'(Y_\lambda(x,y)) \text{ for $(x,y)\in M_-$.}
\]
For $y>0$, we can compute $\mathcal{P}_\lambda$  as
\begin{align*}
\mathcal{P}_\lambda(x,y) &= \overline{Q}'(Y_\lambda(x,y)) = \overline{Q}'(Y(x-\lambda,-y))=\mathcal{P}(x-\lambda, -y)\\
& = \Delta \phi(x-\lambda,-y) \overline{\phi}'(Y(x-\lambda,-y)) - y\overline{\theta}'(Y(x-\lambda,-y))\\
& = \Delta \phi_\lambda(x,y)\overline{\phi}'(Y_\lambda(x,y))-y\overline{\theta}'(Y_\lambda(x,y)).
\end{align*}

Hence, we have 
\begin{align}\label{B_aftersliding}
\mathcal{P}_\lambda(x,y)=\Delta \phi_\lambda(x,y)\overline{\phi}'(Y_\lambda(x,y))-|y|\overline{\theta}'(Y_\lambda(x,y)),\text{ for a.e. $(x,y)\in M\backslash(\mathbb{T}\times \left\{0\right\})$.}
\end{align}
\begin{remark}\label{discontinuity_overthe_line}
In general, $\partial_2Y_\lambda$ and $\partial_2\phi_\lambda$ are not necessarily continuous across $\mathbb{T}\times\{0\}$, but are only bounded. Consequently, the distributional Laplacian of $\phi_\lambda$ may contain a singular measure supported on $\mathbb{T}\times\{0\}$. In \eqref{B_aftersliding}, however, $\Delta\phi_\lambda$ denotes the classical Laplacian on $M\setminus(\mathbb{T}\times\{0\})$, or equivalently its almost-everywhere defined regular part. Thus, \eqref{B_aftersliding} is understood as an almost-everywhere identity and does not include the singular part of the distributional Laplacian.
\end{remark}

\begin{remark}\label{lambda_unchanged}
In Subsections~\ref{Steiner_onesided_two_dim}-\ref{energy_symmetrization_po}, we will symmetrize the stream function $\phi_\lambda$ and the density $\theta_\lambda$ while $\lambda$ is fixed. All computations throughout these subsections do not depend on $\lambda$, especially Corollary~\ref{each_energy_vanishes} holds for any $\lambda\in\mathbb{R}$.
\end{remark}

\subsection{Preliminaries for continuous Steiner symmetrization}\label{Steiner_onesided_intro} In this subsection, we recall Brock's  continuous Steiner symmetrization \cite{brock1995continuous,Brock,Brock2}. 

 We begin by defining the symmetrization for bounded measurable sets in $\mathbb{R}$.
\begin{definition}\cite[Definition 2.1, Theorem 2.1]{Brock}\label{steiner_original} For $t\ge0$,
\begin{enumerate}[label=(\roman*)]
\item\label{CSS_def1} If $E=[a,b]$ is a  bounded closed interval, then  $
\mathcal{S}_t(E):=[a^t,b^t],$
where \begin{align}\label{def_movement}
a^t:=c^t-\frac{1}{2}l,\quad b^t:=c^t+\frac{1}{2}l,\quad c^t:=e^{-t}c,\quad c:=\frac{b+a}{2},\quad l:=b-a.
\end{align}
\item\label{CSS_def2} If $E$ is a finite union of bounded disjoint compact intervals, $E=\cup_{n=1}^NI_n$, we define $\mathcal{S}_t(E):=\cup_{n=1}^N \mathcal{S}_t(I_n)$ for $0< t< \tau_1$, where $\tau_1$ is the first time that any two or more closed intervals collide. At this time, we merge each collection of colliding intervals into a single closed interval, and repeat the process starting from $t=\tau_1$. 
\item \label{CSS_def4} If $E$ is a finite union of bounded pairwise disjoint open intervals, we define $\mathcal{S}_t(E):=\operatorname{Int}(\mathcal{S}_t(\overline{E}))$.
\item\label{CSS_def3} If $E$ is a countable union of bounded pairwise disjoint open intervals, $E=\cup_{n=1}^\infty I_n$, then we define $\mathcal{S}_t(E):=\cup_{N=1}^\infty \mathcal{S}_t(\cup_{n=1}^N I_n)$.
\item\label{Css_def_5} If $E$ is a bounded measurable set, it has a representation
\[
E=\cap_{n=1}^\infty I_n,
\]
where $I_{n+1}\subset I_n$, $n=1,2,\ldots,$ are open sets, up to a set of measure zero. Then we define
\[
\mathcal{S}_t(E):=\cap_{n=1}^\infty \mathcal{S}_t(I_n).
\]
\end{enumerate}
\end{definition}

\begin{remark}
Strictly speaking, for \ref{CSS_def3} to be well-defined, one must verify
that the resulting set is independent of the enumeration of the open
components of $E$. This follows from the monotonicity of the
symmetrization; see the proof of \cite[Theorem 2.1]{Brock}. The resulting
construction agrees with Brock's precise \emph{pointwise} representative of the
continuous symmetrization of open sets defined in
\cite[Definition 2.2]{Brock}. 
\end{remark}
 
 \begin{lemma}\cite[Definition 2.1]{Brock}\label{basic_prop_lem1}
 For any bounded measurable sets $E,F\subset \mathbb{R}$, the following properties hold for $t\ge0$:
 \begin{enumerate}[label=(\alph*)]
 \item $|\mathcal{S}_t(E)|=|E|$ (equimeasurability),
 \item If $E\subset F$, we have $\mathcal{S}_t(E)\subset\mathcal{S}_t(F)$ (monotonicity),
 \item $\mathcal{S}_{t+s}(E)=\mathcal{S}_{t}(\mathcal{S}_s(E))$ for $t,s\ge 0$ (semigroup property).
 \end{enumerate}
 \end{lemma}
 
 The next lemma will be used to estimate the potential energy variation under the symmetrization.
 \begin{lemma}\label{yweight_variation}
 Let $E\subset \mathbb{R}$ be a bounded open set. Then
 \[
 \int_{E}|y| {\d y} \ge \int_{\mathcal{S}_t(E)}|y|{\d y}.
 \]
 \end{lemma}
\begin{proof}
First, suppose that $E$ is a finite union of bounded
pairwise disjoint open intervals so that 
\[
\mathcal{S}_t(E)=\cup_{n=1}^N (e^{-t}c_n-l_n, e^{-t}c_n+l_n),\text{ for some $c_n\in \mathbb{R}$ and $l_n>0$}
\]
for $t\ge 0$ before the first collision.
Then a direct computation shows that
\[
\int_{\mathcal{S}_t(E)}|y|\,{\d y}
=
\sum_{\left\{n,\ e^{-t}|c_n|\ge l_n\right\}}2l_n e^{-t}|c_n| + \sum_{\left\{n,\ e^{-t}|c_n|<l_n\right\}}\left(e^{-2t}c_n^2+l_n^2\right)
\]
Each summand is nonincreasing in $t$, hence the integral is nonincreasing before the first
collision. At a collision, merging the intervals does
not change the integral. Repeating the argument after
each collision proves the assertion for all $t\ge0$.

For a general bounded open set, write
\[
E=\bigcup_{j=1}^{\infty}I_j,
\qquad
E_N:=\bigcup_{j=1}^N I_j,
\]
where the $I_j$ are its pairwise disjoint open components.
By  Definition~\ref{steiner_original}\ref{CSS_def3} and  monotonicity in Lemma~\ref{basic_prop_lem1},
\[
\cup_{N=1}^\infty E_N= E,
\qquad
\cup_{N=1}^\infty \mathcal S_t(E_N) =\mathcal S_t(E).
\]
The monotone convergence theorem therefore gives
\[
\int_{\mathcal S_t(E)}|y|\,{\d y}
=
\lim_{N\to\infty}\int_{\mathcal S_t(E_N)}|y|\,{\d y}
\le
\lim_{N\to\infty}\int_{E_N}|y|\,{\d y}
=
\int_E|y|\,{\d y}.
\]
\end{proof}
 
 We now turn to two-dimensional sets. For $A\subset \mathbb{T}\times\mathbb{R}$ we denote
 \[
 A_x:=\left\{ y\in  \mathbb{R}: (x,y)\in A\right\}.
 \]
 \begin{definition}\label{twod_steiner_sym}
 For a bounded measurable set $A\subset \mathbb{T}\times \mathbb{R}$, we define its continuous Steiner symmetrization \emph{about $y=0$} by
 \[
 \mathcal{S}_t(A):=\left\{(x,y)\in \mathbb{T}\times  \mathbb{R}: y\in \mathcal{S}_t(A_x)\right\}.
 \]
 \end{definition}
 For a set $A\subset \mathbb{T}\times  \mathbb{R}$, the fibers $A_x$ are periodic in $x$, in the sense that  $A_{x+2\pi}=A_x$. Therefore, $\mathcal{S}_t(A)_{x+2\pi}= \mathcal{S}_t(A)_{x}$.
 The following lemma is an immediate consequence of Lemma~\ref{basic_prop_lem1}, so we omit the proof.
 \begin{lemma}\label{basic_prop_lem2}
 For open sets $A,B\subset \mathbb{T}\times  \mathbb{R}$, the following properties hold for $t,s\ge0$:
 \begin{enumerate}[label=(\arabic*)]
 \item\label{basic_2DsyM_-em1} If $A\subset B$, then $\mathcal{S}_t(A)\subset\mathcal{S}_t(B)$.
 \item\label{basic_2DsyM_-em2} $\mathcal{S}_{t+s}(A)=\mathcal{S}_t(\mathcal{S}_s(A))$.
 \item\label{basic_2DsyM_-em4} $|\mathcal{S}_t(A)|=|A|$.
 \end{enumerate}
 \end{lemma}
 
For a bounded nonnegative continuous function $f:\mathbb{T}\times  \mathbb{R} \to \mathbb{R}$ with compact support, we recall its layer-cake representation:
\[
f(x,y)=\int_{0}^{\sup f} \ind_{A(\eta)}(x,y)\d \eta, \text{ where }A(\eta):=\left\{(x,y)\in\mathbb{T}\times  \mathbb{R}: f(x,y)>\eta\right\}.
\]
Since $f$ is continuous, $A(\eta)$ is open. We define the continuous symmetrization of $f$ as
\begin{align}\label{def_conti_f}
f^t(x,y):=\int_{0}^{\sup f}\ind_{\mathcal{S}_t(A(\eta))}(x,y)\d\eta.
\end{align}
We have (see \cite[Remark 2.5]{Brock})
\begin{align}\label{equimeasurability_symmetrizedf}
\left\{ (x,y)\in\mathbb{T}\times\mathbb{R}: f^t(x,y)>\eta\right\} = \mathcal{S}_t(A(\eta)),\text{ up to a set of measure zero, for all $\eta>0, t\ge0$.}
\end{align}
Note that, in general, exact equality need not hold, since the continuous symmetrization merges disjoint open intervals whose closures collide. More precisely, we have
\begin{align}
\left\{y\in\mathbb{R}: f^t(x,y)>\eta\right\} &= \mathcal{S}_t(A(\eta)_x)\text{ up to a set of measure zero (in $\mathbb{R}$)},\label{onedversion_levelset1}\\
\left\{y\in\mathbb{R}: f^t(x,y)>\eta\right\}&\subset \mathcal{S}_t(A(\eta)_x)\subset \left\{y\in\mathbb{R}: f^t(x,y)\ge \eta\right\}\,\label{onedversion_levelset2},
\end{align}
for all $x\in\mathbb{T},\ \eta>0$, and $t\ge0$.

\begin{lemma}\label{nondecreasing_conjugation}
Let $G:[0,\infty)\to [0,\infty]$ be a nondecreasing continuous function such that $G(0)=0$. Then
\[
G(f^t)=G(f)^t.
\]
\end{lemma}
\begin{proof}
 For $0\le s<G(\sup f)$, define
 $
\gamma(s):=\sup\{r\in[0,\sup f]:G(r)\le s\}.$
Since $G$ is continuous and nondecreasing, we have
$
\{G(f)>s\}=\{f>\gamma(s)\}.$
Hence, by \eqref{equimeasurability_symmetrizedf},
\[
\begin{aligned}
\{G(f^t)>s\}=\{f^t>\gamma(s)\}=\mathcal S_t(\{f>\gamma(s)\})=\mathcal S_t(\{G(f)>s\})=\{G(f)^t>s\}
\end{aligned}
\]
up to a set of measure zero. The layer-cake representation then yields the desired result.
\end{proof}

 In the remainder of this subsection, we collect several useful properties of the symmetrization $f^t$ from \cite{Brock}. Strictly speaking,
the properties presented below are stated in \cite{Brock} for functions on
$\mathbb R^n$. However, the proofs extend to
$\mathbb T\times\mathbb R$ with minor modifications, since the
symmetrization acts only in the $y$-variable and the arguments in the
transverse variables are local. We therefore omit the proofs.

 \begin{lemma}\label{regularity_lemma_steiner}
Let $f:\mathbb{T}\times  \mathbb{R}\to \mathbb{R}$ be a nonnegative bounded Lipschitz  function with compact support. Then the following hold:
\begin{enumerate}[label=(\alph*)]
\item\label{lipshcitaz_steiner} $\rVert \nabla f^t\rVert_{L^{p}}\le \rVert \nabla f\rVert_{L^{p}}$ for all $t\ge0$ and $p\in [1,\infty]$.
\item\label{time_cosnti_steiner1} $\rVert f^t-f\rVert_{L^p}\le tR\rVert \partial_y f\rVert_{L^p}$, for all $t\ge0$ and $p\in [1,\infty]$, where $R:=\inf\left\{r>0: \text{supp}(f)\Subset \mathbb{T}\times (-r,r)\right\}$.
\end{enumerate}
\end{lemma}
The proofs of Properties~\ref{lipshcitaz_steiner} and \ref{time_cosnti_steiner1} can be found in
\cite[Theorem 3.2, Theorem 4.2]{Brock}, respectively. 

Next, let us recall the notion of \emph{local symmetry}. 
 \begin{definition}\cite[Definition 6.1]{Brock}\label{local_symmetry_def}
 Let $f:\mathbb{T}\times  \mathbb{R}\to \mathbb{R}$ be a compactly supported nonnegative function such that $f$ is continuously differentiable on $\left\{ (x,y): 0<f(x,y)<\sup f\right\}$. We say $f$ is \emph{locally symmetric in the $y$-direction} if the following property holds:
 
  Suppose $(x,y_1)\in\mathbb{T}\times \mathbb{R}$ satisfies
  \[
  0 < f(x,y_1)<\sup f,\quad \text{ and }\quad \frac{\partial f}{\partial y}(x,y_1)>0.
  \]
  Then for the unique point $(x,y_2)$ satisfying
  \[
  y_2>y_1,\quad f(x,y_1)=f(x,y_2)<f(x,y)\text{ for all $y\in (y_1,y_2)$},
  \]
 we have
  \begin{align}\label{local_prooperty_fapplied}
  \frac{\partial f}{\partial x}(x,y_1)=\frac{\partial f}{\partial x}(x,y_2),\quad \text{ and }\quad \frac{\partial f}{\partial y}(x,y_1)=-\frac{\partial f}{\partial y}(x,y_2).
  \end{align}
 \end{definition}

  \begin{theorem}\cite[Theorem 6.2]{Brock}\label{dirichelet_lemma}
  Let $f:\mathbb{T}\times  \mathbb{R}\to \mathbb{R}$ be a compactly supported nonnegative function such that $f$ is continuously differentiable on $\left\{ (x,y): 0<f(x,y)<\sup f\right\}$, and let $f^t$ be its continuous Steiner symmetrization. If
  \[
  \lim_{t\searrow 0} \frac{1}t\left( \int_{\mathbb{T}\times\mathbb{R}}|\nabla f|^2\d x\d y -\int_{\mathbb{T}\times\mathbb{R}}|\nabla f^t|^2\d x\d y\right)=0, 
  \]
  then $f$ is locally symmetric in the $y$-direction.
  \end{theorem}

 \subsection{Symmetrization of the  diffeomorphism}\label{Steiner_onesided_two_dim}
 Recall that $Y_\lambda$ in \eqref{def_Y_extended} is defined in $\mathbb{T}\times \mathbb{R}$ and $Y_\lambda(x,y)=-1$ for $|y|\ge 1$. 
  We define the continuous Steiner symmetrization of $Y_\lambda$ about $y=0$ as
 \begin{align}\label{Y_definition}
 Y_\lambda^t(x,y):=(Y_\lambda+1)^t(x,y)-1,
\end{align}
 where $(Y_\lambda+1)^t$ is defined by \eqref{def_conti_f}.

 Let us collect several properties of $Y^t_\lambda$. By construction, for every $(x,y)\in M$, 
  \begin{align}\label{Ysteiner_definition}
 Y_\lambda^t(x,y)=\int_{-1}^0 \ind_{\mathcal{S}_t(A(\eta))}(x,y)\d\eta-1,\text{ where }A(\eta):=\left\{(x,y)\in\mathbb{T}\times (-1,1): Y_\lambda(x,y)>\eta\right\} \text{ for $\eta\in (-1,0)$}.
  \end{align}
  Importantly, we have
  \begin{align}\label{ykeepsthezero}
   \left\{ (x,y)\in\mathbb{T}\times\mathbb{R}: Y_\lambda^t(x,y)=0\right\}=\mathbb{T}\times \left\{0\right\}\text{ for all $t\ge0$.}
 \end{align}
Now Lemma~\ref{regularity_lemma_steiner} reads
\begin{align}
\rVert \nabla Y_\lambda^t\rVert_{L^p}&\le \rVert \nabla Y_\lambda\rVert_{L^p},\label{Lipschitz_preservation}\\
\rVert Y_\lambda^t-Y_\lambda\rVert_{L^p}&\le 2t\rVert \partial_yY_\lambda\rVert_{L^p},\label{L_pconvergence_steiner}
\end{align}
for all $t\ge0$, $p\in [1,\infty]$.

Next, let us denote
\[
A(\eta)_x:=\left\{ y\in (-1,1): (x,y)\in A(\eta)\right\}\text{ and }A^t(\eta)_x:=\mathcal{S}_t(A(\eta)_x).
\]
Since $y\mapsto Y_\lambda(x,y)$ is continuous for every $x\in\mathbb{T}$, it holds that $A(\eta)_x$ is open in $\mathbb{R}$, thus so is $\mathcal{S}_t(A(\eta)_x)$.

From Lemma~\ref{yweight_variation}, we have
\begin{align}\label{variation_density_2}
\int_{-1}^1 |y| \left(\ind_{A(\eta)_x}-\ind_{A^t(\eta)_x}\right){\d y} \ge0,\text{ for all $\eta\in (-1,0)$, $x\in\mathbb{T}$ and $t\ge0$.}
\end{align}

From \eqref{onedversion_levelset1} and \eqref{onedversion_levelset2}, we deduce
  \begin{align}
\left\{ y\in (-1,1) : Y_\lambda^t(x,y)>\eta\right\} &=A^t(\eta)_x, \text{ up to a set of measure zero},\label{finer_level_set_oned1}
  \end{align}
  for all $(x,\eta)\in \mathbb{T}\times (-1,0)$. In particular, the Fubini-Tonelli theorem gives
 \[
 |\left\{ (x,y)\in M : Y_\lambda^t(x,y)>\eta\right\}| =|\mathcal{S}_t(A(\eta))|=|A(\eta)| = |\left\{ (x,y)\in M : Y_\lambda(x,y)>\eta\right\}|,
 \]
  for all $\eta\in (-1,0)$. Therefore $Y_\lambda^t$ and $Y_\lambda$ have the same distribution on $M$, and hence
   \begin{align}\label{equidistribution}
 \int_{M}\overline{G}(Y_\lambda^t(x,y))\d x\d y = \int_{M} \overline{G}(Y_\lambda(x,y))\d x\d y\text{ for all $\overline{G}\in C([-1,0])$ and  $t\ge 0$.}
 \end{align}
 Consequently, we obtain the following lemma.
\begin{lemma}\label{limit_equidistribution}
Let $\overline{G}\in W^{1,\infty}([-1,0])$.
Then
\begin{align}
\lim_{t\searrow0} \frac{1}{t}\int_M \left| {\overline{G}(Y^t_\lambda)-\overline{G}(Y_\lambda)-\overline{G}'(Y_\lambda)(Y^t_\lambda - Y_\lambda)}\right| \d x\d y&=0, \label{G_limit_1p}\\
\lim_{t\searrow0}\frac{1}{t}\int_M {\overline{G}'(Y_\lambda)(Y^t_\lambda-Y_\lambda)} \d x\d y &= 0. \label{G_limit_1p2}
\end{align}
\end{lemma}

\begin{proof}
  
We claim that 
\begin{align}\label{claim_dominated_conv_rigorous}
\lim_{t\searrow 0}\frac{\overline{G}(Y^t_\lambda)-\overline{G}(Y_\lambda)-\overline{G}'(Y_\lambda)(Y^t_\lambda - Y_\lambda)}{t}=0,\text{ a.e. on $M$.}
\end{align}

To see this, let us denote $\mathcal{N}$ the set of all points in $(-1,0)$ where $\overline{G}$ is not differentiable. Since $\overline{G}\in W^{1,\infty}$, the set $\mathcal{N}$ is negligible. Using  \eqref{nulset_estimate_YX}, we have
\[
|Y^{-1}(\mathcal{N})|
=0.
\]
The reflected and translated definition of $Y_\lambda$ in \eqref{def_Y_extended}
therefore gives
\[
|\mathcal{N}_\lambda|=0,\text{ where }\mathcal{N}_\lambda:=\{(x,y)\in M:Y_\lambda(x,y)\in \mathcal{N}\}.
\]
Thanks to \eqref{L_pconvergence_steiner} with $p=\infty$, we have
 \begin{align}\label{bounds_betternotation}
\left|\frac{Y^t_\lambda -Y_\lambda}{t}\right|\le 2\rVert \partial_yY_\lambda\rVert_{L^\infty},\quad  
\left| \frac{\overline{G}(Y^t_\lambda)-\overline{G}(Y_\lambda)-\overline{G}'(Y_\lambda)(Y^t_\lambda - Y_\lambda)}{t}\right|\le 4\rVert \overline{G}\rVert_{W^{1,\infty}}\rVert \partial_y Y_\lambda\rVert_{L^\infty},
\end{align}
 in $M\backslash \mathcal{N}_\lambda$. Fix $(x,y)\in M\backslash \mathcal{N}_\lambda$. Since $\overline{G}$ is differentiable in $\mathcal{N}_\lambda^c$, especially at $Y_\lambda(x,y)$, it holds that
\[
\overline{G}(Y_\lambda^t)-\overline{G}(Y_\lambda)
-\overline{G}'(Y_\lambda)(Y^t_\lambda -Y_\lambda)=o(\left|Y^t_\lambda -Y_\lambda\right|).
\]
Thanks to the first bound in \eqref{bounds_betternotation}, we have
\[
\frac{\overline{G}(Y_\lambda^t)-\overline{G}(Y_\lambda)
-\overline{G}'(Y_\lambda)(Y^t_\lambda -Y_\lambda)}{t}\to 0
\]
 at each fixed $(x,y)\in M\backslash \mathcal{N}_\lambda$. Since $\mathcal{N}_\lambda$ is negligible, we obtain \eqref{claim_dominated_conv_rigorous}. Combining this with the second bound in \eqref{bounds_betternotation} and applying the dominated convergence theorem, we conclude \eqref{G_limit_1p}. Furthermore, this, together with \eqref{equidistribution}, gives \eqref{G_limit_1p2}.
\end{proof} 
 \subsection{Variational properties of Dirichlet and potential energies}\label{energy_symmetrization_po}
Recall from \eqref{def_psi_rho_B} that the stream function $\phi_\lambda$ is given by
\begin{align}\label{psi_conjugated}
\phi_\lambda=\overline{\phi}(Y_\lambda)
\end{align}
where $\overline{\phi}\in W^{2,\infty}$ is nonnegative and monotone increasing (but not necessarily strictly increasing). As with $Y_\lambda$, we define its continuous Steiner symmetrization $\phi_\lambda^t$ by \eqref{def_conti_f}. We have
\begin{align}\label{Steiner_symmetrization_streamfunction}
\phi_\lambda^t(x,y)=\overline{\phi}(Y_\lambda^t(x,y)).
\end{align}
Indeed, since $\overline{\varphi}(\eta):=\overline{\phi}(\eta-1)$ is a nondecreasing continuous function such that $\overline{\varphi}(0)=0$, we have
\[
\phi_\lambda^t=(\overline{\phi}(Y_\lambda))^t= \left(\overline{\varphi}(Y_\lambda+1)\right)^t=\overline{\varphi}((Y_\lambda+1)^t)=\overline{\varphi}(Y_\lambda^t+1)=\overline{\phi}(Y_\lambda^t),
\]
where the third equality is due to Lemma~\ref{nondecreasing_conjugation} and the fourth equality is due to \eqref{Y_definition}. Importantly, we have
\begin{align}\label{psi_lambdat_online}
\phi_\lambda^t(x,0)=\overline{\phi}(Y_\lambda^t(x,0))=\overline{\phi}(0)=\overline{\phi}(Y_\lambda(x,0))=\phi_\lambda(x,0),
\end{align}
where the second equality follows from \eqref{ykeepsthezero}. Furthermore, we use Lemma~\ref{regularity_lemma_steiner} and observe that
\[
\rVert \nabla \phi_\lambda^t \rVert_{L^p}\le \rVert \nabla \phi_\lambda \rVert_{L^p}
\] for all $t\ge0$, $p\in [1,\infty]$. Especially, we have
\begin{align}\label{Lipschitz_preservation_stream}
\rVert \nabla\phi_\lambda^t\rVert_{L^2}^2\le \rVert \nabla \phi_\lambda\rVert_{L^2}^2.
\end{align}

 Next, we consider the density symmetrization. Recall from \eqref{def_psi_rho_B} that the density $\theta_\lambda$ is given by
 \[
 \theta_\lambda(x,y)=\overline{\theta}(Y_\lambda(x,y)),
 \]
 where $\overline{\theta}$ is nondecreasing. We define
 \begin{align}\label{density_stratification}
 \theta_\lambda^t(x,y):=\overline{\theta}(Y_\lambda^t(x,y)).
 \end{align}
 Again, Lemma~\ref{nondecreasing_conjugation} implies that $\theta^t_\lambda$ is the continuous Steiner symmetrization of $\theta_\lambda$.

\begin{proposition}\label{potential_energy_variation_1231}
For $t\ge 0$, it holds that
\begin{align}\label{potential_zero_gT}
\int_{M}|y|\theta_\lambda^t(x,y)\d x\d y - \int_{M}|y| \theta_\lambda(x,y)\d x\d y \le0,
\end{align}\end{proposition}
\begin{proof}
We compute
\[
\theta_\lambda^t(x,y)=\overline{\theta}(Y_\lambda^t(x,y))=\int_{-1}^{Y_\lambda^t(x,y)}\overline{\theta}'(\eta)\d\eta + \overline{\theta}(-1) =\int_{-1}^0 \ind_{\left\{\eta\in (-1,0):\ Y_\lambda^t(x,y)>\eta\right\}}\overline{\theta}'(\eta)\d\eta+\overline{\theta}(-1).
\] 
Therefore, by Fubini's theorem,
\[
\int_{M}|y|\theta_\lambda^t(x,y)\d x\d y=\overline{\theta}(-1)\int_{M} |y|\d x\d y +\int_{\mathbb{T}}\int_{-1}^0 \overline{\theta}'(\eta)\left(\int_{-1}^1 |y| \ind_{\left\{(x,y):\ Y_\lambda^t(x,y)>\eta\right\}} {\d y}\right) \d\eta {\d x}.
\]
It follows from \eqref{finer_level_set_oned1} that 
\[
\ind_{\left\{(x,y):\ Y_\lambda^t(x,y)>\eta\right\}} = \ind_{A^t(\eta)_x}(y) \text{ for almost every $y\in (-1,1)$.}
\]
 Therefore
\begin{align*}
\int_{M}|y|\theta_\lambda^t(x,y)\d x\d y - \int_{M}|y|\theta_\lambda(x,y)\d x\d y = \int_{\mathbb{T}}\int_{-1}^0 \overline{\theta}'(\eta)\int_{-1}^1 |y| \left(\ind_{A^t(\eta)_x}-\ind_{A(\eta)_x} \right){\d y} d\eta {\d x}\le 0,
\end{align*} 
where the last inequality follows from \eqref{variation_density_2} and Lemma~\ref{density_indeed_monotone}.
\end{proof}

    Let us define
  \[
  \mathcal{A}(t):=\frac{1}{2}\int_{M} |\nabla \phi_\lambda^t(x,y)|^2\d x\d y  +\int_{M}|y|\theta_\lambda^t(x,y)\d x\d y,\text{ for $t\ge0$.}
  \]
  \begin{proposition}\label{Steniner_variation_P}
 We have
  \[
  \liminf_{t\searrow0 }\frac{\mathcal{A}(t)-\mathcal{A}(0)}{t}\ge0.
  \]
  \end{proposition}
  \begin{proof}
  We first compute the difference of the Dirichlet energy. Observe that
  \[
  |\nabla \phi_\lambda^t|^2 - |\nabla \phi_\lambda|^2 = |\nabla \phi_\lambda^t -\nabla \phi_\lambda|^2 +2\nabla \phi_\lambda \cdot\nabla (\phi_\lambda^t -\phi_\lambda).
  \]
  Therefore
  \begin{align*}
  \frac{1}{2}\int_{M}|\nabla \phi_\lambda^t|^2 - |\nabla \phi_\lambda|^2 \d x\d y &= \frac{1}{2}\int_{M}|\nabla \phi_\lambda^t -\nabla \phi_\lambda|^2 dxdy +\int_{M}\nabla \phi_\lambda \cdot \nabla(\phi_\lambda^t -\phi_\lambda)\d x\d y\nonumber\\
  &\ge \int_{M}\nabla \phi_\lambda \cdot \nabla(\phi_\lambda^t -\phi_\lambda)\d x\d y\\
  &= -\int_{M}\Delta \phi_\lambda (\phi_\lambda^t -\phi_\lambda)\d x\d y.
  \end{align*}
  The integration by parts in the last equality is justified because $\phi^t_\lambda=\phi_\lambda$ on the boundary $\partial M$ and also on $\mathbb{T}\times \left\{0\right\}$ (see \eqref{psi_lambdat_online}), for all $t\ge0$. It follows from \eqref{Steiner_symmetrization_streamfunction} that $ \phi_\lambda^t-\phi_\lambda = \overline{\phi}(Y_\lambda^t)-\overline{\phi}(Y_\lambda)$, and thus
  \[
  | \phi_\lambda^t-\phi_\lambda - \overline{\phi}'(Y_\lambda)(Y_\lambda^t-Y_\lambda)|\le \rVert \overline{\phi}\rVert_{W^{2,\infty}}\rVert Y_\lambda^t-Y_\lambda\rVert_{L^\infty}^2=o(t),
  \]
  where the last equality follows from \eqref{L_pconvergence_steiner}. Hence, we obtain
  \begin{align}\label{dirichlet_firstpart}
   \frac{1}{2}\int_{M}|\nabla \phi_\lambda^t|^2 - |\nabla \phi_\lambda|^2 \d x\d y  \ge -\int_{M} \Delta \phi_\lambda \overline{\phi}'(Y_\lambda)(Y_\lambda^t-Y_\lambda)\d x\d y+ o(t).
  \end{align}
For the potential energy, it follows from \eqref{density_stratification} that
  \[
\theta_\lambda^t-\theta_\lambda= \overline{\theta}(Y_\lambda^t)-\overline{\theta}(Y_\lambda).
  \]
 Therefore 
  \begin{align*}
  \int_{M}|y|(\theta_\lambda^t-\theta_\lambda)\d x\d y &= \int_{M} |y|\overline{\theta}'(Y_\lambda)(Y_\lambda^t-Y_\lambda)\d x\d y + \int_{M}|y| \left(\overline{\theta}(Y^t_\lambda) - \overline{\theta}(Y_\lambda) -\overline{\theta}'(Y_\lambda)(Y_\lambda^t-Y_\lambda) \right) {\d x} {\d y} \\
  &= \int_{M} |y|\overline{\theta}'(Y_\lambda)(Y_\lambda^t-Y_\lambda)\d x\d y  + o(t),
  \end{align*}
  where the last equality follows from Lemma~\ref{limit_equidistribution}.
  Combining this with \eqref{dirichlet_firstpart}, we obtain
  \begin{align}\label{t_linearterms}
  \mathcal{A}(t)-\mathcal{A}(0)&\ge -\int_{M}(\Delta\phi_\lambda\overline{\phi}'(Y_\lambda) -|y|\overline{\theta}'(Y_\lambda))(Y_\lambda^t-Y_\lambda)\d x\d y +o(t)\\
  & = -\int_{M} \overline{Q}'(Y_\lambda)(Y_\lambda^t-Y_\lambda)\d x\d y + o(t),
  \end{align}
  where the last equality follows from \eqref{B_aftersliding} and \eqref{blambda_def}. Lemma~\ref{B_condition} ensures that $\overline{Q}$ is $W^{1,\infty}$, therefore we can apply Lemma~\ref{limit_equidistribution} and obtain
\[
\frac{\mathcal{A}(t)-\mathcal{A}(0)}{t}\ge o(1).
\]
This completes the proof. 
  \end{proof}
  
  \begin{corollary}\label{each_energy_vanishes}
We have
   \begin{align*}
  \lim_{t\searrow 0}\frac{1}{t} \left( \int_{M}|\nabla \phi_\lambda|^2\d x\d y -\int_{M}|\nabla \phi_\lambda^t|^2\d x\d y\right)&=0.
  \end{align*}
  Consequently, $\phi_\lambda$ is locally symmetric in the $y$-direction, in the sense of Definition~\ref{local_symmetry_def}.
   \end{corollary}
   Note that the qualitative property that $\phi_\lambda$ is locally symmetric holds for any $\lambda$ (see Remark~\ref{lambda_unchanged}).
     \begin{proof}
  Proposition~\ref{potential_energy_variation_1231} gives
  \[
  \limsup_{t\searrow 0}\frac{1}{t}\left(\int_{M}|y|\theta_\lambda^t(x,y)\d x\d y - \int_{M}|y| \theta_\lambda(x,y)\d x\d y \right)\le 0.
  \] Thus Proposition~\ref{Steniner_variation_P} gives
  \begin{align}\label{inequalitychains}
  0\le \liminf_{t\searrow0 }\frac{\mathcal{A}(t)-\mathcal{A}(0)}{t}&\le   \liminf_{t\searrow 0 }\frac{1}{2t} \left( \int_{M}|\nabla \phi^t_\lambda|^2\d x\d y -\int_{M}|\nabla \phi_\lambda|^2\d x\d y\right) \le 0,
  \end{align}
  where the last inequality follows from \eqref{Lipschitz_preservation_stream}. Therefore, the $\liminf_{t\searrow 0}$ can be replaced by $\lim_{t\searrow 0}$ and all the inequalities are equalities. Theorem~\ref{dirichelet_lemma} implies that $\phi_\lambda$ is locally symmetric.
    \end{proof}

\subsection{Proof of Theorem~\ref{Rigidity_modified_solutions}}\label{prooftheoremRigidity_modified_solutions}\begin{proof}
Fix $\eta\in(-1,0)\cap \mathcal{R}$. By \eqref{stagnant_sets_rigidity_1} and \eqref{def_monotone_streamfunction1}, we have $\overline{\phi}'(\eta)>0$. Define
\[
w(x):=\max\{y\in (-1,0):Y(x,y)=\eta\}.
\]
The maximum exists since $\Gamma_\eta=\left\{(x,y)\in M_-: Y(x,y)=\eta\right\}$ is compact,
and
\[
Y(x,y)>\eta\qquad\text{for }w(x)<y\le0.
\]
Choose $(x_0,y_0)\in\Gamma_\eta$ at which the vertical
coordinate attains its maximum over $\Gamma_\eta$. Then  
\[
w(x_0)=y_0,\qquad \partial_xY(x_0,y_0)=0,
\qquad
\partial_yY(x_0,y_0)>0.
\]
Consequently, for every $\lambda\in\mathbb R$,
\[
\partial_x\phi_\lambda(x_0,y_0)=0,
\qquad
\partial_y\phi_\lambda(x_0,y_0)>0.
\]
By our construction of $Y_\lambda$ and $\phi_\lambda$ in \eqref{def_Y_extended} and \eqref{def_psi_rho_B}, and $\overline{\phi}'(\eta)>0$,  we have
\[
\phi_\lambda(x_0,-w(x_0-\lambda))=\overline{\phi}(\eta),\qquad \phi_\lambda(x_0,y)>\overline{\phi}(\eta)\text{ for $y\in (y_0,-w(x_0-\lambda))$.}
\]
Therefore, by local symmetry for $\phi_\lambda$, proved in Corollary~\ref{each_energy_vanishes}, we have
\[
\partial_x\phi_\lambda(x_0,-w(x_0-\lambda))=0,
\qquad
\partial_y\phi_\lambda(x_0,-w(x_0-\lambda))<0.
\]
Using again the definitions of $Y_\lambda$ and $\phi_\lambda$, and $\overline{\phi}'(\eta)>0$, we obtain
\[
\partial_xY(x_0-\lambda,w(x_0-\lambda))=0,
\qquad
\partial_yY(x_0-\lambda,w(x_0-\lambda))>0.
\]
Since $\lambda$ is arbitrary, these identities yield
\[
\partial_xY(x,w(x))=0,
\qquad
\partial_yY(x,w(x))>0,
\]
for all $x\in\mathbb{T}$. Hence the implicit function theorem and compactness of
$\Gamma_\eta$  give $w\in C^1(\mathbb T)$ and
\[
w'(x)
=-\frac{\partial_xY(x,w(x))}
        {\partial_yY(x,w(x))}=0.
\]
Thus $w$ is constant and consequently, $\Gamma_\eta$ is a straight horizontal line. 
Hence $
\partial_x\phi=\overline{\phi}'(Y)\partial_xY=0$ on $\Gamma_\eta$. If $\eta\notin \mathcal{R}$, again \eqref{def_monotone_streamfunction1} gives $\partial_x\phi=0$ on $\Gamma_\eta$. Therefore, $\phi$ is independent of the variable $x$. Lastly, since $\bv=\nabla^\perp \phi$ and $\bv$ is shear, $\bv\cdot\nabla \bv =0$. Hence, the momentum equation in \eqref{Boussinesq_P} yields $\theta e_2$ is a gradient, thus
  \[
  \partial_x\theta = 0
  \]
in the distributional sense. This finishes the proof of Theorem~\ref{Rigidity_modified_solutions}.
\end{proof}

\section{Rigidity for non-resonant Richardson numbers}\label{sec:rigidity_richardson}
We investigate the effects of the Richardson number on the rigidity properties of stationary solutions \eqref{eq:BoussinesqVel}. The main goal of the section is to prove the following result.
\begin{theorem}\label{thm:rigidity_noncrit_J}
    Assume $(\u:=(u,v),\rho)$ is a smooth steady solution of the Boussinesq equations such that {$\lbrace u = 0 \rbrace\subset M$ is nonempty and }for each connected component of  $\Gamma_u = \lbrace u = 0 \rbrace$ there holds $|\partial_y u| > 0$ and $\partial_y\rho \not\equiv0$. Assume further that for $\mathcal{J}\coloneqq  \frac{-\partial_y \rho}{|\partial_y u|^2}$ there holds $\mathcal{J}\neq -m(m-1)$ for all $\N\ni m\geq 2$ along all connected components of $\Gamma_u$.  
    Then $v\equiv 0$ and $u=u(y)$ and $\rho=\rho(y)$.
\end{theorem}

\begin{remark}\label{remark:rigidity_smooth_carleman}
    The smoothness assumptions on $u,v,\rho$ can be relaxed to $u,v,\rho\in C^2(M)$,  and the conclusions of Theorem \ref{thm:rigidity_noncrit_J} are still true if one further assumes that $\mathcal{J}\in (-\frac34, \frac14) \setminus \lbrace 0 \rbrace$ on all connected components of $\Gamma_u$. We shall give more details later during the section. 
\end{remark}

Since $\partial_y u \neq 0$ whenever $u=0$ the implicit function theorem shows that each connected component of $\Gamma_u$ is given by a graph, namely
\begin{align}
    \Gamma_u = \bigcup_{i=1}^n \lbrace (x, h_i(x)) : x \in \T \rbrace
\end{align}
for smooth functions $h_i:\T\rightarrow \R$. As we shall see from the arguments, we may assume without loss of generality that $\Gamma_u$ is connected and given by $\lbrace (x, h(x)) : x \in \T \rbrace$ for some $h:\T\rightarrow \R$. We define the functions
\begin{align}
    a(x) \coloneqq   \partial_y u(x,h(x)), \quad b(x) \coloneqq  \partial_y \rho(x,h(x))
\end{align}
and we set 
\begin{align}
    \mathcal{B} = \lbrace x\in \T : b(x) \neq 0 \rbrace.
\end{align}

\begin{lemma}\label{lemma:flat_stagnation_curve}
    There holds $\mathcal{B} = \T$ and $h(x)=h_0$, for some $h_0\in (-1,1)$, for all $x\in \T$.
\end{lemma}

\begin{proof}
    Let $x\in \mathcal{B}$, then $b(x)\neq 0$, which is an open property and $b(x)$ is continuous. The continuity equation $u\partial_x\rho + v \partial_y \rho = 0$ along $\Gamma_u$ then yields $v(x,h(x))=0$ since $u(x,h(x))=0$. Therefore, $|(u,v)|=0$ at $(x,h(x))$ and the steady Boussinesq equation $(u,v)\cdot\nabla \omega = -\partial_x \rho$ immediately gives $\partial_x \rho(x,h(x)) = 0$ on $\Gamma_u$, whenever $x\in \mathcal{B}$. Taking a $\partial_x$ derivative in $\u\cdot\nabla\rho=0$, we observe that
    \begin{align*}
        0 = \partial_x u \partial_x \rho + \partial_x v \partial_y \rho
    \end{align*}
    at $(x,h(x))$, so that $\partial_x v(x,h(x))=0$. Similarly, differentiating once with respect to $\partial_y$, we obtain 
    \[
    0=\partial_y\u\cdot\nabla \rho = \partial_yu\partial_x\rho+\partial_yv\partial_y\rho=\partial_y v\partial_y\rho \qquad \text{on $\Gamma_u$.}
    \]
    Hence, $\partial_yv=0$ on $\Gamma_u$ and the incompressibility yields $\partial_x u =0$ on $\Gamma_u$.  Thus, $h'(x) = - \frac{\partial_x u(x,h(x))}{\partial_y u(x,h(x))}=0$. Hence, all connected components of $\mathcal{B}$ are given by open intervals $I_j$ such that $h(x) = h_j$ for all $x\in I_j$. {We next show that there is only one such $I_j=\T$}.

    Indeed, from the equation $(u,v)\cdot\nabla \rho =0$, taking two $\partial_y$ derivative and evaluating at $x\in \mathcal{B}$ we have
    \begin{align*}
        2a(x) b'(x) - b(x)a'(x) = 0.
    \end{align*}
  Similarly, differentiating  (see \eqref{vorticity_eq_rigidt2})
  \[
u^2\Delta v-u(\Delta u)v-\partial_y\rho\,v=0
\] twice with respect to $y$ and evaluating at $(x,h(x))$, $x\in\mathcal{B}$,
we therefore obtain
\[
2(\partial_y u)^2\Delta v-\partial_y\rho\partial_y^2v=0 \text{ at $(x,h(x))$, $x\in\mathcal{B}$.}
\]
Since $v(x,h(x))\equiv0$ and $h(x)$ is constant for $x\in\mathcal{B}$, we have
$\partial_x^2v=0$, and hence
$\Delta v(x,h(x))=\partial_y^2v(x,h(x))$ for $x\in\mathcal{B}$.
Moreover, incompressibility gives
\[
\partial_y^2v(x,h(x))
=-\partial_x\partial_yu(x,h(x))
=-a'(x).
\]\color{black}
Consequently,
\[
\bigl(2a(x)^2-b(x)\bigr)a'(x)=0
\qquad\text{for all }x\in\mathcal B.
\]
    Assume next that $a'(x)\neq 0$ for some $x\in \cB$. By continuity it is non-zero on a small neighborhood of it. Then, $b(x) = 2a(x)^2$ on such neighborhood and $b'(x)=4a(x) a'(x)$, so that
    \begin{align}
        0 = 2a(x) b'(x) - b(x) a'(x) = 8 a^2(x)a'(x) - 2a^2(x)a'(x) = 6 a^2(x)a'(x)
    \end{align}
    from which we deduce that $a'(x)=0$. Hence, $a'(x)=0$ and thus $b'(x)=0$ for all $x\in \mathcal{B}$. This means that $b(x)$ is a non-zero (perhaps distinct) constant on each connected component of $\mathcal{B}$. If $\mathcal{B}\neq \T$, then $b(x)=0$ on $\mathcal{B}^c$ and in particular on $\partial \mathcal{B}$. Since $b:\T\rightarrow\R$ is continuous, we reach a contradiction, so that $\mathcal{B}=\T$ and $h(x)=h_0$, for some $h_0\in (-1,1)$, for all $x\in \T$.    
\end{proof}

\begin{remark}[Continuation of Remark \ref{remark:rigidity_smooth_carleman}]
    Once we know that all possible stagnation lines are horizontal, that $u,v,\rho\in C^2(M)$ and that $v= \nabla v = 0$ there, we can appeal to the local unique continuation Theorem 3.11 and Theorem 4.1 of \cite{shao2024approximate}. 
\end{remark}

Without loss, we assume that $h_0=0$ and we denote $a_0=\partial_y u(x,0)$ and $b_0=\partial_y \rho(x,0)$, with $a_0>0$ and $b_0\neq 0$. We next show that $v(x,y)$ vanishes at infinite order at $y=0$.

\begin{lemma}
    For all $j\geq 0$ and $n\geq 0$ there holds $\partial_x^j \partial_y^n v(x,0)=0$ for all $x\in \T$.
\end{lemma}

\begin{proof}
    It is sufficient to show $\partial_y^m v(x,0)=0$ for all $x\in \T$ and all $m\geq 0$. We have seen in the proof of the lemma above that it holds for $m=0,1$ so we assume that $m\geq 2$ and we proceed by induction, so that $\partial_y^{m'}v(x,0)=0$ for all $0\leq m' < m$. For $v_m(x) = \partial_y^m v(x,0)$ we have by Taylor's formula and smoothness of $v(x,y)$ that $v(x,y) = \frac{v_m(x)}{m!}y^m + O(y^{m+1})$. Moreover, $u(x,y) = a_0 y + O(y^2)$ while $\partial_y \rho(x,y) = b_0 + O(y)$. Hence, there holds
    \begin{align}
        u^2 \Delta v &= a_0^2\frac{v_m(x)}{(m-2)!}y^m + O(y^{m+1}), \\
        u(\Delta u) v &= O(y^{m+1}), \\
        \partial_y \rho v &= b_0 \frac{v_m(x)}{m!}y^m + O(y^{m+1}).
    \end{align}
    Here, we say that $f(x,y)\in O(y^n)$ if there exists $C>0$ such that $\lim_{y\rightarrow 0}\frac{|f(x,y)|}{|y|^n}\leq C$ for all $x\in \T$. Therefore,
    \begin{align}
        0 = u^2 \Delta v - u(\Delta u) v - \partial_y \rho v = \frac{1}{m!}\left( a_0^2 m(m-1) - b_0 \right) v_m(x) y^m + O(y^{m+1})
    \end{align}
    from which we deduce that $v_m(x)\equiv 0$ since $\mathcal J(x,0) = \frac{-b_0}{a_0^2}\neq -m(m-1)$ for all $m\geq 2$.
\end{proof}

Since $v(x,y)$ vanishes to infinite order at $y=0$, we aim to show that $v\equiv 0$ in a neighborhood of $\T\times \lbrace 0 \rbrace$ by means of a Carleman estimate for an effective equation on $v(x,y)$. To that purpose, since $u(x,0)=0$ and $\partial_y u(x,0)=a_0\neq 0$, we have that
\begin{align}
    u(x,y) = y \int_0^1 \partial_y u(x,ty) \d t = y\alpha(x,y),
\end{align}
for $\alpha:\T\times(-1,1)$ smooth and such that $\alpha(x,0)=a_0\neq 0$ and $\alpha(x,y) \geq \frac{a_0}{2}>0$ for all $(x,y)\in \T \times (-\delta,\delta)$, for some $\delta>0$. Therefore, on $\T\times (-\delta, \delta)$ there holds
\begin{align}\label{eq:reduced_B_equation}
    y^2 \Delta v + V(x,y) v =0, \quad V(x,y) \coloneqq  -y\frac{\Delta u}{\alpha} - \frac{\partial_y \rho}{\alpha^2}
\end{align}
with also $V\in C^\infty(\T\times (-\delta,\delta))$.

Before proceeding further, we first show a Carleman-type estimate that will be useful later.

\begin{lemma}\label{lemma:carleman_estimate}
Let $T_0>0$ and $L\coloneqq \partial_t^2 + e^{-2t}\partial_x^2 + W(x,t)$ on $(x,t)\in \T\times (T_0,\infty)$, for some $W:\T\times (T_0,\infty)\rightarrow \R$ such that $\Vert W \Vert_{L^\infty} + \Vert \partial_t W \Vert_{L^\infty}\leq M$, for some $M>0$. Then, there exist $C>0$ and $\tau_0>0$ such that for all $\tau\geq \tau_0$ and all $g\in C^\infty_0(\T\times (T_0,\infty))$ there holds
\begin{align}
    \tau^3 \Vert G \Vert_{L^2}^2 + \tau^2 \Vert \partial_t G \Vert_{L^2}^2 + \tau \Vert e^{-t}\partial_x G \Vert_{L^2}^2 \leq C \Vert e^{\tau t}L g\Vert_{L^2}^2
\end{align}
for $G\coloneqq  e^{\tau t}g$.
\end{lemma}

\begin{proof}
    We begin by writing 
    \begin{align}
        e^{\tau t} L e^{-\tau t} = S_\tau + A_\tau,
    \end{align}
    where $S_\tau \coloneqq   \partial_t^2 + e^{-2t}\partial_x^2 + \tau^2 + W(x,t)$ is symmetric and $A_\tau\coloneqq  -2\tau \partial_t$ is skew-symmetric. We next define $E\coloneqq  \Vert (S_\tau + A_\tau) G\Vert_{L^2}^2$. It is such that $E=\Vert e^{\tau t} Lg \Vert_{L^2} ^2$ and, moreover,
    \begin{align}
        E = \Vert S_\tau G \Vert_{L^2}^2 + 2\langle S_\tau G, A_\tau G \rangle + \Vert A_\tau G \Vert_{L^2}^2. 
    \end{align}
    Furthermore,
    \begin{align}
        2\langle S_\tau G, A_\tau G \rangle = {\langle S_\tau G, A_\tau G \rangle}+ \langle S_\tau G, A_\tau G \rangle = \langle A_\tau G, S_\tau G \rangle + \langle S_\tau G, A_\tau G \rangle = \langle [S_\tau,A_\tau]G, G \rangle
    \end{align}
    and 
    \begin{align}
        [S_\tau,A_\tau] &= -2\tau \left( \partial_x^2 [e^{-2t},\partial_t] + [W, \partial_t] \right) \\
        &= -4\tau e^{-2t}\partial_x^2 +2 \tau \partial_t W(x,t). 
    \end{align}
    Therefore,
    \begin{align}
        E = \Vert S_\tau G \Vert_{L^2}^2 + 4\tau^2 \Vert \partial_t G \Vert_{L^2}^2 + 4\tau \Vert e^{-t} \partial_x G \Vert_{L^2}^2  + 2\tau \int_{\T\times(T_0,\infty)} \partial_t W |G|^2 \d t \d x
    \end{align}
    and thus
    \begin{align}
        \Vert S_\tau G \Vert_{L^2}^2 + 4\tau^2 \Vert \partial_t G \Vert_{L^2}^2 + 4\tau \Vert e^{-t} \partial_x G \Vert_{L^2}^2 \leq E + 2\tau M \Vert G \Vert_{L^2}^2.
    \end{align}
    On the other hand, 
    \begin{align}
        \langle S_\tau G, G \rangle = - \Vert \partial_t G \Vert_{L^2}^2 - \Vert e^{-t}\partial_x G\Vert_{L^2}^2 + \tau^2 \Vert G \Vert_{L^2}^2 + \int_{\T\times(T_0,\infty)} W |G|^2 \d x \d t
    \end{align}
    from which we deduce that, for $\tau_0>0$ large enough,
    \begin{align}
        \frac{\tau^2}{2}\Vert G \Vert_{L^2}^2 &\leq \langle S_\tau G, G \rangle + \Vert \partial_t G \Vert_{L^2}^2 + \Vert e^{-t} \partial_x G \Vert_{L^2}^2 \\
        &\leq \frac{\tau^2}{4}\Vert G \Vert_{L^2}^2 + 4\tau^{-2}\Vert S_\tau G \Vert_{L^2}^2 + \Vert \partial_t G \Vert_{L^2}^2 + \Vert e^{-t} \partial_x G \Vert_{L^2}^2.
    \end{align}
    Then, 
    \begin{align}
        \frac{\tau^2}{4}\Vert G \Vert_{L^2}^2 \leq \left( 4\tau^{-2} + \frac{\tau^{-2}}{4} + \frac{\tau^{-1}}{4}\right) (E + 2\tau M \Vert G \Vert_{L^2}^2)\leq C( \tau^{-1} E + \Vert G \Vert_{L^2}^2)  
    \end{align}
    which yields $\tau^3\Vert G \Vert_{L^2}^2 \leq C' \Vert e^{\tau t}Lg \Vert_{L^2}^2$ for $\tau_0$ large enough and some $C'>C$. As a result,
    \begin{align}
        \tau^2 \Vert \partial_t G \Vert_{L^2}^2 + \tau \Vert e^{-t}\partial_x G \Vert_{L^2}^2 \leq E + 2\tau M \Vert G \Vert_{L^2}^2 \leq C''E
    \end{align}
    for some $C''>C'$. With this, the proof is concluded.
\end{proof}

The next lemma shows that any smooth solution $v$ to \eqref{eq:reduced_B_equation} vanishing infinitely fast at $y=0$ must be identically 0 in a neighborhood of $\lbrace y = 0 \rbrace$.

\begin{lemma}
    Let $\delta>0$ and $V\in C^1(\T\times[0,\delta])$. Let $v\in C^\infty(\T\times[0,\delta])$ satisfy
    \begin{align}
        y^2 \Delta v + V(x,y) v = 0, \quad (x,y)\in \T\times (0,\delta)
    \end{align}
    and assume that $\partial_x^j\partial_y^n v(x,0)=0$ for all $j,n\geq 0$ and all $x\in \T$. Then, there exists $0< \delta_1 \leq \delta$ such that $v\equiv 0$ on $\T\times(0,\delta_1)$.
\end{lemma}

\begin{proof}
    We introduce the change of variables $y=e^{-t}$, equivalently $t=-\log(y)$ and we define $f(x,t)\coloneqq  v(x,e^{-t})$ for $(x,t)\in \T\times (T_0,\infty)$, with $T_0 = -\log(\delta)$. Then,
    \begin{align}
        \partial_t f = -e^{-t}\partial_y v, \quad \partial_t^2 f = -\partial_t f + y^2 \partial_y^2 v
    \end{align}
    and thus there holds
    \begin{align}
        \partial_t^2 f + \partial_t f + e^{-2t}\partial_x^2 f + \tilde{V}(x,t) f = 0, \quad \tilde V(x,t) = V(x,e^{-t}).
    \end{align}
    To remove the $\partial_t f$ term we introduce $w(x,t)\coloneqq  e^{t/2}f(x,t)$ so that 
    \begin{align}
        \left(\partial_t^2 + \partial_t \right) f = e^{-t/2} \left( \partial_t^2 - \frac14 \right)w
    \end{align}
    from which we obtain
    \begin{align}
        \partial_t^2 w + e^{-2t}\partial_x^2 w + W(x,t)w = 0, \quad W(x,t) \coloneqq  \tilde V(x,t) - \frac14
    \end{align}
    with further $\Vert W \Vert_{L^\infty} + \Vert \partial_t W \Vert_{L^\infty}\leq M$ for some $M>0$. Let $T, R>0$ such that $T>T_0$ and $R> T+2$, we introduce the smooth cut-off functions $\chi_T$ and $\eta_R$. They are such that $0\leq \chi_T, \eta_R\leq 1$ and
    \begin{align}
        \chi_T(t) = \begin{cases}
            0, & t \leq T, \\
            1, & t \geq T+1
        \end{cases}, \quad \eta_R(t) = \begin{cases}
            1, & t\leq R, \\
            0, & t\geq R+1
        \end{cases}.
    \end{align}
    We note that the supports of $\chi_T'$ and of $\eta_R'$ are disjoint, $\eta_R\equiv 1$ on the support of $\chi_T'$ and $\chi_T\equiv 1$ on the support of $\eta_R'$. We assume without loss of generality that $\Vert \chi_T \Vert_{C^2}+ \Vert \eta_R \Vert_{C^2}\leq 10$ and we define 
    \begin{align}
        g(x,t)\coloneqq \chi_T(t)\eta_R(t)w(x,t)\in C^\infty_0(\T\times (T_0,\infty)).
    \end{align}
    In view of Lemma \ref{lemma:carleman_estimate}, we have
    \begin{align}
        \tau^3 \Vert e^{\tau t}g \Vert_{L^2}^2 \leq C \Vert e^{\tau t}L g \Vert_{L^2}^2.
    \end{align}
    By the support properties of the cut-off derivatives and the identity $Lw = 0$ we have
    \begin{align}
        L g = [L, \chi_T]\eta_R w + \chi_T [L,\eta_R]w.
    \end{align}
    Next, since $|w(x,t)| + |\partial_x w(x,t)| + |\partial_t w(x,t)| \leq C_N e^{-N t}$ for all $N\geq 1$ due to the infinite vanishing of $v(x,y)$ at $y=0$, for $N>\tau+1$ we have that
    \begin{align}
        \Vert e^{\tau t}\chi_T[L, \eta_R]w \Vert_{L^2}^2 \leq C_N\int_R^{R+1} e^{2\tau t} e^{-2N t} \d t  \lesssim e^{-2R}\rightarrow 0
    \end{align}
    as $R\rightarrow \infty$. On the other hand,
    \begin{align}
        \Vert e^{\tau t}[L,\chi_T]\eta_R w \Vert_{L^2}^2 \leq C \int_\T\int_{T}^{T+1} e^{2\tau t}\left(|w(x,t)|^2 + |\partial_t w(x,t)|^2 \right) \d t \d x \leq C e^{2\tau(T+1)}M_T,
    \end{align}
    with $M_T\coloneqq  \int_\T\int_{T}^{\infty}\left(|w(x,t)|^2 + |\partial_t w(x,t)|^2 \right) \d t \d x< \infty$ due to the exponential decay of $|w(x,t)|$ and $|\partial_t w(x,t)|$. Finally,
    \begin{align}
        \tau^3 e^{2\tau(T+1)} \int_\T\int_{T+1}^{R} |w(x,t)|^2 \d t \d x \leq \tau^3 \Vert e^{\tau t}\chi_T\eta_R w \Vert_{L^2}^2  \leq C e^{2\tau(T+1)} M_T + C e^{-2R}. 
    \end{align}
    Since $|w|^2\in L^1(\T\times (T+1,\infty))$, the dominated convergence theorem yields
    \begin{align}
        \tau^3 \int_\T\int_{T+1}^{\infty}|w(x,t)|^2 \d t \d x \leq CM_T
    \end{align}
    for all $\tau\geq \tau_0$, with $CM_T$ independent of $\tau$. Hence, for $\tau$ sufficiently large we obtain a contradiction unless $w\equiv 0$ on $\T\times (T+1,\infty)$. From here, we deduce that $v(x,y)\equiv 0$ on $\T\times (0,\delta_1)$, for $\delta_1 = e^{-(T+1)}$ and the lemma is proved.
\end{proof}

We have all the ingredients to prove Theorem \ref{thm:rigidity_noncrit_J}.

\begin{proof}[Proof of Theorem \ref{thm:rigidity_noncrit_J}]
    We know from Lemma \ref{lemma:flat_stagnation_curve} that $\Gamma_u = \bigcup_{i=1}^n \lbrace (x,h_i) : x\in \T \rbrace$ for some $-1 < h_1 < ... < h_n < 1$. Moreover, $v(x,y)\equiv 0$ on $\bigcup_{i=1}^n \T \times (h_i-\delta_i, h_i + \delta_i)$ for sufficiently small $\delta_i>0$, for $i=1,...,n$. Finally, 
    \begin{align}
        \Delta v - \frac{\Delta u}{u}v - \frac{\partial_y \rho}{u^2}v = 0
    \end{align}
    is a smooth second-order linear elliptic equation on $v(x,y)$ in $\Omega\coloneqq \T\times(-1,1)\setminus \bigcup_{i=1}^n \T \times (h_i-\delta_i/2, h_i + \delta_i/2)$ and $v(x,y)\equiv 0$ in non-empty compact subsets of $\Omega$. By unique continuation, $v(x,y)\equiv 0$ in $\T\times (-1,1)$. By incompressibility, we have that $u=u(y)$ and since $(u,v)\cdot\nabla\omega=0$ we conclude that $\rho=\rho(y)$.
\end{proof}

We finish the section proving Theorem \ref{informal_rigidity_3}.

\begin{theorem}\label{thm:rigidity_noncrit_shear_dens}
    Let $u_0,\rho_0\in C^2([-1,1])$. Assume that $u_0(y_i)=0$ for finitely many $\lbrace y_i \rbrace_{i=1}^n\subset (-1,1)$, with $u'_0(y_i)\neq 0$ and $u_0(\pm 1)\neq 0$, and assume that for
    \begin{align}
        \mathcal{J}_0(y_i)\coloneqq  -\frac{\rho_0'(y_i)}{(u_0'(y_i))^2}
    \end{align}
    there holds $\mathcal{J}_0(y_i)\neq -m(m-1)$ for all $m\geq 1$. Then, there exists $\ep_0>0$ such that all smooth steady Boussinesq solutions $((u,v),\rho)$ with
    \begin{align}
        \Vert u - u_0 \Vert_{C^2} + \Vert v \Vert_{C^2} + \Vert \rho - \rho_0\Vert_{C^2} < \ep_0
    \end{align}
    must have $v(x,y)=0$, $u=u(y)$ and $\rho=\rho(y)$.
\end{theorem}

\begin{proof}
    Local monotonicity of $u_0$ near its zeros and $\Vert u - u_0 \Vert_{C^2}\leq \ep_0$ small enough assert the existence of $n$ smooth graphs $\lbrace (x,h_i(x))\rbrace_{i=1}^n$ over $x\in \T$ where $u$ vanishes and $\partial_y u$ is non-zero. Moreover, $\sup_{i=1,...,n}\Vert h_i(x) - y_i\Vert_{L^\infty(\T)}\lesssim \ep_0$. Due to $\Vert \rho - \rho_0\Vert_{C^2}\leq \ep_0$ we further have that $|\mathcal{J}(x,h_i(x))-\mathcal{J}_0(y_i)|\lesssim \ep_0$, for all $x\in \T$ and $1\leq i \leq n$, for  $\mathcal{J}=-\frac{\partial_y\rho}{(\partial_y u)^2}$. For $\ep_0>0$ small enough, $((u,v), \rho)$ satisfies the conditions of Theorem \ref{thm:rigidity_noncrit_J}. It is thus a stratified shear flow.
\end{proof}

\section{Bifurcation from non-zero velocity}\label{sec: bifurcation non-zero}

We next show the existence of non-trivial steady states arbitrarily close to a background shear stationary solution with non-zero velocity. 

\begin{theorem}\label{thm:nonzeroflexibility}
Let $q\in[1,\infty)$ and let $(\u_0(y),\rho_0(y))\in C^4(\overline{M})$ be such that $\u_0(y)= (u_0(y),0)$ with $u_0(y)\neq 0$ for all $y\in [-1,1]$. Then, for any $\varepsilon>0$ there exist \underline{laminar} $(\u, \rho)$ stationary solutions to the Boussinesq equations such that $\partial_x(\u, \rho) \not\equiv 0$ and
\begin{align}
    \Vert \u - \u_0\Vert_{L^q} + \Vert \rho - \rho_0 \Vert_{W^{1,q}} < \ep.
\end{align}
\end{theorem}


\begin{remark}
    An analogous statement holds for $H^s$ velocity proximity, for $s<\frac12$. For any $q\in [1,\infty)$ and $\ep>0$, there exist laminar non-shear stationary solutions $(\u,\rho)$ such that $\Vert \u - \u_0\Vert_{H^s} + \Vert \rho - \rho_0\Vert_{W^{1,q}}<\ep$. See \eqref{eq:pos_density_profiles} and Lemma \ref{lemma:perturbation close to original no stag} for closeness in $H^s$ and see Theorem \ref{thm rigidity non-zero Sobolev} for sharpness on the Sobolev regularity threshold.
\end{remark}

We will construct the solutions claimed in Theorem~\ref{thm:nonzeroflexibility} by means of a bifurcation argument. The first step is to perturb the profiles $u_0$ and $\rho_0$ into suitable nearby profiles that allow us to control the spectrum of the relevant operators.

\begin{remark}
Since the steady Boussinesq equations are invariant under the transformation $\u\mapsto -\u$, we may assume $u_0(y)>0$ for all $y\in [-1,1]$.
\end{remark} 

Let $\eta\in C^\infty(\mathbb{R},[0,1])$ such that $\eta(x)=1$ if $|x|\leq 1$ and $\eta(x)=0$ if $|x|\geq 2$. Fix any $s\in(0,q^{-1})$ and let $a\in(-1,1)$ be a parameter that will be free to move and $\delta>0$ a fixed number to be determined later. We define:
\begin{equation}\label{eq:pos_density_profiles}
\begin{aligned}
u_{\delta}(y)&\coloneqq u_0(y)\left(\delta^2+1-\eta(\delta^{-1}y)\right), 
\\\rho_{\delta,a}(y)&\coloneqq\rho_0(y)+\int_{-1}^y a \delta^{-s}\eta(\delta^{-1}z)\,dz.
\end{aligned}
\end{equation}

Note that by our choice of $s$, the new pairs $(u_\delta,\rho_{\delta,a})$ can be made arbitrarily close to $(u_0,\rho_0)$ in $L^q\times W^{1,q}$ if the parameter $\delta>0$ is chosen small enough. \textcolor{black}{We omit the proof of the lemma since the proof is straightforward.}

\begin{lemma}\label{lemma:perturbation close to original no stag}
    Let $(u_{\delta},\rho_{\delta,a})$ be given by \eqref{eq:pos_density_profiles} for $\delta>0$ and $a\in(-1,1)$. For all $q\in[1,\infty)$ and all $\varepsilon>0$ there exists $\delta_0>0$ such that
\begin{align}
    \Vert u_{\delta}(y) - u_0(y) \Vert_{L^q} + \Vert \rho_{\delta,a} - \rho_0 \Vert_{W^{1,q}}\leq \varepsilon/2,
\end{align}
for all $\delta< \delta_0$.
\end{lemma}

\color{black}
Let 
\[
\psi_\delta(y):=-\int_{-1}^y u_\delta(z) dz
\] be a stream function associated to $(u_\delta,0)$.
Since $u_\delta$ does not vanish, $\psi_\delta$ is strictly monotone:
\begin{align*}
   -\psi_\delta'(y) \geq u_0(y)\delta^2 > 0.
\end{align*}
Let $I_\delta:=\psi_\delta([-1,1])$ and denote its inverse by
$\varphi_\delta:I_\delta\to[-1,1]$.
Define
\begin{equation}
p_{\delta,a}(y):=-\int_{-1}^y\rho_{\delta,a}(z)dz,
\qquad
B_{\delta,a}(y):=
p_{\delta,a}(y)+\frac12u_\delta(y)^2+y\rho_{\delta,a}(y).
\label{def p B non stagnation}
\end{equation}
Setting
\begin{align}\label{theta_G_regularity_p}
\Theta_{\delta,a}:=\rho_{\delta,a}\circ\varphi_\delta \in C^4,
\qquad
G_{\delta,a}:=-B_{\delta,a}\circ\varphi_\delta\in C^4,
\end{align}
it is straightforward to check that
\begin{align}\label{eq:ODEpsiBackGround}
\rho_{\delta,a}(y)=\Theta_{\delta,a}(\psi_\delta(y)),\qquad \psi_\delta''(y)+y\Theta_{\delta,a}'(\psi_\delta(y))
+G_{\delta,a}'(\psi_\delta(y))=0.
\end{align}
Note that $\Theta_{\delta,a}$ and $G_{\delta,a}$ are $C^4$ in $I_\delta$. In the rest of the section, we will assume that they are extended to $\mathbb{R}$ so that $C^4$ regularity still holds in $\mathbb{R}$.

Next, for $\xi=jx$, and $\Delta_\xi=j^2 \partial_\xi^2 + \partial_y^2$,  we look for a nontrivial stream function $\psi\in H^3(\mathbb{T}\times [-1,1])$ such that $\psi$ is constant on each component of $\partial M$, and $(j,a)\in (\mathbb{R}\backslash\left\{0\right\})\times (-1,1)$ such that
\begin{align}\label{functional_equation_laminar_p}
\mathcal{F}(j,\psi,a):=j^2\partial_\xi^2\psi +\partial_y^2\psi + y\Theta'_{\delta,a}(\psi)+G'_{\delta,a}(\psi)=0.
\end{align}
Indeed, by Proposition~\ref{Bernoullie_function_2}, if
$(j,\psi,a)$ solves \eqref{functional_equation_laminar_p}
with $j\ne0$, then the functions
\[
\Psi(x,y):=\psi(jx,y),\qquad
\mathbf u:=\nabla^\perp\Psi,\qquad
\rho:=\Theta_{\delta,a}(\Psi)
\]
yield a stationary Boussinesq solution $(\mathbf u,\rho)$
with horizontal period $2\pi/|j|$.

Thanks to \eqref{eq:ODEpsiBackGround}, we have that 
\begin{align}
    \mathcal{F}(j,\psi_{\delta},a)=0
\end{align}
for all $a\in(-1,1)$ and all $j\in \R\backslash\left\{0\right\}$. It is in fact more convenient to work with the translated functional
\begin{align}\label{eq:defsF}
    \mathsf{F}(j,\phi,a)\coloneqq  \mathcal{F}(j, \psi_{\delta} + \phi, a),\quad \phi(\xi,\pm1)=0
\end{align}
 Using the definition of $\Theta_{\delta,a}$ and $G_{\delta,a}$, it is straightforward to see that
 \begin{align}
    \partial_\phi \sF (j,0,a)\phi = j^2\partial_\xi^2 \phi - \cL_{\delta,a}\phi,
    \label{eq: linearized operator positive}
\end{align}
where $\cL_{\delta,a}$ is a one-dimensional linear operator
\begin{align}
    \mathcal{L}_{\delta,a}\phi \coloneqq -\partial_y^2 \phi + \left( \frac{u''_{\delta}(y)}{u_{\delta}(y)} + \frac{\rho'_{\delta,a}(y)}{u^2_{\delta}(y)} \right) \phi,
    \label{eq:def Ldeltaa positive}
\end{align}
In order to apply the Crandall-Rabinowitz local bifurcation theorem (Theorem~\ref{thm:unifCR}), in the next subsection we show that $\mathcal{L}_{\delta,a}$ admits negative eigenvalues for $\delta$ small enough.

\color{black}
\subsection{Spectrum of the vertical linearized operator}
In this subsection we will show that we can control the first eigenvalue $\lambda_{\delta,a}$ of $\mathcal{L}_{\delta,a}$ with a suitable choice of $\delta$ and $a$. Our goal is to be able to prescribe any value in a sufficiently small neighborhood of $-j_\ast^2$ for some $j_\ast\in \N$. This will allow us to construct a function in the kernel of the linearized operator \eqref{eq: linearized operator positive} with controlled period.

The mechanism that allows us to prescribe $\lambda_{\delta,a}$ is that the perturbation \eqref{eq:pos_density_profiles} produces a region where the velocity is very small, of order $\delta^2$. By forcing the derivative of the density to be negative, we can produce enough negativity to shift the first eigenvalue of $\mathcal{L}_{\delta,a}$. However, perturbing the velocity and density profiles also changes the first term in parentheses in \eqref{eq:def Ldeltaa positive}, which could potentially shift the eigenvalue in the opposite direction. As we will see, the perturbation is designed so that the desired effect prevails. 

We begin by collecting some basic properties:
\begin{lemma}
For fixed $(\delta,a)\in (0,1)\times (-1,1)$, consider the operator
\[\mathcal{L}_{\delta,a}: D\subset L^2(-1,1)\,\to\, L^2(-1,1), \quad\phi\mapsto -\phi''+F_{\delta,a}\phi,\]
where 
\[D\coloneqq H^2(-1,1)\cap H^1_0(-1,1), \qquad \qquad F_{\delta,a}(y)\coloneqq \frac{u''_{\delta}(y)}{u_{\delta}(y)}+\frac{\rho'_{\delta,a}(y)}{u^2_{\delta}(y)}.\]
Then the spectrum of $\mathcal{L}_{\delta,a}$ consists of an increasing sequence of real eigenvalues $\lambda^1_{\delta,a}<\lambda_{\delta,a}^2 <\cdots$ that are simple and accumulate only at $+\infty$. The corresponding eigenfunctions are {$C^4([-1,1])$}. Furthermore, the first eigenvalue is given by
\[\lambda^1_{\delta,a}=\inf_{0\ne \phi\in H^1_0(-1,1)}\frac{\int_{-1}^1 |\phi'(y)|^2 \d y + \int_{-1}^{1}F_{\delta,a}(y)|\phi(y)|^2 \d y}{\int_{-1}^1 |\phi(y)|^2 \d y}.\]
The associated normalized eigenfunction $\phi^1_{\delta,a}$ can be chosen to be positive on $(-1,1)$.
\end{lemma}
\begin{proof}
Non-degeneracy of $u_\delta$ guarantees that the potential $F_{\delta,a}$ is continuous and bounded. The results then follow from the classical Sturm--Liouville theory (e.g. \cite[Theorem~5.11,  Lemma~5.12]{Teschl2012}).
\end{proof}

Next, we prove continuous dependence on $a$ of the first eigenvalue $\lambda_{\delta,a}^1$ and the associated normalized eigenfunction $\phi^1_{\delta,a}$, which will be necessary to apply the uniform Crandall-Rabinowitz theorem. Proving continuity in $\delta$  of the first eigenvalue will be useful for controlling  its value in Lemma~\ref{lemma: control eigenvalues no stag}. In the rest of this section, omitting the superscript in $\lambda_{\delta,a}$ or $\phi_{\delta,a}$ will refer to the first eigen-pair.
\begin{lemma}\label{lemma:properties_lambda_ep_a}
For all $(\delta,a)\in(0,1)\times(-1,1)$, the first eigenvalue $\lambda_{\delta,a}$ is continuous in $(\delta,a)\in(0,1)\times(-1,1)$. Furthermore, for a fixed $\delta\in (0,1)$, the following map is continuous: $$(-1,1)\;\to\; H^4(-1,1)\cap H^1_0(-1,1), \qquad a\,\mapsto\, \phi_{\delta,a}.$$
\end{lemma}
\begin{proof}
 Consider  $\delta,\gamma\in(0,1)$ and $a,b\in (-1,1)$ and let $(\lambda_{\delta,a},\phi_{\delta,a})$ and $(\lambda_{\gamma,b},\phi_{\gamma,b})$ be the eigen-pairs associated to $\cL_{\delta,a}$ and $\cL_{\gamma,b}$, respectively. Then,
$$
    \lambda_{\delta,a} = \langle \mathcal{L}_{\delta,a}\phi^1_{\delta,a},\phi_{\delta,a} \rangle = \langle \cL_{\gamma,b} \phi_{\delta,a}, \phi_{\delta,a} \rangle + \int_{-1}^1 (F_{\delta,a}(y)-F_{\gamma,b}(y))|\phi_{\delta,a}(y)|^2 \d y \geq \lambda_{\gamma,b} - \Vert F_{\delta,a}-F_{\gamma,b} \Vert_{L^\infty}.$$
Similarly, we have
$$
    \lambda_{\gamma,b} = \langle \cL_{\delta,a} \phi_{\gamma,b}, \phi_{\gamma,b} \rangle + \int_{-1}^1 (F_{\gamma,b}(y)-F_{\delta,a}(y))|\phi_{\gamma,b}(y)|^2 \d y \geq \lambda_{\delta,a} - \Vert F_{\delta,a}-F_{\gamma,b} \Vert_{L^\infty},
$$
Thus, $|\lambda_{\delta,a}-\lambda_{\gamma,b}| \leq \Vert F_{\delta,a}-F_{\gamma,b} \Vert_{L^\infty}$, which implies that $\lambda_{\delta,a}$ is continuous in $(\delta,a)$ because so is $F_{\delta,a}$.

Let us now prove continuity in $a$ of the eigenfunction. Fix $b\in (-1,1)$ and let $\{a_n\}_{n=1}^\infty\subset (-1,1)$ be a sequence converging to $b$. Since $F_{\delta,a}$ depends smoothly on $a$, we have $F_{\delta,a_n}\to F_{\delta,b}$ in $C^0([-1,1])$ and $\lambda_{\delta,a_n}\to\lambda_{\delta,b}$. Testing the eigenvalue equation against $\phi_{\delta,a_n}$ and using $\norm{\phi_{\delta,a_n}}_{L^2}=1$ lead to:
\[\norm{\phi_{\delta,a_n}'}_{L^2}^2=\lambda_{\delta,a_n}-\int_{-1}^1F_{\delta,a_n}\phi_{\delta,a_n}^2\leq \abs{\lambda_{\delta,a_n}}+\norm{F_{\delta,a_n}}_{L^\infty}.\]
The right-hand side is bounded uniformly in $n$. Thus $\phi_{\delta,a_n}$ is bounded in $H^1_0(-1,1)$. Therefore, every subsequence has a further subsequence converging weakly in $H^1_0$ and strongly in $L^2$ to some $\phi\in H^1_0$. Strong $L^2$ convergence preserves nonnegativity and normalization:
\begin{equation}
\phi\geq0\quad \text{a.e.,} \qquad \qquad \norm{\phi}_{L^2}=1.
\label{aux continuous dependence eigenalue}
\end{equation}
On the other hand, multiplying the eigenvalue equation by a fixed test function {$\varphi\in C^\infty_c(-1,1)$}, integrating by parts and passing to the limit yields:
\[0=\int_{-1}^1(\phi_{\delta,a_n}'\varphi'+F_{\delta,a_n}\phi_{\delta,a_n}\varphi-\lambda_{\delta,a_n}\phi_{\delta,a_n}\varphi)\to \int_{-1}^1(\phi'\varphi'+F_{\delta,b}\phi\varphi-\lambda_{\delta,b}\phi\varphi).\]
Since the test function $\varphi$ is arbitrary, we see that $\phi$ is a weak solution of the eigenvalue equation for $\lambda_{\delta,b}$. Uniqueness and \eqref{aux continuous dependence eigenalue} imply $\phi=\phi_{\delta,b}$.  Since every subsequence has a further
subsequence with this same limit, the entire sequence satisfies $\phi_{\delta,a_n}\to\phi_{\delta,b}$ in $L^2$.

We conclude that $(-1,1)\to L^2, \;a\mapsto \phi_{\delta,a}$ is continuous. Hence, the right-hand side of the equation $-\phi_{\delta,a}''=(\lambda_{\delta,a}-F_{\delta,a})\phi_{\delta,a}$ is continuous as a map $(-1,1)\to L^2$, from which we deduce that $a\mapsto \phi_{\delta,a}$ is continuous as a map $(-1,1)\to H^2$. Repeating this argument bootstraps the regularity to $H^4$.
\end{proof}

Finally, we prove that we can control $\lambda_{\delta,a}$ with a suitable choice of the parameters. Recall the value $\delta_0$ given by Lemma~\ref{lemma:perturbation close to original no stag}.
\begin{lemma}\label{lemma: control eigenvalues no stag}
There exist $\delta_\ast\in(0,\delta_0)$, an integer $j_\ast\geq 1$ and intervals $\mathcal{J}\ni -j_\ast^2$ and $\mathcal{A}\subset (-1,1)$ such that:
\begin{enumerate}
\renewcommand{\labelenumi}{(\roman{enumi})}
    \item $a\mapsto \lambda_{\delta_\ast,a}$ is strictly increasing.
    \item For any $\lambda\in \mathcal{J}$ there exists $a\in \cA$ such that $\lambda_{\delta_\ast,a}=\lambda$. In particular, there exists $a_\ast\in\cA$ such that $\lambda_{\delta_\ast,a_\ast}=-j_\ast^2$.
    \item $0$ is not an eigenvalue of $\mathcal{L}_{\delta_\ast,a}$ for any $a\in \cA$.
\end{enumerate}
\end{lemma}
\begin{proof}
For the moment, we keep $\delta>0$ as a parameter. Since the potential $F_{\delta,a}$  depends smoothly on the parameter $a$, it follows from \cite[Chapter VIII, Theorem~2.6]{kato_perturbation_1995} that the $n$-th eigenvalue $\lambda^n_{\delta,a}$ of $\mathcal{L}_{\delta,a}$ is differentiable in $a$ and its derivative is given by
\begin{equation}
\label{eq: derivative eigenvalues no stagnation1}
\frac{\d}{\d a}\lambda^n_{\delta,a}=\int_{-1}^1\frac{\delta^{-s}\eta(\delta^{-1}y)}{u_0^2[\delta^2+1-\eta(\delta^{-1}y)]^2}\abs{\phi^n_{\delta,a}(y)}^2\,{\d y}\geq \left(2\delta^s\norm{u_0}_{L^\infty}^2\right)^{-1}\int_{-\delta}^{\delta}\abs{\phi^n_{\delta,a}}^2,
\end{equation}
where $\phi_{\delta,a}^n$ is an $L^2$-normalized eigenfunction
associated to $\lambda_{\delta,a}^n$. It solves
\begin{align}\label{eigenequation}
-\phi'' + F_{\delta,a}\phi = \lambda_{\delta,a}^n\phi,\quad \phi(-1)=\phi(1)=0,
\end{align}
and a nonzero solution cannot be zero in an open interval in $(-1,1)$. Indeed, if $\phi$ vanished identically on a nonempty open
interval, then at an interior point $y_0$ we would have
$\phi(y_0)=\phi'(y_0)=0$. The initial-value problem
\[\phi''(y)=(F_{\delta,a}(y)-\lambda^n_{\delta,a})\phi(y),
\qquad \phi(y_0)=\phi'(y_0)=0,\]
has only the zero solution on $(-1,1)$ by uniqueness \cite[Theorem~3.9]{Teschl2012}, contradicting $\|\phi\|_{L^2}=1$. Hence,
\begin{align}\label{eq: derivative eigenvalues no stagnation}
\frac{\d}{\d a}\lambda^n_{\delta,a}\ge \frac{\delta^{-s-4}}{\rVert u_0\rVert_{L^\infty}^2}\int_{-\delta}^{\delta}|\phi^n_{\delta,a}|^2dy >0,
\end{align}
so $a\mapsto \lambda^n_{\delta,a}$ is strictly increasing for any $\delta>0$.

Regarding \textit{(ii)}, since $a\mapsto \lambda^1_{\delta,a}$ is continuous and strictly increasing, it suffices to prove that $\lambda^1_{\delta_\ast,a_\ast}=-j_\ast^2$ for some integer $j_\ast\geq 1$ and some values of the parameters $\delta_\ast\in(0,\delta_0)$ and $a_\ast\in(-1,1)$. Let us consider the infimum of $\lambda^1_{\delta,0}$ for $\delta>0$. If it equals $-\infty$, by continuity of $\delta\mapsto \lambda^1_{\delta,0}$, there exists $\delta_\ast\in(0,\delta_0)$ such that $\lambda^1_{\delta_\ast,0}=-j_\ast^2$ for a sufficiently large integer $j_\ast$. On the other hand, if the infimum is finite, there exists an integer $j_\ast\geq 1$ such that $\lambda^1_{\delta,0}>-j_\ast^2$ for any $\delta>0$. We claim that, for $\delta_\ast>0$ sufficiently small, we have $\lambda^1_{\delta_\ast,-1}<-j_\ast^2$. By continuity of $a\mapsto \lambda^1_{\delta_\ast,a}$, there would then exist $a_\ast\in(-1,0)$ such that $\lambda^1_{\delta_\ast,a_\ast}=-j_\ast^2$, thus completing the proof of \textit{(ii)}.

Indeed, setting $F_0\coloneqq u_0^{-1}u_0''+u_0^{-2}\rho_0'$, a straightforward computation yields
\begin{align*}
F_{\delta,-1}(y)&=F_0(y)+\frac{\rho_0'}{u_0^2}\left(\frac{1}{(1+\delta^2)^2}-1\right)+\frac{\rho_0'}{u_0^2}\left(\frac{1}{\left[\delta^2+1-\eta(\delta^{-1}\,y\,)\right]^2}-\frac{1}{(1+\delta^2)^2}\right)-\frac{2\delta^{-1}\,u_0'\,\eta'(\delta^{-1}\,y\,)}{u_0\,[\delta^2+1-\eta(\delta^{-1}\,y\,)]}\\[4pt]
&\hspace{12pt}-\delta^{-2}\frac{\eta''(\delta^{-1}\,y\,)}{\delta^2+1-\eta(\delta^{-1}\,y\,)}-\frac{\delta^{-s}\eta(\delta^{-1}\,y\,)}{u_0^2\left[\delta^2+1-\eta(\delta^{-1}\,y\,)\right]^2}\\[6pt]
&\equiv F_0(y)+F_{1,\delta}(y)+F_{2,\delta}(y)+F_{3,\delta}(y)+F_{4,\delta}(y)- F_{5,\delta}(y).
\end{align*}
Since $\inf\abs{u_0}>0$ and the cutoff $\eta(\delta^{-1}\,\cdot\,)$ takes values in $[0,1]$ and is supported in $[-2\delta,2\delta]$, we see that there exists a universal constant depending only on $u_0$ and $\rho_0$ such that:
$$\norm{F_0}_{L^1}+\norm{F_{1,\delta}}_{L^1}\leq C, \qquad \norm{F_{3,\delta}}_{L^1}\leq C\delta^{-2},\qquad\norm{F_{2,\delta}}_{L^1}+\norm{F_{4,\delta}}_{L^1}\leq C\delta^{-3}.$$
On the other hand, since the cutoff $\eta(\delta^{-1}\,\cdot\,)$ is nonnegative and equals 1 on $[-\delta,\delta]$, we have
$$
   \int_{-1}^{1}F_{5,\delta}(y)|\phi(y)|^2\,{\d y} \geq \frac{\delta^{-s}}{\delta^{4} \Vert u_0\Vert_{L^\infty}^2} \int_{-\delta}^{\delta}|\phi(y)|^2 \d y
$$
for any $\phi\in H^1_0(-1,1)$. In particular, for $\varphi(y)=1-|y|\in H^1_0(-1,1)$ we have
$$
   \int_{-1}^{1}F_{5,\delta}(y)|\varphi(y)|^2\,{\d y} \geq \delta^{-3-s}\norm{u_0}_{L^\infty}^{-2},
$$
where we have assumed that $2\delta\leq1$ so that $2\varphi(y)\geq 1$ for $y\in[-\delta,\delta]$. Testing the Rayleigh quotient with this particular $\varphi$ leads to
\begin{align}
    \lambda^1_{\delta,-1} \leq \frac{2 + C(1+\delta^{-2}+\delta^{-3}) - \Vert u_0 \Vert_{L^\infty}^{-2}\delta^{-3-s}}{2/3}.
\end{align}
Here we have used that $\varphi$ satisfies $\norm{\varphi}_{L^\infty}=1$, $\Vert \varphi \Vert_{L^2}^2 = 2/3$ and $\Vert \varphi'\Vert_{L^2}^2=2$. Since $s>0$, we will have $\lambda^1_{\delta_\ast,-1}<-j_\ast^2$ for $\delta_\ast>0$ sufficiently small, which proves the claim.

Finally, let us prove \textit{(iii)}. We consider the value $\delta=\delta_\ast$ to be fixed and we omit the dependence on $\delta$, for simplicity. First of all, if $\lambda_{a_\ast}^n\neq 0$ for all $n\geq 2$, then property \textit{(iii)} readily holds for sufficiently small $\mathcal{J}$ and $\mathcal{A}$. Indeed, recall that the eigenvalues depend continuously on $a$ and that they are ordered. Hence, the first eigenvalues to cross the origin would be the two closest eigenvalues. By continuity, this can be prevented by choosing $\mathcal{A}$ to be a sufficiently small neighborhood of $a_\ast$. We then reduce $\mathcal{J}\ni -j_\ast^2$ accordingly, so that \textit{(ii)} still holds.

Let us consider the case that there exists $n_0\geq 2$ such that $\lambda^{n_0}_{a_\ast}=0$. Our goal is to introduce an additional perturbation to break the degeneracy so that we can perform the reduction from the previous paragraph. We fix $g\in C^\infty_c(0,1)$ such that 
\begin{equation}
    \int_{-1}^1g\abs{\phi^1_{a_\ast}}^2=0, \hspace{40pt}\int_{-1}^1g\abs{\phi^{n_0}_{a_\ast}}^2=1,
    \label{eq: properties of g}
\end{equation}
which can be achieved by an appropriate linear combination of two bump functions concentrated around two suitable points. Then, for a small parameter $b\in \R$ we consider $$\rho_{a,b}(y)\coloneqq \rho_a(y)+\int_{-1}^yb\hspace{1pt}g(t) \hspace{1pt}u_{\delta_\ast}(t)^2\d t.$$We denote by $\mathcal{L}_{a,b}$ the operator associated to this density and by $\lambda^n_{a,b}$ and $\phi^n_{a,b}$ its eigenvalues and normalized eigenfunctions, respectively. Let $H(a,b)\coloneqq \lambda^1_{a,b}+j_\ast^2$, which satisfies $H(a_\ast,0)=0$ and $$\partial_aH(a_\ast,0)=\int_{-1}^1\delta_\ast^{-s}\eta(\delta_\ast^{-1}\,\cdot\,)\,u_{\delta_\ast}^{-2}\,|\phi^1_{a_\ast,0}|^2>0,$$as in \eqref{eq: derivative eigenvalues no stagnation}. Hence, by the implicit function theorem, there exists $\epsilon>0$ and a function $a\in C^1((-\epsilon,\epsilon),\mathcal{A})$ such that $a(0)=a_\ast$ and $\lambda^1_{a(b),b}=-j_\ast^2$. Differentiating this expression at $b=0$ and using \cite[Chapter VIII, Theorem~2.6]{kato_perturbation_1995} yields:
\[0=\partial_a \lambda^1_{a_\ast,0}\,a'(0)+\partial_b\lambda^1_{a_\ast,0}=a'(0)\int_{-1}^1\delta_{\ast}^{-s}\eta(\delta_\ast^{-1}\,\cdot\,)\,u_{\delta_\ast}^{-2}\,|\phi^1_{a_\ast,0}|^2+\int_0^1g\abs{\phi^1_{a_\ast,0}}^2.\]
By the properties of $g$, the last term vanishes. Since the integral multiplying $a'(0)$ is strictly positive, we deduce $a'(0)=0$. Next, let $\mu(b)\coloneqq \lambda^{n_0}_{a(b),b}$, which satisfies $\mu(0)=0$, by assumption. Meanwhile:
\[\partial_b\mu(0)=\partial_a\lambda^{n_0}_{a_\ast,0}\,a'(0)+\partial_b\lambda^{n_0}_{a_\ast,0}=\int_{-1}^1g\abs{\phi^{n_0}_{a_\ast,0}}^2=1,\]
where we have substituted \eqref{eq: properties of g} and used $a'(0)=0$. We conclude that for a fixed $b_\ast>0$ sufficiently small so that the eigenvalues for $n\neq n_0$ are not an issue, we will have $\lambda^1_{a(b_\ast),b_\ast}=-j_\ast^2$ and $\lambda^n_{a(b_\ast),b_\ast}\neq 0$ for all $n\in \N$. We can now perform the reduction of $\mathcal{A}$ and $\mathcal{J}$ discussed at the beginning of the proof of property \textit{(iii)}, which completes the proof.
\end{proof}

\begin{remark}
\label{remark abuse of notation}
In the previous lemma we potentially change the density, so that it is no longer given by \eqref{eq:pos_density_profiles}. However, once $b_\ast>0$ sufficiently small is fixed, the additional perturbation is smooth, as small as desired in any $C^m$ or $H^k$ norm, and independent of $a$. It thus has no impact in the subsequent arguments. Therefore, with some abuse of notation, we will denote $\rho_{a,b_\ast}$ and $a(b_\ast)$ simply as $\rho_a$ and $a_\ast$. We also drop the dependence on $b_\ast$ in the notation for eigenvalues and eigenfunctions. Similarly, the quantities defined in \eqref{theta_G_regularity_p}-\eqref{eq:defsF} will keep the same notation but must be understood to be defined using the new density.
\end{remark}

\color{black}
\subsection{Functional framework and proof of Theorem \ref{thm:nonzeroflexibility}}
We introduce the Banach spaces
\begin{align}\label{eq:defX}
    X\coloneqq  \lbrace \phi(\xi,y)\in H^4(\T\times[-1,1]) : \phi(\xi,1)=\phi(\xi,-1)=0, \:\phi(\xi,y) = \phi(-\xi,y) \rbrace, 
\end{align}
and
\begin{align}\label{eq:defY}
    Y\coloneqq  \lbrace \phi(\xi, y) \in H^2(\T\times[-1,1]) : \phi(\xi,y) = \phi(-\xi,y) \rbrace.
\end{align}
For the fixed $\varepsilon>0$ given by Theorem~\ref{thm:nonzeroflexibility}, we fix $\delta_\ast>0$ sufficiently small, as given by Lemma~\ref{lemma:perturbation close to original no stag} and Lemma~\ref{lemma: control eigenvalues no stag}. We consider this value of $\delta=\delta_\ast$ to be fixed and, from now on, omit the dependence on $\delta$, for simplicity. We only keep the index in the velocity $u_{\delta_\ast}$, its streamfunction $\psi_{\delta_\ast}$ and its inverse $\varphi_{\delta_\ast}$.

We define the function $j:{\cA}\rightarrow\R$ given by $j(a)=\sqrt{-\lambda_{a}}$, which is continuous for all $a\in {\cA}$ due to Lemma \ref{lemma:properties_lambda_ep_a} and satisfies $j(a_\ast)=j_\ast$. We further define the functional 
\begin{align}\label{eq:defG}
    \G:\R\times X\times \cA\rightarrow Y, \quad \G(\mu,\phi,a)\coloneqq  \sF(j(a) + \mu, \phi, a). 
\end{align}
To use the uniform Crandall-Rabinowitz local bifurcation Theorem, we first show that $\G(\mu,\phi,a)$ satisfies the conditions of Theorem \ref{thm:unifCR}. \textcolor{black}{While Theorem~\ref{thm:unifCR} argues for $\cA$ an open neigborhood of the origin, it is likewise valid for $\cA$ an open neigborhood of $a_*$, mutatis mutandis.}

\begin{lemma}\label{lemma:conditionsCR}
    There exist $\cI^\mu\subset \R$ and $U\subset X$ open sets with $(0,0,a_\ast)\in\cI^\mu\times U\times \cA$ such that the operator $\G(\mu,\phi,a):\cI^\mu\times U\times \cA\rightarrow Y$ defined in \eqref{eq:defG} satisfies the conditions of Theorem \ref{thm:unifCR}.
\end{lemma}
\begin{proof}
    First, it is immediate to see that $\G$ maps $\R\times X\times \cA$ to $Y$. Next, since
    \begin{align}
        \G(\mu,\phi,a) = (j(a)+\mu)^2\partial_\xi^2\phi+\partial_y^2(\psi_{\delta_\ast}+\phi) + y\Theta'_{a}(\psi_{\delta_\ast}+\phi) + G'_{a}(\psi_{\delta_\ast} + \phi),
    \end{align}
the  Fr\'{e}chet partial derivatives of $\G(\mu,\phi,a)$ read
\begin{equation}\label{eq:Frechet_G}
\begin{aligned}
    \partial_\mu\G(\mu,\phi,a)&= 2(j(a)+\mu)\partial_\xi^2\phi \\
    \partial_\phi\G(\mu,\phi,a)\widetilde{\phi} & = (j(a)+\mu)^2 \partial_\xi^2 \widetilde{\phi} +\partial_y^2 \widetilde{\phi} + y\Theta''_{a}(\psi_{\delta_\ast}+\phi)\widetilde{\phi} + G''_{a}(\psi_{\delta_\ast}+\phi)\widetilde{\phi}, \\
    \partial^2_{\mu,\phi}\mathcal{G}(\mu,\phi,a)\widetilde{\phi} &=2(j(a)+\mu)\partial_\xi^2 \widetilde{\phi}, 
\end{aligned}
\end{equation}
for all $\widetilde{\phi}\in X$ and $(\mu,\phi,a)\in \R\times X\times \cA$.  We recall from \eqref{theta_G_regularity_p} that $\Theta_a$ and $G_a$ are $C^4$. With this and the continuity of $a\mapsto j(a)$ for $a\in \cA$, we conclude that
\[
\mathcal G,\partial_\mu\mathcal G\in C(\mathbb R\times X\times\mathcal A;Y),
\qquad
D_\phi\mathcal G,\partial_\mu D_\phi\mathcal G
\in C(\mathbb R\times X\times\mathcal A;\mathcal L(X,Y)).
\] We now turn to verifying assumptions \ref{CR_condi_1}-\ref{CR_condi_3} of Theorem~\ref{thm:unifCR}. For this, we note that
    \begin{align}
        L_a\phi:= \partial_\phi\G(0,0,a)\phi =j(a)^2\partial_\xi^2\phi - \cL_{a}\phi = -\lambda_{a}\partial_\xi^2\phi - \cL_{a}\phi
    \end{align}
    for all $a\in \cA$ and all $\phi\in X$. We denote the range of $L_a$ by
\[
\mathcal R_a:=\{L_a h:h\in X\}\subset Y.
\]

    \subsubsection*{Proof of \ref{CR_condi_1}} Due to Lemma \ref{lemma: control eigenvalues no stag}, for all $a\in \cA$ the minimum eigenvalue of $\cL_{a}$ is $\lambda_{a}<0$. Let $\phi_{a}\in H^1_0((-1,1))$ be the unique normalized positive eigenfunction associated to the eigenvalue $\lambda_a$. Note that $\lambda_a$ being the minimal eigenvalue and $\cL_a$ having no non-trivial kernel shows that 
    \begin{align}\label{kernel_ondi_section5}
        \ker (L_{a})= \vspan   \lbrace \Phi_{a} \rbrace \subset X
    \end{align}
    for $\Phi_a\coloneqq \beta\phi_{a}(y)\cos(\xi)\in X$ with $\beta>0$ such that $\Vert \Phi_a\Vert_{L^2}=1$. 
    Writing
    \begin{align*}
        L_a = (-\lambda_a\partial_\xi^2 +\partial_y^2  - 1) + \left( 1 - \left( \frac{u_{\delta}''(y)}{u_{\delta}(y)} + \frac{\rho_{\delta,a}'(y)}{u_{\delta}^2(y)}\right)\right),
    \end{align*}
    We see that $L_a$ is the sum of an isomorphism and a compact operator. Hence, $L_a$ is Fredholm of index zero.  Combining this with \eqref{kernel_ondi_section5}, we obtain that $\mathcal{R}_a$ has codimension $1$.
Since $L_a$ is symmetric, we also have
\[
\langle L_a h,\Phi_a\rangle_{L^2}
=\langle h,L_a\Phi_a\rangle_{L^2}=0
\]
for every $h\in X$. Therefore,
\[
\mathcal R_a
\subset\{r\in Y:\langle r,\Phi_a\rangle_{L^2}=0\}.
\]
Since both spaces have codimension one in $Y$, they coincide.
Setting $\mathcal Z_a:=\operatorname{span}\{\Phi_a\}$, we obtain
\[
Y=\mathcal R_a\oplus\mathcal Z_a,
\] 
    which shows $\cR_a$ is a closed subspace of $Y$.  Furthermore, $a\mapsto \Phi_a\in X$ is continuous because so is $a\mapsto \phi_{a}$, by Lemma~\ref{lemma:properties_lambda_ep_a}. 


    \subsubsection*{Proof of \ref{CR_condi_2}} According to \eqref{eq:Frechet_G} we have
    \begin{align}
        \partial^2_{\mu,\phi} \G(0,0,a)\Phi_a = 2j(a)\partial_\xi^2 [\beta\hspace{1pt}\phi_{a}(y)\cos(\xi)] =  - 2j(a)\Phi_a \not \in \mathcal{R}_a
    \end{align}
    and the transversal condition is satisfied.
    \subsubsection*{Proof of \ref{CR_condi_3}} Let us  denote the $L^2$- orthogonal complement of the kernel by
    \begin{align}
        \cW_a = \lbrace h\in X : \langle h, \Phi_a \rangle = 0 \rbrace.
    \end{align}  We define the linear map
    \begin{align}
        T_a^X:\cW_{a_\ast}\rightarrow \cW_a, \quad T_a^X w = w - \langle w,\Phi_a \rangle\Phi_a
    \end{align}
    where $\langle \cdot, \cdot \rangle$ denotes the usual $L^2$ inner product. It is immediate to see that $T_a^X$ is linear, for all $a\in \cA$. \textcolor{black}{Moreover, it is an isomorphism from $\cW_{a_\ast}$ to $\cW_a$; the inverse map is given by
    \begin{align}\label{inverse_op_formula_p}
        (T_a^X)^{-1}:\cW_a\rightarrow \cW_{a_\ast}, \quad (T_a^X)^{-1} w = w - \frac{\langle w,\Phi_{a_\ast} \rangle}{\langle \Phi_{a_\ast},\Phi_a\rangle}\Phi_a
    \end{align}
    Since $\Phi_a\rightarrow\Phi_{a_\ast}$ as $a\rightarrow a_\ast$ and $\Phi_{a_\ast} \neq 0$, the denominator in the second term of \eqref{inverse_op_formula_p} does not vanish when $a$ is sufficiently close to $a_\ast$. Hence by shrinking the set $\mathcal{A}$, if necessary, it follows  that $(T^X_a)^{-1}$ is a bounded operator for $a\in \mathcal{A}$. Since $T^X_a$ is clearly bounded and its inverse is also bounded, it is indeed an isomorphism. Lastly, since the map $a\mapsto \Phi_a$ is continuous, the map $a\mapsto T_a^X\in \cL(\cW_{a_*},X)$ is  continuous in operator norm for all $a\in \cA$.}
\end{proof}



We are now in a position to complete the proof of Theorem \ref{thm:nonzeroflexibility}.

\begin{proof}[Proof of Theorem \ref{thm:nonzeroflexibility}]
In view of Lemma \ref{lemma:conditionsCR}, the functional $\G(\mu,\phi,a)$ satisfies the conditions of the uniform Crandall-Rabinowitz bifurcation Theorem \ref{thm:unifCR}. Hence, for any compact neighborhood $a_\ast\in \cA'\subset \cA$, there exist open neighbourhoods $0\in \mathcal{U}_{\cA'}\subset \cI^\mu\times \mathcal{W}_{a_\ast}$, and $0\in \cI_{\cA'}^s\subset\R$ and continuous functions 
    \begin{align}
        (\kappa(s,a),\varpi(s,a)):\cI_{\cA'}^s\times \cA'\rightarrow\mathcal{U}_{\cA'}
    \end{align}
    with $(\kappa(0,a),\varpi(0,a))=0$ for all $a\in \cA'$ and such that
    \begin{align}
        \G(\kappa(s,a),s\Phi_a + sT_a^X\varpi(s,a),a) =0
    \end{align}
    for all $(s,a)\in \cI_{\cA'}^s\times \cA'$. Moreover, for all $a\in {\cA'}$,
\begin{align}
        \lbrace (s,\mu,w)\in (\cI_{\cA'}^s \setminus \lbrace 0 \rbrace)\times\mathcal{U}_{\cA'} : \G(\mu, s\Phi_a+sT_a^X w, a) = 0,  \rbrace &=  \lbrace (s,\kappa(s,a), \varpi(s,a)) : s\in \cI^s_{\cA'}\setminus \lbrace 0 \rbrace \rbrace.
\end{align}
Thus, we find that
    \begin{align}
        \sF\left(j(a)+ \kappa(s,a), \; s\Phi_a + sT_a^X\varpi(s,a),\;a\right)=0
    \end{align}
    for all $(s,a)\in \cI_{\cA'}^s\times \cA'$. Next, recall that $j(a)=\sqrt{-\lambda_a}$, so $j(a_\ast)=j_\ast$ by definition of $a_\ast$. Since $a\mapsto -\lambda_a$ is strictly decreasing in $a$, for $a_1,a_2\in \mathcal{A}'$ such that $a_1<a_\ast<a_2$, we have $j(a_2)<j_\ast<j(a_1)$. Since $\kappa$ is continuous and $\kappa(0,a)=0$, for sufficiently small and non-zero $s$ we have
\[j(a_2)+\kappa(s,a_2)<j_\ast<j(a_1)+\kappa(s,a_1).\]
Therefore, by the intermediate value theorem, for every non-zero sufficiently small $s$ there exists $a(s)\in \mathcal{A}'$ such that $j(a(s))+\kappa(s,a(s))=j_\ast\in \N$. Setting
$$\widetilde\psi \coloneqq  \psi_{\delta_\ast} + s\Phi_{a(s)} + sT_{a(s)}^X\varpi(s,a(s)).$$
For $s\ne0$ sufficiently small, define
\[
\psi(x,y):=\widetilde\psi(j_\ast x,y),\qquad
\mathbf u:=\nabla^\perp\psi,\qquad
\rho:=\Theta_{a(s)}\circ\psi.
\]
Then $(\mathbf u,\rho)$ is a stationary solution to the
Boussinesq equations on $\mathbb T\times[-1,1]$ such that
    \begin{align}
        \Vert \nabla^\perp \psi - (u_0,0)\Vert_{L^q} + \Vert \rho - \rho_0 \Vert_{W^{1,q}}< \varepsilon.
    \end{align}
Furthermore, it has nontrivial $x$-dependence because so does $\Phi_a$, and $\varpi(s,a)=o(1)$ as $s\to 0$, therefore $sT^X_{a(s)}\varpi(s,a(s))=o(s)$. This concludes the proof of Theorem \ref{thm:nonzeroflexibility}. 
\end{proof}

\section{Local rigidity for non-vanishing stratified shear flows}\label{sec:local_rigidity non-zero_2}


In this section we prove that there exist shear flows without stagnation points that are isolated from non-shears in suitable topologies. Throughout this section, we consider a fixed shear flow $(\u_0,\rho_0)\in C^2(\overline{M})\times C^1(\overline{M})$ of the form $\u_0=(u_0(y),0)$ and $\rho_0=\rho_0(y)$ satisfying
\begin{equation}
u_0(y)\geq m >0\qquad \forall y\in[-1,1].
\label{eq: lower bound u0}
\end{equation} We further assume there exists $c_0>0$ such that
\begin{equation}
\label{eq:coercivity Q_0}
    Q_0(w) := \int_M u_0^2 |\nabla w |^2 + \rho_0' |w|^2 \geq c_0 \Vert w \Vert_{H_0^1}^2,
\end{equation}
for all $w\in H_0^1(M)$.


\begin{remark}
    A sufficient condition for the quadratic form $Q_0$ to be coercive is $$\Vert (\rho_0')_{-} \Vert_{L^\infty} <  \mu_0(M)\,m^2,$$ where $\mu_0(M)$ denotes the first Dirichlet eigenvalue of $M$ and $(\rho_0')_{-} = \max\lbrace 0 , - \rho_0' \rbrace$. In particular, $\rho_0'\geq 0$ would suffice. 
\end{remark}
\begin{remark}
Writing $w=\frac{g}{u_0}$ for some $g\in H_0^1(M)$, we have 
    $$Q_0(w) = \int_M |\nabla g|^2 + \left( \frac{u''_0}{u_0} + \frac{\rho_0'}{u_0^2} \right) |g|^2,$$
    which relates $Q_0$ and the operator $-\Delta + \left( \frac{u''_0}{u_0} + \frac{\rho_0'}{u_0^2} \right)$ that appears in Section~\ref{sec: bifurcation non-zero}. In that section we perturb the initial shear to break coercivity of $Q_0$. 
\end{remark}

In the following theorem we show that $(\u_0,\rho_0)$ is isolated from non-shears in $L^\infty\times W^{1,q}$, showing the sharpness of Theorem~\ref{thm:nonzeroflexibility}.
\begin{theorem}\label{thm rigidity non-zero Linfty}
Let $q\in(1,\infty)$. Let $(\u_0,\rho_0)\in C^2(\overline{M})\times C^1(\overline{M})$ be a shear flow satisfying \eqref{eq: lower bound u0} and \eqref{eq:coercivity Q_0}. There exists $\varepsilon_0>0$ such that if a steady Boussinesq solution $(\u,\rho)\in C^2(\overline{M})\times C^1(\overline{M})$ satisfies
\[\norm{\u-\u_0}_{L^\infty}+\norm{\rho-\rho_0}_{W^{1,q}}<\varepsilon_0,\]
then $\u=(u(y),0)$ for some $u(y)$ and $\rho=\rho(y)$.
\end{theorem}
\begin{proof}
Writing $\u=(u,v)$, it suffices to prove that $v\equiv 0$. Indeed, by incompressibility, $u$ would then depend only on $y$. Furthermore, by reducing $\varepsilon_0>0$, we may assume that $\inf\abs{u}>0$. By the second equation in \eqref{eq:BoussinesqVel}, we conclude that $\rho$ would not depend on $x$. Since $\u$ and $\rho$ would not depend on $x$, neither would $p$. Therefore, it suffices that $v\equiv 0$ for $(\u,\rho,p)$ to be a shear flow.

To prove that $v$ vanishes, we recall from \eqref{vorticity_eq_rigidt2} that 
\[-u\Delta v+v\Delta u+\frac{v}{u}\partial_y\rho=0.\]
Let $w\coloneqq \frac{v}{u}$, so that $v=uw$. By the chain rule, we have:
\[-u\Delta(uw)+uw\Delta u=-u^2\Delta w-2u\nabla u\cdot\nabla w=-\Div(u^2\nabla w).\]
Therefore, $w$ satisfies
\[-\Div(u^2\nabla w)+\partial_y\rho\, w=0.\]
Since $\u$ is tangent to $\partial M$, the second component $v$ must vanish at $\partial M$. Hence, so does $w$. Multiplying the previous equation by $w$ and integrating by parts, we finally obtain:
\[\int_Mu^2\abs{\nabla w}^2+\partial_y\rho\,w^2=0.\]
Let us denote the left-hand side of the previous equation as $Q(w)$ and let us compare it with $Q_0(w)$:
\begin{align*}
\abs{Q(w)-Q_0(w)}&\leq \int_M\abs{u^2-u_0^2}\,\abs{\nabla w}^2+\int_M\abs{\partial_y\rho-\partial_y\rho_0}\,w^2\\&\leq \norm{u^2-u_0^2}_{L^\infty}\norm{w}_{H^1}^2+\norm{\partial_y \rho-\partial_y\rho_0}_{L^q}\norm{w}_{L^{2q'}}^2 \\&\leq C\left( \norm{u-u_0}_{L^\infty}+\norm{\rho-\rho_0}_{W^{1,q}} \right)\norm{w}_{H^1}^2\leq C\varepsilon_0\norm{w}_{H^1}^2
\end{align*}
for some constant $C$ depending on $q$ and $\norm{u}_{L^\infty}\leq \norm{u}_{L^\infty} + \ep_0$. Here $q'\in (1,+\infty)$ satisfies $\frac{1}{q}+\frac{1}{q'}=1$ and we have used the Sobolev embedding $H^1\hookrightarrow L^{2q'}$.

Since $\abs{Q_0(w)}\geq c_0\norm{w}_{H^1}^2$, taking $\varepsilon_0\leq \frac{c_0}{2C}$ yields $\abs{Q(w)}\geq \frac{1}{2}c_0\norm{w}_{H^1}^2$. Taking into account that $Q(w)=0$, we conclude that $w=0$, which implies $v=0$. This completes the proof.
\end{proof}

We will now consider the behavior if we measure the velocity in Sobolev spaces. The previous theorem readily shows that $(\u_0,\rho_0)$ is isolated from non-shears in $H^s\times W^{1,q}$ for $s>1$, due to the embedding $H^s\hookrightarrow L^{\infty}$. 

However, if we restrict to nearby solutions that are laminar, we can improve the rigidity to any $s>\frac12$. Here we only consider solutions in which the field $\u=(u,v)$ satisfies $\inf u>0$; in particular, its streamlines are graphs. This condition holds for the solutions in Theorem \ref{thm:nonzeroflexibility}, so the following theorem proves sharpness of our construction in Sobolev spaces:
\begin{theorem}\label{thm rigidity non-zero Sobolev}
Let $s>\frac{1}{2}$ and $q\in(1,\infty)$. Let $(\u_0,\rho_0)\in C^2(\overline{M})\times C^1(\overline{M})$ be a stratified shear flow satisfying \eqref{eq: lower bound u0} and \eqref{eq:coercivity Q_0}. There exists $\varepsilon_0>0$ such that if steady Boussinesq solution $(\u,\rho)$ with $\u=(u,v)\in C^2(\overline{M})\times C^1(\overline{M})$ satisfies
\begin{align*}
\inf_M u&>0, \\
\norm{\u-\u_0}_{H^s}+\norm{\rho-\rho_0}_{W^{1,q}}&\leq\varepsilon_0,
\end{align*}
then $v\equiv 0$, while $u$ and $\rho$ depend only on $y$.
\end{theorem}
\begin{remark}
The main difficulty of the proof is the lack of quantitative control on $\inf u$.
\end{remark}
\begin{proof}
The main idea is to construct an energy functional $\mathcal{K}$ that measures the deviation from being a shear flow and controls the kinetic part of $Q_0$. We also derive an exact identity involving $\mathcal{K}$ and a stratification term which, after normalization and passage to the limit, recovers the density part of $Q_0$. If the theorem were false, this identity and a compactness argument would yield a nonzero $z\in H^1_0(M)$ satisfying $Q_0(z)\leq0$, which contradicts \eqref{eq:coercivity Q_0}.

To improve readability, we break down the proof into three different steps. In the first two steps we work with a fixed solution $(\u,\rho)$, with $\u=(u,v)$ and $\inf u>0$. We denote as $\ep$ its distance to $(\u_0,\rho_0)$ in $H^s\times W^{1,q}$. In the last step we consider a sequence of solutions, adding a subindex $n$ to the quantities constructed in the previous steps. As usual, $\lesssim$ denotes the corresponding inequality modulo a constant that may depend on the exponents $s,q$ and the background shear $(\u_0,\rho_0)$ but not on $\ep$ or the sequence index $n$.

\vspace{2pt}
\noindent \underline{Step 1}\\[4pt]
In this step we introduce the streamline coordinates and derive estimates for some horizontally-averaged quantities. The crucial point is that averaging ensures that these quantities satisfy favorable uniform bounds, independent of $\inf u$.

Let $\psi$ be the streamfunction of $\u$, which is locally constant on $\partial M$ because $\u$ is assumed to be tangent to the boundary. We normalize its value to 0 at $\{y=-1\}$ and denote by $c_-$ its value at $\{y=1\}$. Since $\partial_y \psi=-u<0$, we see that $c_-<0$. Furthermore, by the implicit function theorem, there exists a function $h(x,c)$ such that \begin{equation}
\psi(x,h(x,c))=c
\label{def height}
\end{equation}
for all $c_- \leq c \leq 0$. We will work on the streamline coordinates $(x,c)\in\T\times[c_-,0]$. For a function $f$ defined on $M$, we denote $f_g(x,c)\coloneqq f(x,h(x,c))$. Differentiating \eqref{def height} and recalling $u=-\partial_y\psi$, $v=\partial_x\psi$ yields
\begin{equation}
\partial_xh=\frac{v_g}{u_g}, \hspace{40pt} \partial_c h=-\frac{1}{u_g}.
\label{derivatives h}
\end{equation}
In addition, we denote the horizontal average as $\avg{f}=(2\pi)^{-1}\int_{\T}f\,{\d x}$. We define the average height and speed as:
\begin{equation}\label{averages}
 H(c)\coloneqq\avg{h}(c),\qquad\qquad
 a(c)\coloneqq\avg{u_g^2+v_g^2}^{1/2}.
\end{equation}

Next, we prove some bounds on these quantities. For the first one, choose extensions $\tilde{\u}, \tilde{u}_0$ of $\u,u_0$ to $\mathbb{T}^2$ such that we still have $\norm{\tilde{\u}-\tilde{u}_0}_{H^s}\lesssim \ep$. Comparing the Fourier weights $(1+\eta^2)^s\leq(1+k^2+\eta^2)^s$, we obtain
$$H^s(M)\hookrightarrow L^2_xH^s_y\hookrightarrow L^2_xL^\infty_y.$$
Here the second embedding is just the 1D Sobolev embedding. Restricting again to $y\in[-1,1]$, we conclude
\begin{equation}\label{prelim1}
 \int_{\T}\sup_y|\u(x,y)-(u_0(y),0)|^2\,{\d x}\lesssim\ep^2.
\end{equation}

We also estimate the difference between the streamfunctions. By subtracting a constant, we may assume that $\psi_0$ also vanishes at $\{y=-1\}$. Therefore, since $\nabla(\psi-\psi_0)=(v,u_0-u)$ and $\psi-\psi_0$ vanishes at $\{y=-1\}$, we have
$$\norm{\psi-\psi_0}_{L^\infty}\lesssim \norm{\psi-\psi_0}_{H^{1+s}}\lesssim \norm{\nabla(\psi-\psi_0)}_{H^s}=\norm{\u-\u_0}_{H^s}\leq \ep.$$
Here we have used the Sobolev embedding and Poincar\'e's inequality. In particular, we see that two points $(x_1,y_1)$ and $(x_2,y_2)$ on the same streamline $\{\psi=c\}$ have $\psi_0$ values differing by at most $C\ep$. Since $u_0\geq m$, we thus have
$$C\ep\geq \abs{\psi_0(y_1)-\psi_0(y_2)}=\abs{\int_{y_1}^{y_2}u_0(t){\d t}}\geq m\abs{y_2-y_1}.$$
Since both points are on the streamline $\{\psi=c\}$, they satisfy $y_i=h(x_i,c)$ for $i=1,2$. Hence,
$$\abs{h(x_1,c)-h(x_2,c)}\lesssim \ep.$$
Therefore, the average in $x$ satisfies
\begin{equation}
\label{difference average height}
\norm{h-H}_{L^\infty}\lesssim\ep.
\end{equation}
Combining this with \eqref{prelim1} leads to
\begin{align*}
\abs{a(c)-u_0(H(c))}
 &\leq\avg{|(u_g,v_g)-(u_0(H(c)),0)|^2}^{1/2}\\&\leq \avg{|(u_g,v_g)-(u_0(h(\cdot,c)),0)|^2}^{1/2}+\avg{|u_0(H(c))-u_0(h(\cdot,c))|^2}^{1/2}
 \lesssim \ep,
\end{align*}
where we have used the reverse triangle inequality for the averaged $L^2$-norm $\langle\,\abs{\hspace{1pt}\cdot\hspace{1pt}}^2\,\rangle^{1/2}$. Thus, if $\ep$ is sufficiently small:
\begin{equation}\label{uniform-a}
 \norm{a(\psi)-u_0}_{L^\infty}\lesssim\ep,\qquad\qquad 
 \frac m2\leq a\leq C.
\end{equation}

\noindent\underline{Step 2}\\[4pt]
In this step we construct a suitable energy functional $\mathcal{K}$, which is shown to satisfy useful estimates that are independent of $\inf u$. We also establish a key identity relating $\mathcal{K}$ to a stratification term. This identity will produce a contradiction with \eqref{eq:coercivity Q_0} along a suitable normalized sequence.

We begin by deriving equations adapted to the streamline coordinates. First, we write the momentum equation as $\u\cdot\nabla(p+|\u|^2/2+y\rho)=0$. Hence, the density and the Bernoulli function are constant on each connected streamline, see Section \ref{sec:steady_reformulation},
$$\rho=\Theta(\psi),\qquad \qquad\quad p+\tfrac12|\u|^2+y\rho=\mathcal B(\psi).$$
Next, using \eqref{def height} and \eqref{derivatives h} we compute
$$u_g\partial_x v_g=(u\partial_xv)_g+(u\partial_y v)_g\,\partial_xh=(u\partial_x v+v\partial_y v)_g=(-\partial_yp-\rho)_g=u_g\partial_cp_g-\Theta(c).$$
Hence, the vertical momentum equation in these coordinates is
$$\partial_xv_g-\partial_cp_g=\Theta(c)\partial_c h$$
Substituting the Bernoulli expression for $p_g$:
$$\partial_c\left[\frac{1}{2}(u_g^2+v_g^2)+h\Theta(c)-\mathcal{B}(c)\right]+\partial_x v_g-\Theta(c)\partial_c h=0.$$
Subtracting the horizontal average gets rid of the unknown $\mathcal{B}$ and will lead to an energy that vanishes for shears:
\begin{equation}\label{height-equation}
 \partial_c\left(\frac{u_g^2+v_g^2-a^2}{2}\right)
 +\partial_xv_g+\Theta'(c)(h-H)=0.
\end{equation}
By definition of $h$ and the values of $\psi$, we have $h(x,0)=-1$ and $h(x,c_-)=1$ for all $x\in\T$. Hence, $h-H$ vanishes at $c=0,c_-$. We multiply the previous equation by $h-H$ and integrate over $\mathbb{T}\times[c_-,0]$. After integration by parts, we obtain:
\begin{equation}\label{height-energy}
 \mathcal{K}-\int_{c_-}^{0}\int_{\T}\Theta'(c)(h-H)^2\,{\d x}\,{\d c}=0,
\end{equation}
where,
\begin{equation}
\mathcal{K}\coloneqq\frac12\int_{c_-}^{0}\int_{\T}\frac{a^2-u_g^2+v_g^2}{u_g}\,{\d x}\,{\d c}
\label{original def K}
\end{equation}
To obtain \eqref{height-energy} we have used \eqref{derivatives h} and the fact that the term involving $H'$ vanishes because $\int_\T (u_g^2+v_g^2-a^2)\hspace{1pt}{\d x}=0$, by definition of $a$. These are the energy functional and the key identity of the proof. From this expression it is not clear that $\mathcal{K}\geq0$ and that it vanishes for shears. To see this, we add the integral of $(u_g^2+v_g^2-a^2)/(2a)$, which is zero, by definition of $a$. Operating then leads to:
\begin{equation}
\mathcal{K}=\frac12\int_{c_-}^{0}\int_{\T}
       \left(\frac1{u_g}+\frac1a\right)
       \bigl((u_g-a)^2+v_g^2\bigr)\,{\d x}\,{\d c}\geq0.
 \label{energy-positive}
\end{equation}
From this expression it is clear that $\mathcal{K}\geq 0$. Furthermore $\mathcal{K}=0$ implies $v\equiv 0$; arguing as in Theorem~\ref{thm rigidity non-zero Linfty} then shows that the solution is a shear.

In their current form, these expressions are not very convenient, since they include the denominator $u_g$, which can be very small. The solution is changing coordinates back to $(x,y)\in M$, which will get rid of the denominator because the absolute Jacobian of $(x,y)\mapsto(x,\psi(x,y))$ is $u$. Doing so for $\mathcal{K}$ gives
\begin{align*}
    \mathcal{K} &= \frac12 \int_M \left(\frac{1}{u(x,y)} + \frac{1}{a(\psi(x,y))}\right) \left( (u(x,y) - a(\psi(x,y)))^2 + v(x,y)^2 \right) u(x,y) \d x \d y \\
    &=\frac12\int_M a(\psi(x,y)) (a(\psi(x,y)) +u(x,y))  \left( \left(\frac{v(x,y)}{a(\psi(x,y))}\right)^2 + \left( 1 - \frac{u(x,y)}{a(\psi(x,y))}\right)^2 \right) \d x \d y \\
    &= \int_M A(x,y) |\nabla w(x,y)|^2 \d x \d y 
\end{align*}
for 
\begin{align*}
    A(x,y) :=\frac{a(\psi(x,y))( a(\psi(x,y)) + u(x,y))}{2}\geq \frac{m^2}{8}, \quad w(x,y):= 1+y + \int_0^{\psi(x,y)}\frac{\d t}{a(t)}
\end{align*}
with further $w(x,-1) = 0$ for all $x\in \T$. Similarly, since $\partial_y\rho = -\Theta'(\psi)u$ we have
\begin{align*}
    -\int_{c_-}^0\int_\T \Theta'(c) (h- H)^2 \d x \d c = \int_M \partial_y\rho (y - H(\psi(x,y)))^2 \d x \d y = \int_M \partial_y\rho \eta(x,y)^2 \d x \d y ,
\end{align*}
for $\eta(x,y):= y - H(\psi(x,y))$. Hence, \eqref{height-energy} now reads
\begin{align}\label{eq:new-height-energy}
    0 = \int_M A |\nabla w|^2  + \int_M \partial_y \rho \eta^2.
\end{align}
Notice that the coefficient that appears in the expression of $\mathcal{K}$ has a uniform lower bound. The denominators in $w$ are also harmless, due to \eqref{uniform-a}.

In the next step we will link these objects to $Q_0$. Since $Q_0$ is defined on $H^1_0$, we will need to control the trace of $w$. By construction, $w$ does vanish on $\{y=-1\}$, but it may be nonzero on $\{y=1\}$. To control its values on the upper wall, we compare it with $\eta$, which vanishes on $\partial M$. Indeed, it follows from the definition of both quantities that the difference has the simple expression
\begin{equation}
\begin{aligned}
w-\eta&=H(\psi)+1+\int_0^\psi a(t)^{-1}\,{\d t}\\
&=\int_0^\psi\bigl(H'(t)+a(t)^{-1}\bigr)\,{\d t}
=\int_\psi^0\bigl(-H'(t)-a(t)^{-1}\bigr)\,{\d t},
\end{aligned}
\label{diff w-eta}
\end{equation}
where we have used that $H(0)=-1$. Differentiating under the integral sign in the definition of $H$ and using \eqref{derivatives h} leads to
$$-H'=\avg{u_g^{-1}}\geq\avg{u_g}^{-1}\geq a^{-1}.$$
In particular, $w-\eta\geq 0$. Furthermore, writing 
\begin{align*}
    \mathcal{K} = \int_{c_-}^0\frac12\int_\T \frac{a^2 - u_g^2 + v_g^2}{u_g} \d x \d c = \int_{c_-}^0 \mathcal{K}(c) \d c
\end{align*}
we have 
\begin{align*}
    \mathcal{K}(c) = \pi \left(-a^2 H'(c) - \langle u_g \rangle + \langle v_g^2 u_g^{-1} \rangle \right) \geq \pi a^2 \left( - H'(c) - a^{-1}(c) \right)
\end{align*}
since $H'=-\langle u_g^{-1}\rangle$ and $\avg{u_g}^2\leq\langle u_g^{2}\rangle\leq a^2$. Combining the previous bound with \eqref{uniform-a} and \eqref{diff w-eta}, we conclude
\begin{equation}\label{boundary}
 0\leq w-\eta\lesssim \mathcal{K},\qquad
 w|_{y=-1}=0,\qquad \norm{w|_{y=1}}_{L^\infty}\lesssim \mathcal{K}.
\end{equation}
Finally, for $\ep$ sufficiently small, it follows from \eqref{difference average height} and \eqref{eq:new-height-energy} that
\begin{equation}
\label{cota K}
0\leq \mathcal{K}\leq \norm{\partial_y \rho}_{L^1}\norm{\eta}_{L^\infty}^2\lesssim \ep^2.
\end{equation}

\noindent\underline{Step 3}\\[4pt]
In this step we complete the proof of the theorem by contradiction. If untrue, we shall use a compactness argument to construct a sequence whose limit $z\in H^1_0(M)$ satisfies $1+\int_M\partial_y\rho_0z^2=0$. This arises as the limit of the identity in \eqref{eq:new-height-energy}. Furthermore, we show that the left-hand side is an upper bound for $Q_0(z)$, reaching a contradiction with \eqref{eq:coercivity Q_0}.

Suppose that the theorem is false. We can then find a sequence of non-shear solutions $(\u_n,\rho_n)$ of the steady Boussinesq equations such that $\inf u_n>0$ and their distance $\ep_n$ to $(\u_0,\rho_0)$ in $H^s\times W^{1,q}$ converges to $0$. We claim no bound on $\inf_n \inf_M u_n$. As mentioned in the previous step, since these solutions are not shears, we have $\mathcal{K}_n>0$. Using the objects defined in the previous steps, we consider the sequence $z_n\coloneqq \mathcal{K}_n^{-1/2}\,w_n$. This normalization ensures that the sequence is bounded in $H^1(M)$. Indeed, from $\mathcal{K}_n = \int_M A_n |\nabla w_n |^2$ and Poincar\'e's inequality using the bottom wall, we deduce
\begin{equation}\label{normalization}
 \int_M A_n|\nabla z_n|^2=1,\qquad \qquad\norm{z_n}_{H^1}\lesssim 1.
\end{equation}
Moreover, \eqref{boundary}--\eqref{cota K} imply
\begin{equation}\label{normalized-error}
 \norm{z_n|_{\partial M}}_{L^\infty}\lesssim \mathcal{K}_n^{1/2}\to0,
 \qquad \qquad
 \left\lVert \mathcal{K}_n^{-1/2}\eta_n-z_n\right\rVert_{L^\infty}\lesssim \mathcal{K}_n^{1/2}\to0.
\end{equation}
After possibly taking a subsequence, for $q'\coloneqq q/(q-1)<\infty$ we have:
\begin{equation}\label{convergence}
 z_n\rightharpoonup z\ \text{in }H^1(M),\qquad
 z_n\to z\ \text{in }L^{2q'}(M),\qquad z\in H^1_0(M).
\end{equation}
Indeed, the first claim follows from the Banach-Alaoglu theorem and the second from compactness of the embedding $H^1(M)\hookrightarrow L^{2q'}(M)$. The zero trace follows from \eqref{normalized-error} and continuity of the trace map $H^1(M)\to L^2(\partial M)$.
Combining this with \eqref{normalized-error}, we also deduce
$$\mathcal{K}_n^{-1/2}\eta_n\to z \qquad \text{in }L^{2q'}.$$
Therefore, dividing the identity in \eqref{eq:new-height-energy} by $\mathcal{K}_n$ and using $\partial_y\rho_{n}\to\rho_0'$ in $L^q$ leads to
\begin{equation}\label{potential-limit}
 1+\int_M\rho_0'z^2=0.
\end{equation}
In particular $z\neq0$, so \eqref{eq:coercivity Q_0} implies $Q_0(z)>0$. In the remainder of the proof we show that $K\coloneqq \int_M u_0^2\abs{\nabla z}^2\leq 1$. Combining this with \eqref{potential-limit} leads to $Q_0(z)\leq 0$, reaching a contradiction. Indeed, by \eqref{uniform-a}, we have that $b_n:=a_n(\psi_n)\to u_0$ uniformly, while
$u_n\to u_0$ in $L^2$. Thus $$A_n=(b_n^2+b_nu_n)/2\to u_0^2\qquad \text{in }L^1.$$
In fact, after taking another subsequence, we have convergence almost everywhere. Next, by Cauchy--Schwarz and \eqref{normalization}, we have
$$\left|\int_M u_0^2\nabla z_n\cdot\nabla z\right|^2\leq\left(\int_M A_n|\nabla z_n|^2\right)\left(\int_M\frac{u_0^4}{A_n}|\nabla z|^2\right)=\int_M\frac{u_0^4}{A_n}|\nabla z|^2.$$
The left-hand side tends to $K^2$ by weak convergence. The right-hand side tends to $K$ by dominated convergence, since
$$\frac{u_0^4}{A_n}\to u_0^2 \hspace{6pt} \text{ a.e.},\qquad \qquad 0\leq\frac{u_0^4}{A_n}|\nabla z|^2 \leq\frac{8\norm{u_0}_{L^\infty}^4}{m^2}|\nabla z|^2\in L^1(M).$$
Hence $K^2\leq K$, so $K\leq1$. Combining this with \eqref{potential-limit} gives $Q_0(z)=K-1\le0$. Since $z\neq 0$, this contradicts \eqref{eq:coercivity Q_0} and completes the proof.
\end{proof}

\section{Laminar bifurcation nearby stagnation points}\label{sec:bifurcation_-zero_01}
In this section we show the existence of non-trivial steady states arbitrarily close to a shear stationary solution containing stagnation points, such as the Couette flow $\u(y)= (y,0)$. 

\begin{theorem}\label{thm:zeroflexibility}
Let $((u_0(y),0),\rho_0(y))$ be a
shear equilibrium with $u_0, \rho_0\in C^{\infty}([-1,1])$ and $u_0(y_0)=0$ for some
$y_0\in[-1,1]$.
For every \textcolor{black}{$k\geq 2$ and} every $1\le q<\infty$,
and every $\varepsilon>0$, there exists a nontrivial
laminar stationary Boussinesq solution $(\u,\rho)$
satisfying
\[\|\u-(u_0,0)\|_{C^{k-1}}+\|\rho-\rho_0\|_{W^{1,q}}<\varepsilon.\]
\end{theorem}
If the shear flow is allowed to vanish, the density need no longer be a function of the streamfunction $\psi_0$. Hence, the functional framework used for the case with no stagnation points does not work. Here we will consider deformations of the level sets of $\psi_0$ and $\rho_0$:

\begin{lemma}\label{lemma formulation h}
Let $\psi_0, \rho_0\in C^{\infty}(\R)$ be two smooth profiles and let $h\in C^k (M)$ be a function vanishing at $\partial M$, with $k\geq 2$. Consider the functions $\psi,\rho \in C^k(M)$ given by 
\begin{equation}
    \psi(x,y)\coloneqq \psi_0(y+h(x,y)), \qquad \rho(x,y)\coloneqq \rho_0(y+h(x,y)).
    \label{eq:deformation_streamlines}
\end{equation}
If $h$ satisfies
\begin{equation}
-\psi_0'(y+h)^2\Delta h-\psi_0'(y+h)\psi''_0(y+h)\left[(\partial_xh)^2+(\partial_yh)^2+2\partial_y h\right]+\rho_0'(y+h)h=0,
\label{eq:sufficient_h}
\end{equation}
then $(\nabla^\perp \psi,\rho)$ is a steady solution of the Boussinesq equations.
\end{lemma}
\begin{proof}
We prove that $(\nabla^\perp \psi,\rho)$ is a steady solution of the Boussinesq equations using Proposition~\ref{Bernoullie_function_2}. We denote $Y(x,y)\coloneqq y+h(x,y)$. Since $\nabla \psi$ and $\nabla \rho$ are both parallel to $\nabla Y$, we see that \eqref{second_firstintegral1} holds. Furthermore, the fact that $h$ vanishes on $\partial M$ implies that $\psi$ is constant on each component of $\partial M$. Regarding \eqref{first_firstintegral1}, we write
\[\Delta\psi\,\nabla\psi+y\nabla\rho=\bigl[\psi_0'(Y)\Delta\psi+\rho_0'(Y)y\bigr]\nabla Y.\]
For this expression to be a gradient, it suffices that
\[\psi_0'(Y)\Delta\psi+\rho_0'(Y)y=F(Y)\]
for some smooth function $F$. We choose $F$ so that the previous condition holds for $h\equiv0$, namely
\[F(y):=\psi_0'(y)\psi_0''(y)+\rho_0'(y)y.\]
Substituting this value of $F$, we obtain the condition
\[\psi_0'(Y)\Delta\psi+\rho_0'(Y)y=\psi_0'(Y)\psi_0''(Y)+\rho_0'(Y)Y.\]
If this equation holds, then $(\nabla^\perp\psi,\rho)$ is a steady solution of the Boussinesq equations, by Proposition~\ref{Bernoullie_function_2}. Substituting our ansatz for $\psi$, we have
\[\Delta\psi=\operatorname{div}\bigl(\psi_0'(Y)\nabla Y\bigr)=\psi_0''(Y)|\nabla Y|^2+\psi_0'(Y)\Delta Y.\]
Finally, using the definition of $Y$ and doing some straightforward computations leads to \eqref{eq:sufficient_h}. All computations involve at most second derivatives of $h$, so they are valid for $h\in C^k(M)$ with $k\ge2$. The resulting velocity $\u\coloneqq\nabla^\perp \psi$ and pressure given by Proposition~\ref{Bernoullie_function_2} belong to $C^{k-1}(M)$, and the equations hold classically.
\end{proof}

Without loss of generality, we will make the following simplifying assumptions:
\begin{lemma}
\label{lemma:reduction}
Let $u_0,\rho_0\in C^{\infty}([-1,1])$ and assume that $u_0$ has at least one zero in $[-1,1]$. Then, after an arbitrarily small perturbation in $C^k$, we may assume that all zeros of $u_0$ are simple and contained in $(-1,1)$. Furthermore, if $y_0$ denotes the right-most zero of $u_0$, after an arbitrarily small perturbation in $W^{1,q}, \:q<\infty$, we may assume that $\rho_0$ satisfies 
\begin{equation}
\label{eq:condition rho0 eigenvalues}
\rho_0'(y)=(k+1)(k+2)u_0'(y_0)^2+u_0(y)^2f(y) \qquad \forall y\in[y_0,1]
\end{equation}
for some smooth function $f$.
\end{lemma}
\begin{proof}
If $u_0$ is constant, it suffices to take $\tilde{u}_0\coloneqq sy$ for a sufficiently small $s>0$. Next, we assume that $u_0$ is not constant. Since $0\in u_0([-1,1])$, the interior of the interval $u_0([-1,1])$ contains values arbitrarily close to zero. By Sard's theorem, we can choose a regular value $c\in u_0([-1,1])$ of $u_0$ with $\abs{c}$ arbitrarily small. By further reducing $\abs{c}$, we may assume that $c\notin\{u_0(-1),u_0(1)\}$. Then $\widetilde u_0\coloneqq u_0-c$ has at least one zero, all its zeros are simple and contained in $(-1,1)$, and can be made arbitrarily close to $u_0$ in any norm by taking $\abs{c}$ sufficiently small.

On the other hand, since $y_0$ is simple and is the only zero of $u_0$ on $[y_0,1]$, we see that \eqref{eq:condition rho0 eigenvalues} essentially amounts to choosing $\rho_0'(y_0)=(k+1)(k+2)u_0'(y_0)^2$ and $\rho_0''(y_0)=0$. This can be achieved with an arbitrarily small perturbation in $L^q$, since $q<\infty$.
\end{proof}
\begin{remark}\label{remark:reduction}
Throughout the rest of this section we will assume $y_0=0$. This is done merely for notational convenience; the arguments are the same, mutatis mutandis. 
\end{remark}

Let us summarize our strategy to prove Theorem~\ref{thm:zeroflexibility}. We work on $M_+\coloneqq \T\times(0,1)$. We are looking for a function $h\in C^k(M_+)$ with $k\geq 2$ whose derivatives up to order $k$ vanish at $\{y=0\}$ and such that
\begin{equation}
-\partial_y^2h-j^2\partial_x^2 h-\frac{u'_0(y+h)}{u_0(y+h)}\left[j^2(\partial_xh)^2+(\partial_yh)^2+2\partial_y h\right]+\frac{\rho_0'(y+h)}{u_0(y+h)^2}h=0.
\label{aux eq h}
\end{equation}
Here $j$ acts as a bifurcation parameter. We will perturb the given profiles $\rho_0$ to a nearby density profile $\rho_a$ that allow us to bifurcate from $h\equiv 0$ for some value $j_0$ of the parameter. We will obtain a solution $h$ with non-trivial $x$-dependence that is arbitrarily close to $h\equiv 0$. We then extend $h$ by zero to the whole channel and consider the $C^k$ function $\tilde{h}(x,y)\coloneqq h(j x,y)$. As before, the actual argument is more complicated because we need to ensure that $j$ is an integer. By Lemma~\ref{lemma formulation h} the streamfunction $\psi_0(y+\tilde{h})$ and density $\rho_a(y+\tilde{h})$ will provide a solution to the steady Boussinesq equations that has non-trivial $x$-dependence and is close to the original solution $(u_0,\rho_0)$.

As in the case without stagnation points, the underlying mechanism is to force $\rho_a'$ to be very negative in a region where $u_0$ is very small. Since here $u_0$ has a zero, we do not need to perturb it to achieve this. In addition, by controlling the value of $\rho_a'(0)$ we can make $h$ as flat as we want at $y=0$, which allows us to treat one portion of the channel and then extend by zero. These are the two reasons why we can use any $C^k$ norm in the theorem. 

\subsection{Study of the vertical operator} The first step to prove Theorem~\ref{thm:zeroflexibility} is to understand how the linearization of the operator in \eqref{aux eq h} acts in the vertical variable. For every integer $m\geq 0$ we define the following spaces, which we endow with the $C^{m,\frac12}$ H\"older norm:
\begin{align*}
Z^m\coloneqq \left\{f\in C^{m,\frac12}([0,1]):\;  f^{(j)}(0)=0 \;\;\,\forall 0\leq j\leq m\right\}, \hspace{30pt} Z^m_0\coloneqq \left\{f\in Z^m:\; f(1)=0\right\}.
\end{align*}
As usual, we need to use H\"older spaces to prevent the failure of elliptic estimates at integer index $C^m$. These H\"older spaces will be truncated to $C^m$ spaces at the end of the proof. The value $\frac12$ is not important: any $0<\beta\leq \frac12$ would suffice. In the next lemma we study how the relevant operator behaves on these spaces:
\begin{lemma}
\label{lemma: Fredholmness T}
Let $u_0,\rho_0\in C^\infty([0,1])$ be as in Lemma~\ref{lemma:reduction} and Remark~\ref{remark:reduction}. Then, for any integer $2\leq m\leq k$, the operator
\[\mathcal{L}_0: Z^m_0\to Z^{m-2}, \quad\phi(y)\mapsto -\phi''(y)-2\frac{u_0'(y)}{u_0(y)}\phi'(y)+\frac{\rho_0'(y)}{u_0(y)^2}\phi(y)\]
is Fredholm and has index 0.
\end{lemma}
The proof is just a 1D version of the proof of Lemma~\ref{fredholmness of 2D operator}, so we omit it. We next introduce some ODE tools that will be useful when studying the spectrum of the operator:
\begin{lemma}
\label{lemma:integral operators}
Let $\alpha>0$ be a real number and $n\geq 2$ an integer. Consider the following operators acting on functions $f:(0,1)\to\R$:
\[(A_\alpha f)(y)\coloneqq y^{-\alpha-1}\int_0^y t^{\alpha} f(t)\,{\d t}, \qquad (B_n f)(y)\coloneqq y^{n-1}\int_y^1 t^{-n} f(t)\,{\d t}.\]
Both of these operators are bounded on $L^2((0,1),y^{2r}{\d y})$ for any $0\leq r<\alpha$.
\end{lemma}
\begin{proof}
We write
\[y^r(A_\alpha f)(y)=\int_0^1y^{-\alpha-1+r}t^{\alpha}\,\ind_{(0,y)}(t)\,f(t)\;{\d t}\equiv \int_0^1K(y,t)\,[t^rf(t)]\,{\d t}.\]
Then, we have
\begin{align*}
\int_0^1K(y,t)\,{\d y}&=t^{\alpha-r}\int_t^1y^{-\alpha-1+r}\,{\d y}=\frac{1}{\alpha-r}(1-t^{\alpha-r})\leq \frac{1}{\alpha-r}, \\ \int_0^1K(y,t)\,{\d t}&=y^{-\alpha-1+r}\int_0^yt^{\alpha-r}\,{\d t}=\frac{1}{\alpha-r+1}.
\end{align*}
By the Schur test, we conclude that
\begin{align*}
    \norm{A_\alpha f}_{L^2((0,1),y^{2r}{\d y})}=\norm{y^rA_\alpha f(y)}_{L^2(0,1)}&\leq [(\alpha-r)(\alpha-r+1)]^{-1/2}\norm{t^rf(t)}_{L^2(0,1)}\\&=[(\alpha-r)(\alpha-r+1)]^{-1/2}\norm{f}_{L^2((0,1),y^{2r}{\d y})}.
\end{align*}
The argument for $B_n$ is completely analogous.
\end{proof}
In the following lemma we study the spectrum and eigenfunctions of the operator $\mathcal{L}_0$:
\begin{lemma}
\label{lemma eigenvalues with stagnation points}
Let $u_0,\rho_0\in C^\infty([0,1])$ be as in Lemma~\ref{lemma:reduction} and Remark~\ref{remark:reduction}. Consider the operator
\begin{equation}
\mathcal{L}_0:D\subset L^2_w(0,1)\to L^2_w(0,1), \quad \phi(y)\mapsto-\phi''(y)-2\frac{u_0'(y)}{u_0(y)}\phi'(y)+\frac{\rho_0'(y)}{u_0(y)^2}\phi(y),
\label{eq:def T L^2_w}
\end{equation}
where the weight is $w=u_0^2$ and
\[D\coloneqq \{\phi\in H^2(\hspace{0.3pt}(0,1),u_0^2dy)\cap H^1_0(\hspace{0.3pt}(0,1),{\d y}): u_0^{-2}\phi\in L^2_w(0,1)\}.\]
Then the spectrum of $\mathcal{L}_0$ consists of an increasing sequence of real eigenvalues $\lambda_1<\lambda_2<\cdots$ that are simple and accumulate only at $+\infty$. The corresponding eigenfunctions belong to $Z^k_0$. Furthermore, the first eigenvalue is given by
\begin{equation}
\lambda_1=\inf\left\{\frac{\int_0^1u_0^2\abs{\phi'}^2+\rho_0'\phi^2}{\int_0^1u_0^2\phi^2}\,:\; \phi\in C^\infty_c(0,1),\; \phi\neq 0\right\}.
\label{eq:Rayleigh quotient T}
\end{equation}
\end{lemma}
\begin{proof}
We see that $\phi'\in L^2(0,1)$ is equivalent to $u_0^{-1}\phi'\in L^2_w(0,1)$. Thus, the operator $\mathcal{L}_0:D\to L^2_w$ is well-defined. Taking into account \eqref{eq:condition rho0 eigenvalues} and the fact that $y=0$ is a simple zero of the smooth function $u_0$, we can write
\begin{equation}
    \mathcal{L}_0\phi(y)=\left[-\phi''(y)-\frac{2}{y}\phi'(y)+\frac{(k+1)(k+2)}{y^2}\phi(y)\right]+\left[g_1(y)\phi'(y)+g_0(y)\frac{\phi(y)}{y}\right]\equiv T_0\phi(y)+K\phi(y)
    \label{eq:decomposition T}
\end{equation}
for some smooth functions $g_0,g_1\in C^\infty([0,1])$. Again, $T_0$ and $K$ are well-defined operators $D\subset L^2_w\to L^2_w$.

Let us study the operator $T_0:D\subset L^2_w\to L^2_w$. First of all, the two linearly independent solutions to the ODE $-\phi''-2y^{-1}\phi'+(k+1)(k+2)y^{-2}\phi=0$ are $\phi_1(y)\coloneqq y^{k+1}$ and $\phi_2(y)\coloneqq y^{-(k+2)}$. Since none of them belong to $D$ due to the boundary conditions, we see that $T_0$ is injective. Regarding surjectivity, let us fix a smooth function $h\in C^\infty([0,1])$. The unique solution to the ODE $-\phi''-2y^{-1}\phi'+(k+1)(k+2)y^{-2}\phi=h$ vanishing at the endpoints of the interval is given by
$$\phi_h(y)\coloneqq \frac{1}{2k+3}[F_h(y)-y^{k+1}F_h(1)],$$
where 
$$F_h(y)\coloneqq y^{-(k+2)}\int_0^yt^{k+3}h(t)\,{\d t}+y^{k+1}\int_y^1t^{-k}h(t)\,{\d t}=y^2\,A_{k+3}[h](y)+y^2\,B_{k}[h](y)$$
and $A_{k+3}$ and $B_{k}$ were defined in Lemma~\ref{lemma:integral operators}. Indeed, direct computation shows:
\begin{align*}
F'_h&=-(k+2)y\,A_{k+3}h+(k+1)y\,B_{k}h, \\ F_h''&=-(2k+3)h+(k+2)(k+3)A_{k+3}h+k(k+1)B_{k}h.
\end{align*}
Hence, we see that $\phi_h$ satisfies the ODE. Since $y^{-1}u_0$ is a bounded smooth function, it follows from Lemma~\ref{lemma:integral operators} that $A_{k+3}$ and $B_{k}$ are bounded on $L^2_w(0,1)$. Hence, we see that $F_h''$, $y^{-1}F_h'$ and $y^{-2}F_h$ are in $L^2_w(0,1)$ and
\[\norm{F_h'}_{L^2(0,1)}\leq C\norm{y^{-1}F'_h}_{L^2_w(0,1)}\leq C\norm{h}_{L^2_w(0,1)}, \hspace{30pt}\norm{F_h}_{H^2_w(0,1)}\leq C\norm{h}_{L^2_w(0,1)}.\]
Here the constant $C$ is allowed to change between inequalities. By the embedding $H^1_0((0,1),{\d y})\hookrightarrow L^\infty(0,1)$, we conclude that the inverse $T_0^{-1}:L^2_w(0,1)\to D\subset L^2_w(0,1)$ exists and satisfies
\begin{equation}
\norm{T_0^{-1}h}_{H^1(0,1)}+\norm{T_0^{-1}h}_{H^2_w(0,1)}\leq C\norm{h}_{L^2_w(0,1)}.
\label{eq:estimates T_0^{-1}}
\end{equation}
This has been checked only for smooth $h$ but holds for any $h\in L^2_w(0,1)$ by density of the smooth functions.

Using the compactness of the embedding $H^2_w(0,1)\hookrightarrow L^2_w(0,1)$, we deduce that $T_0$ has compact resolvent. Furthermore, it follows from \eqref{eq:estimates T_0^{-1}} and the compactness of the embeddings $H^2_w\hookrightarrow H^1_w$ and $H^1\hookrightarrow L^2$ that $KT_0^{-1}:L^2_w\to L^2_w$ is compact, i.e., $K$ is relatively compact with respect to $T_0$. This implies that it is $T_0$-bounded with relative bound 0, that is:
\begin{equation}
\norm{K\psi}_{L^2_w}\leq s\norm{T_0\psi}_{L^2_w}+C_s\norm{\psi}_{L^2_w}
\label{eq:K relatively bound 0}
\end{equation}
for any $s>0$ and some constant $C_s$. 
For later reference, choosing $s<1$ in \eqref{eq:K relatively bound 0} so that we can bound $K\psi$ in terms of $\mathcal{L}_0\psi$ instead of $T_0\psi$ and applying \eqref{eq:estimates T_0^{-1}} to $T_0\psi=\mathcal{L}_0\psi-K\psi$ leads to
\begin{equation}
\norm{\psi}_{H^2_w}+\norm{\psi}_{H^1_0}\leq C'(\norm{\mathcal{L}_0\psi}_{L^2_w}+\norm{\psi}_{L^2_w}).
\label{eq:bounds for complete T}
\end{equation}

Next, we prove that $\mathcal{L}_0$ is self-adjoint. Since it is symmetric, it suffices to prove that $\mathcal{L}_0+\mu\Id$ is surjective for some $\mu\in \R$. Consider the decomposition
$$\mathcal{L}_0+\mu\Id=(\Id+KT_0^{-1}+\mu T_0^{-1})T_0.$$
We have proved that $T_0:D\to L^2_w$ is a bijection. Meanwhile, the operator $L^2_w\to L^2_w$ in parenthesis is the sum of the identity and two compact operators, so it is Fredholm with index 0. Therefore, to prove that $\mathcal{L}_0+\mu\Id$ is surjective, it suffices to prove that it is injective. Fix $\phi\in D$. Since $D\subset H^1_0(0,1)$, we can choose a sequence $\{\psi_n\}_{n=0}^\infty\subset C^\infty_c(0,1)$ converging to $\phi$ in $H^1$. Then, integration by parts yields
\[(\mathcal{L}_0\phi,\psi_n)_{L^2_w}=\int_0^1(-u_0^2\phi''\psi_n-2u_0u_0'\phi'\psi_n+\rho_0'\phi\psi_n)=\int_0^1(u_0^2\phi'\psi_n'+\rho_0'\phi\psi_n).\]
Since $\psi_n\to \phi$ in $H^1$, passing to the limit leads to:
\begin{equation}
(\mathcal{L}_0\phi,\phi)_{L^2_w}=\int_0^1(u_0^2\abs{\phi'}^2+\rho_0'\phi^2)\geq \int_0^1fu_0^2\phi^2\geq -\norm{f}_{L^\infty}\norm{\phi}_{L_w^2}^2.
\label{lower bound L_0}
\end{equation}
Choosing $\mu\coloneqq 1+\norm{f}_{L^\infty}$, we obtain:
\begin{equation}
((\mathcal{L}_0+\mu\Id)\phi,\phi)_{L^2_w}\geq \norm{\phi}_{L^2_w}^2
\label{coercivity mathcal L0}
\end{equation}
We deduce that $\mathcal{L}_0+\mu\Id$ is injective, which shows that $\mathcal{L}_0$ is self-adjoint.

We next show that $\mathcal{L}_0$ has compact resolvent. First of all, \eqref{coercivity mathcal L0} shows that $-\mu$ is in the resolvent set $\rho(\mathcal{L}_0)$, which is thus nonempty. For any $z\in\rho(\mathcal{L}_0)$, we consider the decomposition
\[\mathcal{L}_0-z\Id=(\Id+KT_0^{-1}-zT_0^{-1})T_0.\]
Since $\mathcal{L}_0-z\Id$ must have a bounded inverse, and this holds for $T_0$, we see that the bounded operator $(\Id+KT_0^{-1}-zT_0^{-1})$ must have a bounded inverse. Hence,
\[(\mathcal{L}_0-z\Id)^{-1}=T_0^{-1}(\Id+KT_0^{-1}-zT_0^{-1})^{-1},\]
and compactness of $T_0^{-1}$ implies compactness of $(\mathcal{L}_0-z\Id)^{-1}$.

We conclude that $\mathcal{L}_0$ is a self-adjoint operator with compact resolvent, so its spectrum consists of real eigenvalues that accumulate only at infinity. From \eqref{lower bound L_0} we infer that the spectrum is bounded below, so the eigenvalues form an increasing sequence. Let us prove that the first eigenvalue is given by \eqref{eq:Rayleigh quotient T}. First of all, the first equality in \eqref{lower bound L_0} shows that the infimum that appears in \eqref{eq:Rayleigh quotient T} is at least $\lambda_1$. Hence, it suffices to prove that the infimum is at most $\lambda_1$. Since $D\subset H^1_0(0,1)$, we can choose a sequence $\psi_n\in C^\infty_c(0,1)$ converging to a normalized first eigenfunction $\phi_1$ in $H^1$. Since $u_0^2$ and $\rho_0'$ are bounded, it follows from $\psi_n\to \phi$ in $H^1$ that the numerators converge to $\lambda_1$ and the denominators to $1$. Therefore, the infimum is at most $\lambda_1$.

Let $\phi$ be an eigenfunction associated to an eigenvalue $\lambda$. We will prove that $\phi\in Z^k_0$. To this end, let $\mu\in \R$ be in the resolvent set of $\mathcal{L}_0:D\subset L^2_w\to L^2_w$ and consider the operator $(\mathcal{L}_0-\mu\Id ):Z^2_0\to Z^0$. By Lemma~\ref{lemma: Fredholmness T}, this operator has Fredholm index 0. Since $Z^2_0\subset D$, the fact that $\mu$ is not an eigenvalue implies that the kernel of $\mathcal{L}_0-\mu\Id $ must be trivial. Therefore, the operator has a bounded inverse $(\mathcal{L}_0-\mu\Id )^{-1}:Z^0\to Z^2_0$. Due to the embedding $H^1_0(0,1)\hookrightarrow C^{0,\frac{1}{2}}(0,1)$, we have $\phi\in Z^0$, so we may consider \[\psi\coloneqq (\lambda-\mu)(\mathcal{L}_0-\mu\Id )^{-1}\phi\in Z^2_0.\] By construction of $\psi$ and the fact that $\phi$ is an eigenfunction associated to the eigenvalue $\lambda$, we have: \[(\mathcal{L}_0-\mu\Id )\psi=(\lambda-\mu)\phi=(\mathcal{L}_0-\mu\Id )\phi.\]
Since $\mu$ is not an eigenvalue, $(\mathcal{L}_0-\mu\Id ):D\subset L^2_w\to L^2_w$ must have trivial kernel. Hence, $\phi=\psi$, as expected, and we have $\phi\in Z^2_0$. Iterating this argument yields $\phi\in Z^k_0$.

Finally, we use the standard trick to prove that the eigenvalues are simple: suppose  that $\phi_1$ and $\phi_2$ are two eigenfunctions associated to the same eigenvalue $\lambda$ and consider the Wronskian $W\coloneqq \phi_1\phi_2'-\phi_1'\phi_2$. Since $\phi_1$ and $\phi_2$ are $C^2$ and vanish at $y=1$, we see that $W$ is $C^1$ and vanishes at $y=1$. Furthermore, we have:
\[W'=\phi_1\phi_2''-\phi_1''\phi_2=\phi_1\left[-2\frac{u_0'}{u_0}\phi_2'+\frac{\rho_0'}{u_0^2}\phi_2-\lambda\phi_2\right]-\left[-2\frac{u_0'}{u_0}\phi_1'+\frac{\rho_0'}{u_0^2}\phi_1-\lambda\phi_1\right]\phi_2=-2\frac{u_0'}{u_0}W.\]
Thus, $W$ must be proportional to $u_0^{-2}$ for $y>0$. The fact that $W(1)=0$ but $u_0(1)\neq 0$ forces $W\equiv 0$, which means that $\phi_1$ and $\phi_2$ are linearly dependent, so the eigenvalues are simple.
\end{proof}

Next, we show that we can prescribe the first eigenvalue with an appropriate perturbation of the density:
\begin{lemma}
\label{lemma: perturbation eigenvalues stagnation}
Let $\varepsilon>0$. Let $u_0,\rho_0\in C^\infty([0,1])$ be as in Lemma~\ref{lemma:reduction} and Remark~\ref{remark:reduction}. There exist an integer $j_\ast\geq 1$ depending on $(u_0,\rho_0)$ and functions $\rho_a\in C^\infty([0,1])$ for $a\in (-1,1)$ with the following properties: 
\begin{enumerate}
\renewcommand{\labelenumi}{(\roman{enumi})}
    \item $\norm{\rho_0-\rho_a}_{W^{1,q}(0,1)}<\varepsilon$.
    \item $\rho_a$ satisfies \eqref{eq:condition rho0 eigenvalues} for every $a\in (-1,1)$.
    \item the first eigenvalue $\lambda_a$ of the operator $\mathcal{L}_a$ defined as in \eqref{eq:def T L^2_w} for $(u_0,\rho_a)$ depends continuously on $a$ and is strictly increasing.
    \item There exist intervals $\mathcal{A}\subset (-1,1)$ and $\mathcal{J}\ni -j_\ast^2$ such that for all $\lambda\in \mathcal{J}$ there exists a unique $a_\lambda\in \mathcal{A}$ such that $\lambda_{a_\lambda}=\lambda$. In particular, there exists $a_\ast\in \mathcal{A}$ such that $\lambda_{a_\ast}=-j_\ast^2$.
    \item $0$ is not an eigenvalue of $\mathcal{L}_a$ for any $a\in\mathcal{A}$.
    \item The map $(-1,1)\to Z^k_0$, $a\mapsto \phi_a$ is continuous. Here $\phi_a$ denotes the $L^2_w$-normalized eigenfunction associated to $\lambda_a$ (with the appropriate sign).
\end{enumerate}
\end{lemma}
\begin{remark}
Since $\rho_a$ satisfies \eqref{eq:condition rho0 eigenvalues}, Lemma~\ref{lemma: Fredholmness T} and Lemma~\ref{lemma eigenvalues with stagnation points} also hold for the operator $\mathcal{L}_a$.
\end{remark}
\begin{proof}
Let $\delta\in (0,1)$ be a parameter to be determined later and fix a nonnegative bump function $\chi \in C^\infty_c(0,1)$ with $\norm{\chi}_{L^2(0,1)}=1$. We define
\[\rho_a(y)=\rho_0(y)+\int_0^ya\delta^{-1/2q}\chi(\delta^{-1}t)\,{\d t}.\]
Since $\chi$ vanishes in a neighborhood of $y=0$, we see that $\rho_a$ equals $\rho_0$ in a neighborhood of $y=0$, so \eqref{eq:condition rho0 eigenvalues} still holds. Furthermore, it is easy to check that we can achieve \textit{(i)} by taking $\delta>0$ sufficiently small. 

Let $j_\ast\geq 1$ be the smallest integer such that $\lambda_0^1\geq -j_\ast^2+2$. Such an integer exists because in Lemma~\ref{lemma eigenvalues with stagnation points} we proved that the spectrum of $\mathcal{L}_0$ is bounded below. Fix $C>0$ such that $u_0(y)^2\leq Cy^2$ for all $y\in[0,1]$. We further reduce $\delta>0$ so that
\begin{align*}
    \delta^{-1/2q}\norm{\chi}_{L^3}^3&\geq 1+C\norm{\chi'}_{L^\infty}^2+\norm{\rho_0'}_{L^\infty}, \\ \delta^{-2}&\geq C(j_\ast^2+2).
\end{align*}
Having chosen $\delta>0$ satisfying these conditions, we keep it fixed throughout the rest of the proof. The reasons behind these conditions will be apparent later on.

Concerning property \textit{(iii)}, it follows from \cite[Chapter VIII, Theorem~2.6]{kato_perturbation_1995} that the $n$-th eigenvalue $\lambda^n_a$ of $\mathcal{L}_a$ is differentible in $a$ and its derivative is given by
\begin{equation}
\label{eq: derivative eigenvalues}
\frac{\d}{\d a}\lambda^n_a=\left(\phi^n_a,\, u_0^{-2}\chi_\delta\,\phi^n_a\right)_{L^2_w}=\int_0^1\chi_\delta\abs{\phi^n_a}^2,
\end{equation}
where we have denoted $\chi_\delta\coloneqq \delta^{-1/2q}\chi(\delta^{-1}\,\cdot\,)$ and $\phi^n_a$ is a $L^2_w$-normalized eigenfunction associated to $\lambda^n_a$. Note that these cannot be identically zero on any open set. Indeed, on any interval $[b,1]$ with $b>0$ they solve an ODE with smooth coefficients that also has the trivial solution $\phi\equiv 0$. Hence, at the endpoints of an interval where $\phi^n_a$ is identically zero we would have a contradiction with the uniqueness of solutions for the ODE. Since $\chi_\delta$ is nonnegative and its support has nonempty interior, it then follows from \eqref{eq: derivative eigenvalues} that $\lambda^n_a$ is strictly increasing in $a$.

Regarding \textit{(iv)}, we consider the smooth test-function $\theta(y)\coloneqq \delta^{-3/2}\chi(\delta^{-1}y)$ and compute:
\begin{align*}
&\int_0^1u_0^2\abs{\theta'}^2\leq C\int_0^1 y^2\delta^{-5}\chi'(\delta^{-1}y)^2\,{\d y}\leq C\norm{\chi'}_{L^\infty}^2\delta^{-2}, \\
&\int_0^1\rho_0'\theta^2\leq \norm{\rho_0'}_{L^\infty}\delta^{-2}, \\ &\int_0^1[\delta^{-1/2q}\chi(\delta^{-1}y)]\theta(y)^2\,{\d y} =\delta^{-2-1/2q}\int_0^1\delta^{-1}\chi(\delta^{-1}y)^3\,{\d y}=\delta^{-2-1/2q}\norm{\chi}_{L^3}^3,\\
&\int_0^1u_0^2\theta^2\leq C\int_0^1y^2\delta^{-3}\chi(\delta^{-1}y)^2\,{\d y}\leq C.
\end{align*}
Here we have used the fact that the support of $\chi(\delta^{-1}\,\cdot\,)$ and $\chi'(\delta^{-1}\,\cdot\,)$ is contained in $(0,\delta)$. We conclude:
\[\int_0^1\left(u_0^2\abs{\theta'}^2+\rho_{-1}'\theta^2\right)\leq \delta^{-2}\left(C\norm{\chi'}_{L^\infty}^2+\norm{\rho_0'}_{L^\infty}-\delta^{-1/2q}\norm{\chi}_{L^3}^3\right)\leq -\delta^{-2}.\]
Therefore, by Lemma~\ref{lemma eigenvalues with stagnation points}, we have
\[\lambda^1_{-1}\leq \frac{\int_0^1u_0^2\abs{\theta'}^2+\rho_{-1}'\theta^2}{\int_0^1u_0^2\theta^2}\leq- C^{-1}\delta^{-2}\leq -j_\ast^2-2.\]
We obtain $\lambda^1_{-1}\leq -j_\ast^2-2$ and $\lambda^1_0\geq -j_\ast^2+2$. Therefore, property \textit{(iv)} follows from the continuity and strict monotonicity of the map $a\mapsto \lambda^1_a$ for $\mathcal{J}\coloneqq (-j_\ast^2-1,-j_\ast^2+1)$ and $\mathcal{A}$ sufficiently small. Furthermore, the property still holds if we choose a smaller neighborhood $\mathcal{J}$ of $-j_\ast^2$ if we reduce $\mathcal{A}$ accordingly, and vice versa.

The proof of property \textit{(v)} is completely analogous to the proof of property \textit{(iii)} in Lemma~\ref{lemma: control eigenvalues no stag}. The only difference is that now we are working in the weighted space $L^2_w$, which removes some denominators in the expressions. In particular, here we should take $\rho_{a,b}(y)\coloneqq \rho_a(y)+b\int_0^yg(t)\,{\d t}$. In addition, here \eqref{eq: derivative eigenvalues} replaces \eqref{eq: derivative eigenvalues no stagnation}. Other than these two minor differences, the argument is the same. Once $b_\ast$ is fixed, we omit the subindex $b_\ast$ to simplify the notation. 

Finally, we prove \textit{(vi)}. We drop the index 1 for simplicity. For the time being, $\phi_a$ denotes a normalized eigenfunction associated to the first eigenvalue $\lambda_a$. At the moment, the choice between $+\phi_a$ and $-\phi_a$ is arbitrary; which is the right one will be decided later. Fix $a_\infty\in (-1,1)$.  We claim that for any sequence $\{a_n\}_{n=1}^\infty\subset (-1,1)$ that converges to $a_\infty$, the associated sequence of eigenfunctions contains a subsequence that converges to either $\phi_{a_\infty}$ or $-\phi_{a_\infty}$ in $L^2_w(0,1)$. Indeed, we write
\[\mathcal{L}_0\phi_{a_n}=[\lambda_{a_n}-a_n\delta^{-1/2q}\chi(\delta^{-1}\,\cdot\,)u_0^{-2}-b_\ast gu_0^{-2}]\phi_{a_n}.\]
Note that the term in brackets is a smooth function on $[0,1]$ because $\chi$ and $g$ vanish in a neighborhood of $y=0$ and $\delta>0$ is fixed. Therefore, we see that the right-hand side is bounded on $L^2_w$ uniformly in $n$, due to continuity of $a\mapsto \lambda_a$ and the normalization of $\phi_a$. The claim then follows from \eqref{eq:bounds for complete T}, compactness of the embedding $H^2_w\hookrightarrow L^2_w$ and simplicity of the eigenvalue $\lambda_{a_\infty}$.

Having proved this, we deduce that we must have $(\phi_{a_\infty},\phi_a)_{L^2_w}\neq 0$ for any $a$ in a neighborhood $I_{a_\infty}$ of $a_{\infty}$. Hence, for any $a\in I_{a_\infty}$ we choose the direction ($+\phi_a$ or $-\phi_a$) that makes the previous product positive. After choosing the direction, the property from the previous paragraph ensures that for any sequence $\{a_n\}_{n=1}^\infty\subset I_{a_\infty}$ with $a_n\to a_\infty$, the associated eigenfunctions $\{\phi_{a_n}\}_{n=1}^\infty$ converge to $\phi_{a_\infty}$ in $L^2_w$. This proves continuity of the map $a\mapsto \phi_a$ at $a_\infty$.

If $I_{a_\infty}=(b_1,b_2)$, we can repeat the argument to choose the direction in a neighborhood of both endpoints, thus extending the choice of the direction slightly. This serves to propagate it to the whole interval $(-1,1)$ and then we can prove continuity of $a\mapsto \phi_a$ as before.

Once we know that $a\mapsto \phi_a$ is continuous as a map $(-1,1)\to L^2_w(0,1)$, we can use repeatedly the expression
\[\mathcal{L}_0\phi_{a}=[\lambda_{a}-a\delta^{-1/2q}\chi(\delta^{-1}\,\cdot\,)u_0^{-2}-b_\ast gu_0^{-2}]\phi_{a}\]
to bootstrap regularity. Indeed, the right-hand side is continuous as a map $(-1,1)\to L^2_w(0,1)$. Hence, it follows from \eqref{eq:bounds for complete T} that $a\mapsto \phi_a$ is continuous as a map $(-1,1)\to H^1_0(0,1)$. Since $H^1_0(0,1)\hookrightarrow Z^0$, we can now repeat the argument for $Z^m$ spaces. Indeed, it follows from Lemma~\ref{lemma: Fredholmness T} that for $\mu\in \R$ sufficiently negative so that it is not in the spectrum of $\mathcal{L}_0$, the operator $(\mathcal{L}_0-\mu\Id ):Z^{m}_0\to Z^{m-2}$ has a bounded inverse for $2\leq m\leq k$. Iterating the argument yields the continuity of $a\mapsto \phi_a$ as a map $(-1,1)\to Z^k_0$.
\end{proof}

\subsection{Proof of Theorem 7.1}
We now return to the proof of Theorem~\ref{thm:zeroflexibility}. First, we define the spaces 
\begin{align*}
    X&\coloneqq\left\{h\in C^{k,\frac12}(\overline{M}_+):\;\begin{array}{l}
        h(\xi,1)=\partial^\alpha h(\xi,0)=0 \quad \forall \xi\in\T, \;\forall \alpha \text{ multiindex}, \abs{\alpha}\leq k, \\[4pt] h(\xi,y)=h(-\xi,y) \hspace{32pt}\forall (\xi,y)\in \T\times [0,1]
    \end{array}\hspace{-4pt}\right\}, \\[2pt]
    Y&\coloneqq\left\{f\in C^{k-2,\frac12}(\overline{M}_+):\;\begin{array}{l}
        \partial^\alpha f(\xi,0)=0 \hspace{38pt} \forall \xi\in\T, \;\forall \alpha \text{ multiindex}, \abs{\alpha}\leq k-2, \\[2pt] f(\xi,y)=f(-\xi,y) \hspace{18pt}\forall (\xi,y)\in \T\times [0,1]
    \end{array}\hspace{-4pt}\right\},
\end{align*}
which we endow with the $C^{k,\frac12}$ and $C^{k-2,\frac12}$ norm, respectively. Let $U\subset X$ be the ball of radius $1/2$. We assume $u_0,\rho_0\in C^\infty([0,1])$ are as in Lemma~\ref{lemma:reduction} and Remark~\ref{remark:reduction}. Let $\rho_a$ and $\lambda_a$ be given by Lemma~\ref{lemma: perturbation eigenvalues stagnation}. We define $j(a)\coloneqq \sqrt{-\lambda_a}$ and consider the functional $\mathcal{G}:\R\times U\times \mathcal{A}\to Y$ given by
\[[\mathcal{G}(\mu,h,a)](\xi,y)\coloneqq -(j(a)+\mu)^2\partial_\xi^2h-\partial_y^2h-\frac{u'_0(y+h)}{u_0(y+h)}\left[(j(a)+\mu)^2(\partial_\xi h)^2+(\partial_y h)^2+2\partial_y h\right]+\frac{\rho_a'(y+h)}{u_0(y+h)^2}h.\]
We have to check that $\mathcal{G}$ is well-defined. First of all, writing $u_0(y)=y\hspace{1pt}b(y)$ for some $b\in C^\infty([0,1])$ with $b(0)\neq 0$, we have
$$\frac{u'_0(y+h)}{u_0(y+h)}=\frac{b'(y+h)}{b(y+h)}+\frac{1}{(1+h/y)}\,\frac{1}{y}\hspace{1pt}, \qquad \qquad \frac{\rho_a'(y+h)}{u_0(y+h)^2}h=\frac{\rho_a'(y+h)}{b(y+h)^2}\,
\frac{1}{(1+h/y)^2}\,\frac{h}{y^2}.$$
Since $h\in U$, we have 
$$\frac12\leq1+\frac hy\leq\frac32, \qquad \partial_y(y+h)\geq \frac12, \qquad (y+h)\vert_{y=0}=0, \qquad (y+h)\vert_{y=1}=1.$$
Therefore, it suffices to consider expressions with denominators $y$ and $y^2$. Using the boundary conditions, we obtain the following identities:
\begin{equation}
\label{integral identities quotient}
\begin{aligned}
\frac hy&=\int_0^1\partial_y h(\xi,ty)\,{\d t},
\qquad\qquad\qquad
\frac{\partial_yh}{y}=\int_0^1\partial_y^2h(\xi,ty)\,{\d t},
\\[2pt]
\frac{\partial_\xi h}{y}&=\int_0^1\partial_{y\xi}h(\xi,ty)\,{\d t},
\qquad\qquad\qquad\hspace{3pt}
\frac h{y^2}=\int_0^1(1-t)\partial_{y}^2h(\xi,ty)\,{\d t}.
\end{aligned}
\end{equation}
Hence, we see that the maps $h\mapsto y^{-1}h, \;y^{-1}\partial_yh,\; y^{-1}\partial_\xi h,\; y^{-2}h$ are bounded from $X$ into $C^{k-2,\frac{1}{2}}$. Differentiating under the integral sign shows that their derivatives through order $k-2$ vanish at $y=0$. Since $\mathcal{G}$ also respects parity in $\xi$, we conclude that $\mathcal{G}:\R\times U\times \mathcal{A}\to Y$ is well-defined.

In the remainder of this section we will prove that $\mathcal{G}$ satisfies the conditions of Theorem~\ref{thm:unifCR} and then use it to prove Theorem~\ref{thm:zeroflexibility}. It is clear that $\mathcal{G}(\mu,0,a)=0$ for all $(\mu,a)\in \R\times \mathcal{A}$. Next, for each $a\in \mathcal{A}$ consider the Fr\'echet derivative $L_a\coloneqq \partial_h \mathcal{G}(0,0,a)$ of $\mathcal{G}$ with respect to $h$, which is given by:
    \[[L_a\phi](\xi,y)= -j(a)^2\partial_\xi^2\phi(\xi,y)-\partial_y^2\phi(\xi,y)-2\frac{u_0'(y)}{u_0(y)}\partial_y\phi(\xi,y)+\frac{\rho_a'(y)}{u_0(y)^2}\phi(\xi,y)=[\lambda_a\partial_\xi^2\phi+\mathcal{L}_a\phi](\xi,y).\]
Again, this is well-defined and bounded by \eqref{integral identities quotient}. Our control on the spectrum of $\mathcal{L}_a$ will allow us to control the kernel of the 2D operator $L_a$. In order to control its range, we first show that $L_a:X\to Y$ is Fredholm with index 0. To improve readability, we present this in a separate lemma:
\begin{lemma}\label{fredholmness of 2D operator}
For every fixed $a\in \mathcal{A}$, the operator $L_a:X\to Y$ defined above is Fredholm with index 0.
\end{lemma}

\begin{proof}
First, using \eqref{eq:condition rho0 eigenvalues} we write $L_a=P_{j(a)}+K$, where
\begin{align*}
[P_j\phi](\xi,y)&\coloneqq -j^2\partial_\xi^2\phi(\xi,y)-\partial_y^2\phi(\xi,y)-\frac{2}{y}\partial_y\phi(\xi,y)+\frac{(k+1)(k+2)}{y^2}\phi(\xi,y),\\[2pt]
[K\phi](\xi,y)&\coloneqq g_1(y)\partial_y\phi(\xi,y)+g_0(y)\frac{\phi(\xi,y)}{y}
\end{align*}
for some $g_0,g_1\in C^\infty([0,1])$. Again, each of them is a bounded operator $X\to Y$ by \eqref{integral identities quotient}. Furthermore, it follows from \eqref{integral identities quotient} that $\norm{K\phi}_{C^{k-1,\frac12}}\leq C\norm{\phi}_{C^{k,\frac12}}$. Since the embedding $C^{k-1,\frac12}\hookrightarrow C^{k-2,\frac12}$ is compact, we conclude that $K:X\to Y$ is a compact operator. On the other hand, isomorphisms are Fredholm operators with index 0, and it is a standard fact that adding a compact operator preserves Fredholmness and the index. Therefore, everything reduces to proving that $P_j:X\to Y$ is an isomorphism for every $j>0$. Let us fix $f\in Y$ and set
$$F\coloneqq\norm f_{C^{k-2,\frac{1}{2}}(\overline{M}_+)}.$$
For any multi-index $\alpha$ with $|\alpha|\leq k-2$, the boundary conditions on $f$ and Taylor's theorem lead to
\begin{equation}\label{eq:f-decay}
|D^\alpha f(x,y)|\leq CFy^{k-\frac{3}{2}-|\alpha|}.
\end{equation}

First of all, we show the existence of a weak inverse. For this, it is more convenient to use the substitution $w=y\phi$. Direct differentiation gives, for $y>0$,
$$yP_j\phi=Q_jw,\qquad \qquad Q_j\coloneqq-j^2\partial_x^2-\partial_y^2+(k+1)(k+2)y^{-2}.$$
We consider the associated bilinear form on $H^1_0(M_+)$:
$$a_j(w,v)\coloneqq\int_{M_+}\left(j^2\partial_x w\hspace{1pt}\partial_xv+\partial_y w\hspace{1pt}\partial_yv+(k+1)(k+2)y^{-2}wv\right)\,{\d x}\,{\d y}.$$
Hardy's inequality, applied on each vertical segment and integrated in $x$, shows that $a_j$ is bounded on $H^1_0(M_+)\times H^1_0(M_+)$. Moreover, using Poincar\'e's inequality in the vertical variable, it is easy to see that $a_j$ is coercive. Therefore, the Lax-Milgram theorem gives a unique $w\in H^1_0(M_+)$ such that
$$a_j(w,v)=\int_{M_+} yf\,v\,{\d x}\,{\d y}\qquad\qquad\forall v\in H^1_0(M_+).$$
Furthermore, since $f(-x,y)=f(x,y)$, reflection sends the weak solution to another weak solution. Uniqueness of the wea solution implies $w(-x,y)=w(x,y)$, as desired. Finally, set $\phi\coloneqq w/y$ for $y>0$. Then $P_j\phi=f$ in the sense of distributions on $M_+$. On every strip $\T\times[\delta,1]$ with $\delta>0$, weak elliptic regularity followed by Schauder regularity gives $\phi\in C^{k,\frac{1}{2}}$ locally, including up to the regular boundary $y=1$, where $\phi=0$. We next obtain estimates that remain valid as $y$ approaches zero.

We estimate the size of $\phi$ using $y^{k+\frac{3}{2}}$ as a barrier. First, by \eqref{eq:f-decay}, we have $|yf|\leq CFy^{k-\frac{1}{2}}$. On the other hand,
$$Q_j(y^{k+\frac{3}{2}})=(k+5/4)y^{k-\frac{1}{2}}.$$
Note that the coefficient is positive. We then choose a numerical constant $A>0$ sufficiently large so that $\overline w\coloneqq AFy^{k+\frac{3}{2}}$ satisfies $Q_j\overline w\geq|yf|$. Consider $g\coloneqq(w-\overline w)_+$. At $y=0$ both $w$ and $\overline w$ have zero trace, while at $y=1$ the trace of $w$ is zero and $\overline w\geq0$. Hence, $g\in H^1_0(M_+)$. Although $\overline w$ does not vanish at $y=1$, the expression $a_j(\overline w,g)$ is well-defined and integration by parts yields $a_j(\overline w,g)=\int_{M_+}(Q_j\overline w)g$. The choice of $A$ then leads to:
$$a_j(w-\overline w,g)=\int_{M_+}(yf-Q_j\overline w)g\leq0.$$
The left-hand side equals $a_j(g,g)$ due to the identity $\nabla g=\ind_{\{w>\overline{w}\}}\nabla(w-\overline w)$. Coercivity then implies $g=0$, so $w\leq\overline w$. Applying the same argument to $-w$ proves
\begin{equation}
\label{eq:h-decay}
|w(x,y)|\leq AFy^{k+\frac{3}{2}},\qquad |\phi(x,y)|\leq AFy^{k+\frac{1}{2}}.
\end{equation}
These bounds hold almost everywhere and therefore pointwise for $y>0$, by the interior continuity already obtained.

Next, we estimate its derivatives by rescaling the equations according to the distance to $\{y=0\}$. We fix a point $p=(x_0,r)$ with $0<r<1/2$, so that the periodization of the ball $B_{r/2}(p)$ is contained in $M_+$. We introduce the change of coordinates
$$x=x_0+r\zeta_1,\qquad y=r(1+\zeta_2),$$
and define the following rescaled functions on $B_{1/2}(0)$:
\begin{equation}\label{eq:rescaled-functions}
\phi_r(\zeta)\coloneqq \phi\bigl(x_0+r\zeta_1,r(1+\zeta_2)\bigr),
\qquad f_r(\zeta)\coloneqq r^2f\bigl(x_0+r\zeta_1,r(1+\zeta_2)\bigr).
\end{equation}
Note that there is no amplitude normalization in $\phi_r$. Each differentiation in the original coordinates contributes a factor $r^{-1}$. Substitution into $P_j\phi=f$ and multiplication by $r^2$ give
\begin{equation}\label{eq:rescaled-equation}
-j^2\partial_{\zeta_1}^2\phi_r-\partial_{\zeta_2}^2\phi_r
-\frac2{1+\zeta_2}\partial_{\zeta_2}\phi_r
+\frac{(k+1)(k+2)}{(1+\zeta_2)^2}\phi_r=f_r.
\end{equation}
On $B_{1/2}$ we have $1/2<1+\zeta_2<3/2$, so the coefficients are smooth up to the boundary. Since the principal part of the operator is uniformly elliptic, the interior Schauder estimates~\cite[Chapter~6]{GilbargTrudinger2001} on the two fixed balls yield
$$\norm{\phi_r}_{C^{k,\frac{1}{2}}(B_{1/4})}\leq C\left(\norm{\phi_r}_{L^\infty(B_{1/2})}+\norm{f_r}_{C^{k-2,\frac{1}{2}}(B_{1/2})}\right).$$
Here $C$ depends on $j$ but not on $r$ or $x_0$. However, if $j$ takes values in a compact $J\subset (0,\infty)$, the constant can be chosen to depend on $J$ but not on $j$. Undoing the scaling, this estimate translates into the following bound in the original coordinates:
\begin{align*}
&\sum_{|\alpha|\leq k}r^{|\alpha|}
  \norm{D^\alpha \phi}_{L^\infty(B_{r/4}(p))}
 +r^{k+\frac{1}{2}}[D^k\phi]_{\frac{1}{2};B_{r/4}(p)}\\
&\hspace{40pt}\leq C\Bigg(
 \norm \phi_{L^\infty(B_{r/2}(p))}
 +\sum_{|\alpha|\leq k-2}r^{|\alpha|+2}
    \norm{D^\alpha f}_{L^\infty(B_{r/2}(p))}+r^{k+\frac{1}{2}}[D^{k-2}f]_{\frac{1}{2};B_{r/2}(p)}\Bigg).
\end{align*}
Using the definition of $F$ and substituting \eqref{eq:f-decay}--\eqref{eq:h-decay} leads to the following bound for any point $(x_0,r)\in\T\times(0,1/2)$ and any multi-index $\alpha$ with  $|\alpha|\leq k$:
\begin{equation}\label{eq:local-conclusion}
|D^\alpha \phi(x_0,r)|\leq CFr^{k+\frac{1}{2}-|\alpha|},
\qquad \qquad
[D^k\phi]_{\frac{1}{2};B_{r/4}(p)}\leq CF.
\end{equation}

Meanwhile, away from the boundary the coefficients are smooth and uniformly bounded. The boundary Schauder estimates on the fixed strip $\T\times[1/4,1]$ thus give
\begin{equation}\label{eq:upper-boundary}
\norm \phi_{C^{k,\frac{1}{2}}(\T\times[1/2,1])}
\leq C\Big(\norm \phi_{L^\infty(\T\times[1/4,1])}+\norm f_{C^{k-2,\frac{1}{2}}(\T\times[1/4,1])}\Big)
\leq CF.
\end{equation}
There is no boundary-data term because $\phi(x,1)=0$. The artificial lower boundary is handled by interior estimates. Combining these estimates with \eqref{eq:local-conclusion} yields:
\begin{equation}\label{eq:global-decay}
|D^\alpha \phi(x,y)|\leq CFy^{k+\frac{1}{2}-|\alpha|}
\qquad\qquad\forall(x,y)\in\T\times(0,1],\ |\alpha|\leq k.
\end{equation}
In particular, every derivative through order $k$ extends continuously by zero to $y=0$. Hence $\phi$ extends to a $C^k$ function on $\overline{M}_+$ with the required boundary conditions.

It remains to estimate the global top H\"older seminorm. We fix two points $p=(x,y)$ and $q=(x',y')$ with $0<y\leq y'$, and denote $d=\operatorname{dist}(p,q)$. We first consider the case $y<1/2$, which is split into two subcases depending on the size of $d$. If $d<y/8$, both points lie in the ball $B_{y/4}(p)$, and \eqref{eq:local-conclusion} gives
\[|D^\alpha \phi(p)-D^\alpha \phi(q)|\leq CFd^\frac{1}{2}\qquad\]
for any multi-index with $|\alpha|=k$. If $d\geq y/8$, then $y\leq8d$ and $y'\leq y+d\leq9d$, so \eqref{eq:global-decay} gives the same bound. Meanwhile, when $y\geq 1/2$, both points lie on the strip $\T\times [1/2,1]$, so the desired bound follows from \eqref{eq:upper-boundary}. Therefore, we conclude
\begin{equation}
\norm \phi_{C^{k,\frac{1}{2}}(\overline{M}_+)}\leq CF.
\label{final estimate fredhomlness}
\end{equation}

Finally, we check that the solution is unique. Indeed, if $\phi\in X$ satisfies $P_j\phi=0$, then $w=y\phi$ belongs to $H^1_0(M_+)$ and satisfies $a_j(w,v)=0$ for every $v\in H^1_0(M_+)$. Taking $v=w$ gives $w=0$, by coercivity, and hence $\phi=0$. Thus $P_j$ is bijective and its inverse is bounded, by \eqref{final estimate fredhomlness}. Since $P_{j(a)}$ is an isomorphism and $K$ is compact, we conclude that $L_a=P_{j(a)}+K$ is Fredholm with index 0.
\end{proof}

\begin{proof}[Proof of Theorem~\ref{thm:zeroflexibility}]
We are now ready to study the kernel and image of $L_a$. We define $\Phi_a(\xi,y)\coloneqq \beta \cos(\xi)\hspace{1pt}\phi_a(y)$, where $\phi_a$ is given by Lemma~\ref{lemma: perturbation eigenvalues stagnation} and $\beta>0$ is chosen so that $\Phi_a$ has unit norm in $L^2_w(M_+)$ for $w=u_0^2$. It follows from the definition of $\lambda_a$ and $\phi_a$ that $\Phi_a\in\ker L_a$. To prove that the kernel is spanned by $\Phi_a$, we fix $\phi\in \ker L_a$ and do a Fourier expansion in $\xi$. In what follows, convergence is in the space $L^2_w$:
\[\phi(\xi,y)=\sum_{l=0}^\infty f_l(y)\cos(l\xi)\qquad \Rightarrow \qquad L_a\phi(\xi,y)=\sum_{l=0}^\infty[\mathcal{L}_af_l-l^2\lambda_af_l](y)\,\cos(l\xi).\]
The boundary conditions on $\phi\in X$ imply that $f_l\in Z^k_0$, which is contained in the domain $D$ of $\mathcal{L}_a$. If $\phi\in \ker L_a$, each function $f_l$ must lie in the kernel of $\mathcal{L}_a-l^2\lambda_a$. Since $\lambda_a<0$ is the first eigenvalue, we see that $l^2\lambda_a<\lambda_a$ for $l\geq 2$ cannot be eigenvalues. Hence, we must have $f_l\equiv 0$ for $l\geq 2$. Similarly, since $0$ is not an eigenvalue of $\mathcal{L}_a$, we must have $f_0\equiv 0$. We conclude that $\ker L_a$ and a complement $\mathcal{W}_a$ are given by
\begin{align*}
    \ker L_a &= \vspan \{\Phi_a\}, \\
    \mathcal{W}_a&=\{h\in X: \,(h,\Phi_a)_{L^2_w}=0\}.
\end{align*}
Next, we study the image of $L_a$. For any $\phi\in X$, integration by parts shows that
\[(\Phi_a,L_a\phi)_{L^2_w}=(L_a\Phi_a,\phi)_{L^2_w}=0,\]
which implies that the image of $L_a$ is contained in the orthogonal complement to $\vspan \{\Phi_a\}$. Furthermore, it follows from Lemma~\ref{fredholmness of 2D operator} that $R(L_a)$ is closed and the cokernel $Y/R(L_a)$ is one-dimensional, so it must be spanned by the equivalence class of $\Phi_a$. We conclude that the image $\mathcal{R}_a\coloneqq R(L_a)$ and a complement $\mathcal{Z}_a$ are given by:
\begin{align*}
    \mathcal{R}_a&=\{f\in Y: \,(f,\Phi_a)_{L^2_w}=0\}, \\
    \mathcal{Z}_a &= \vspan\{\Phi_a\}.
\end{align*}

By Lemma~\ref{lemma: perturbation eigenvalues stagnation}, $a\mapsto \phi_a$ is continuous as a map $\mathcal{A}\to Z^k_0$. Hence, $a\mapsto \Phi_a$ is continuous as a map $\mathcal{A}\to X\hookrightarrow Y$. Once we have this characterization of $\mathcal{Z}_a$, $\mathcal{R}_a$ and $\mathcal{W}_a$ and we have continuity of $a\mapsto \Phi_a$, the proof of $(2)$ and $(3)$ in Theorem~\ref{thm:unifCR} is completely analogous to the arguments in Lemma~\ref{lemma:conditionsCR}.


By Theorem~\ref{thm:unifCR} for any compact neighborhood $a_\ast\in \cA'\subset \cA$, there exist open neighbourhoods $0\in U_{\cA'}\subset \cI^\mu\times U$, and $0\in \cI_{\cA'}^s\subset\R$ and continuous functions 
    \begin{align}
        (\kappa(s,a),\varpi(s,a)):\cI_{\cA'}^s\times \cA'\rightarrow\mathcal{U}_{\cA'}
    \end{align}
    with $(\kappa(0,a),\varpi(0,a))=0$ for all $a\in \cA'$ and such that
    \begin{align}
        \G(\kappa(s,a),s\Phi_a + sT_a^X\varpi(s,a),a) =0
    \end{align}
    for all $(s,a)\in \cI_{\cA'}^s\times \cA'$. Moreover, for all $a\in {\cA'}$,
\begin{align}
        \lbrace (s,\mu,w)\in (\cI_{\cA'}^s \setminus \lbrace 0 \rbrace)\times\mathcal{U}_{\cA'} : \G(\mu, s\varphi_a+sT_a^X w, a) = 0,  \rbrace &=  \lbrace (s,\kappa(s,a), \varpi(s,a)) : s\in \cI^s_{\cA'}\setminus \lbrace 0 \rbrace \rbrace.
\end{align}
Recall that $j(a)=\sqrt{-\lambda_a}$, so $j(a_\ast)=j_\ast$ by definition of $a_\ast$. Since $a\mapsto -\lambda_a$ is strictly decreasing in $a$, for $a_1,a_2\in \mathcal{A}'$ such that $a_1<a_\ast<a_2$, we have $j(a_2)<j_\ast<j(a_1)$. Since $\kappa$ is continuous and $\kappa(0,a)=0$, for sufficiently small $s$ we have
\[j(a_2)+\kappa(s,a_2)<j_\ast<j(a_1)+\kappa(s,a_1).\]
Therefore, by the intermediate value theorem, for every $s$ sufficiently small there exists $a(s)\in \mathcal{A}'$ such that $j(a(s))+\kappa(s,a(s))=j_\ast$. Hence, for $s\neq 0$ sufficiently small, $h\coloneqq s\Phi_{a(s)}+sT^X_{a(s)}\varpi(s,a(s))$ satisfies
\[-j_\ast^2\partial_\xi^2h-\partial_y^2h-\frac{u'_0(y+h)}{u_0(y+h)}\left[j_\ast^2(\partial_\xi h)^2+(\partial_y h)^2+2\partial_y h\right]+\frac{\rho'_{a(s)}(y+h)}{u_0(y+h)^2}h=0\]
and has nontrivial $\xi$-dependence because so does $\Phi_a$ and $\varpi(s,a)=o(1)$. Let us consider $\widetilde{h}(x,y)\coloneqq h(j_\ast x, y)$. Since $h\in X$, we see that $\widetilde{h}$ can be extended by zero to a $C^k$ function on the whole channel. Furthermore, we see that it satisfies \eqref{eq:sufficient_h}. It thus follows from Lemma~\ref{lemma formulation h} that $\psi(x,y)\coloneqq \psi_0(y+\widetilde{h}(x,y))$ and $\rho(x,y)\coloneqq \rho_{a(s)}(y+\widetilde{h}(x,y))$ yield a steady solution of the Boussinesq equations $(\nabla^\perp\psi,\rho)$ with nontrivial $x$-dependence. Moreover, given $\varepsilon>0$, by further reducing $s\neq 0$ we have
\[\norm{\psi-\psi_0}_{C^k}+\norm{\rho-\rho_0}_{W^{1,q}}<\varepsilon\]
because $\rho_a$ can be chosen arbitrarily close to $\rho_0$ in $W^{1,q}$ for any $a\in (-1,1)$. Finally, the solution is laminar because the diffeomorphism $F(x,y) \coloneqq (x, y + \tilde{h}(x, y))$ fixes the boundary and has positive Jacobian $1 + \partial_y \tilde{h}$. Hence $\widetilde{X} \coloneqq F^{-1}$ is orientation-preserving and $\psi = \psi_0 \circ \widetilde{X}^{-1}$, as required by Definition~\ref{def_laminar_shear}.
\end{proof}

\section{Non-laminar flexibility and rigidity near stagnant shear flows}\label{sec:local_rigidity-zero_012}
In this section we show that the non-degenerate condition $\min_{y\in[-1,1]} \rho_0'(y) >0$ in Corollary \ref{cor:rigidity_stratified_dens} is sharp. In what follows let $\rho_0\in C^4([-1,1])$ with $\rho'_0(0)=0$. Let also $u_0(y)\in C^4([-1,1])$ be a shear flow such that $u_0(0)=0$, $u_0(y)\neq 0$ for all $y\neq 0$ and $u_0'(0)\neq 0$. We assume without loss of generality that $u_0'(0)=1$. We have the following. 
\begin{theorem}\label{theorem:nonlaminar_flexibility}
    For all $0<\alpha<1$ and all $\ep>0$ there exist non-trivial non-laminar steady Boussinesq pairs $(\u,\rho)$ such that
    \begin{align}
        \Vert \u - (u_0(y),0) \Vert_{C^{1,\alpha}} + \Vert \rho - \rho_0 \Vert_{C^{1,\alpha}} < \ep.
    \end{align}
\end{theorem}

\begin{remark}
    We recall one could use the crude construction presented in \cite[Theorem 1.3]{drivas2024geometric} to obtain non-laminar steady states $(\u, \rho)$ close to $((u_0,0), \rho_0)$ in $C^{0,\alpha}\times C^{1,\alpha}$ by introducing a region of constant density and reducing to the Euler equations. The steady states constructed in Theorem \ref{theorem:nonlaminar_flexibility} also introduce a region of constant density where the system reduces to the Euler equations. There, we bifurcate the shear flow in the spirit of \cite{LinZeng, sinambela2023transition, zhao2026traveling} thus obtaining a non-laminar velocity field.
\end{remark}
The regularity threshold $C^{1,\alpha}$ for $0<\alpha<1$ in Theorem~\ref{theorem:nonlaminar_flexibility} is almost sharp under a spectral assumption. For
\begin{align}
    q_0(y)\coloneqq  \frac{\partial_y\rho_0(y)}{u_0(y)}\in C^{0,\alpha}(M), \quad Q_0(y) \coloneqq u_0''(y) + q_0(y).
\end{align}
we recall the Rayleigh operator
\begin{align}
    \mathcal{R}_0\phi\coloneqq  u_0(y)\Delta\phi(x,y) - Q_0(y)\phi(x,y)
\end{align}
for $\phi \in H^2(M)\cap H_0^1(M)$.

\begin{theorem}\label{thm:nonlaminar_rigidity}
    Assume $((u_0, 0), \rho_0)$ is a stratified shear flow with $u(0)=\rho'(0)=0$ and $u'_0(y)\geq c_0>0$. Assume further that the Rayleigh operator $\mathcal{R}_0$ has trivial kernel in $H^2(M)\cap H^1_0(M)$. Then, for all $0<\alpha<1$ there exists $\ep>0$ such that any stationary configuration $(\u, \rho)$ with
    \begin{align*}
        \Vert \u - (u_0,0) \Vert_{C^{2,\alpha}} + \Vert \rho - \rho_0 \Vert_{C^{2,\alpha}} < \ep
    \end{align*}
    is necessarily a stratified shear flow, namely $\partial_x\u = \partial_x\rho = 0$.
\end{theorem}

\begin{remark}
    In terms of Sobolev spaces, similar ideas to those presented below show that the critical regularity threshold is $H^\frac52$. That is, for all $s>0$ and all $\ep>0$ we can find non-trivial non-laminar steady Boussinesq pairs $(\u, \rho)$ such that
    \begin{align}
        \Vert \u - (u_0(y),0) \Vert_{H^{\frac52-s}} + \Vert \rho - \rho_0 \Vert_{H^{\frac52-s}} < \ep.
    \end{align}
    while there exists $\ep_0>0$ such that all steady solutions with
    \begin{align}
        \Vert u - u_0 \Vert_{H^{\frac52+s}} + \Vert v \Vert_{H^2} + \Vert \rho - \rho_0 \Vert_{H^{\frac52+s}}< \ep_0 
    \end{align}
    must be one-dimensional, $u=u(y)$, $v = 0$ and $\rho=\rho(y)$. 
\end{remark}

\begin{remark}
    If $u_0''(y)u_0(y) + \rho'_0(y) \geq 0$ then $\mathcal{R}_0\phi=0$ has no non-zero solution for $\phi\in H^2(M)\cap H_0^1(M)$. Indeed, 
    \begin{align*}
        \frac{u_0''(y)u_0(y)+\rho'_0(y)}{u_0^2(y)}\phi = \Delta\phi \in L^2(M)
    \end{align*}
    which then yields
    \begin{align*}
        0 = \int_M |\nabla \phi|^2 + \int_M \frac{u_0''(y)u_0(y)+\rho'_0(y)}{u_0^2(y)}|\phi|^2.
    \end{align*}
    The sign condition and $\phi\in H_0^1(M)$ gives $\phi\equiv 0$. On the other hand, there also exist $u_0, \rho'_0$ such that $\mathcal{R}_0$ has non-trivial kernel. Indeed, one may choose $u_0(y)=y$ and $\rho'_0(y) = -\left(k^2+\frac{\pi^2}{4}\right)y^2 \leq 0$, so that $\phi_0=\cos(kx)\cos\left(\frac{\pi y}{2} \right)\in \ker \mathcal{R}_0$. See \cite{enciso2025linear} for a spectral study of $\mathcal{R}_0$ nearby other stratified shear flow configurations.
\end{remark}

To prove Theorem \ref{theorem:nonlaminar_flexibility} we also set up a Crandall-Rabinowitz bifurcation argument in the spirit of \cite{LinZeng, zhao2026traveling}. The proof of Theorem \ref{thm:nonlaminar_rigidity} is postponed to Section \ref{sec:proof-thm-nonlaminar-rigidity}.

\subsection{Functional relations}
We first linearise the shear flow near the origin as follows: let $\chi_\delta = \chi(\cdot/\delta)$, where $\delta>0$ and $\chi\in C^\infty_0$ is a non-negative smooth bump function such that $\chi(y)=1$ for $|y|\leq 1$ and $\chi(y) = 0$ for $|y|\geq 2$. We define
\begin{align}\label{eq:def_u_delta}
    u_\delta(y) = \chi_\delta(y) y + (1-\chi_\delta(y))u_0(y).
\end{align}
It is such that $\Vert u_\delta - u_0 \Vert_{C^{1,\alpha}}< \ep$ for $\delta>0$ small enough. To control the eigenvalue associated to the linearised operator, we further define
\begin{align}\label{eq:def_u_delta_a_gamma}
    u_{\delta,a,\gamma}\coloneqq  u_\delta(y) + a\gamma\int_0^y \Gamma(s/\gamma) \d s,
\end{align}
where $\Gamma(s):=\frac{e^{-s^2}}{2\sqrt{\pi}}\in C^\infty(\R)$ is even, positive and 
\begin{align}
    \int_\R \frac{\Gamma'(s)}{s} \d s = -1.
\end{align}
For $\alpha\in(0,1)$, there holds
\begin{align}
    \Vert u_{\delta,a,\gamma} - u_\delta \Vert_{C^{1,\alpha}} \lesssim_\alpha \gamma^{1-\alpha}
\end{align}
and we define the stream-function
\begin{align}
    \psi_{\delta,a,\gamma}(y) = -\int_0^y u_{\delta,a,\gamma}(z) \d z
\end{align}
and we note that since $yu_{\delta,a,\gamma}(y)>0$ for all $y\neq 0$, $|a|$ bounded and $\gamma>0$ small enough, there exist inverse functions
\begin{align}
    \phi_{\delta,a,\gamma}^+:[\psi_{\delta,a,\gamma}(1),0]\rightarrow [0,1], \qquad \quad\phi_{\delta,a,\gamma}^-:[\psi_{\delta,a,\gamma}(-1),0]\rightarrow [-1,0]
\end{align}
such that $\phi_{\delta,a,\gamma}^+\in C^4([\psi_{\delta,a,\gamma}(1), 0))$ and $\phi_{\delta,a,\gamma}^-\in C^4([\psi_{\delta,a,\gamma}(-1), 0))$, they are the inverse functions of $\psi_{\delta,a,\gamma}$ in the interval $(0,1]$ and $[-1,0)$, where $\psi_{\delta,a,\gamma}$ is monotone. They are uniformly continuous in $a\in \R$. Moreover, since $u_{\delta,a,\gamma}$ is odd in $|y|\leq \delta$ we have $\psi_{\delta,a,\gamma}$ is even in $|y|\leq \delta$, namely there exists some $H=H_{\delta,a,\gamma}$ depending continuously on $a\in\R$ such that $\psi_{\delta,a,\gamma}(y) = H(\frac12 y^2)$ for all $|y|\leq \delta$. Furthermore,
\begin{align}
    H'(\frac12y^2) = \frac{\psi_{\delta,a,\gamma}'(y)}{y} < 0
\end{align}
for all $|y|\leq \delta$ so there exists $H^{-1}\in C^4$ such that $\frac12y^2 = H^{-1}(\psi_{\delta,a,\gamma}(y))$, for all $|y|\leq \delta$. Similarly, since $\frac12|u_{\delta,a,\gamma}(y)|^2$ is even in $|y|\leq \delta$, there exists some function $Q=Q_{\delta,a,\gamma}$ continuous in $a$ such that $\frac12|u_{\delta,a,\gamma}(y)|^2 = Q(\frac12y^2)$ for all $|y|\leq \delta$. In particular, we have
\begin{align}
    \frac12|u_{\delta,a,\gamma}(y)|^2 = (Q\circ H^{-1})(\psi_{\delta,a,\gamma}(y)),
\end{align}
for all $|y|\leq \delta$.

We next define the modified density profile $\tilde\rho_\delta(y) := \chi_\delta(y) \rho_0(0) + (1-\chi_\delta(y))\rho_0(y)$ and
\begin{align}\label{eq:def_rho_delta}
    \rho_{\delta,a}(y) \coloneqq \tilde\rho_\delta(y) - \beta a(a-a_*)\int_0^y \chi_\delta(r-5\delta) \d r
\end{align}
for some $a,a_*\geq 0$ and $\beta > 0$ to be determined later on. We note that $\beta a(a-a_*)\int_0^y \chi_\delta(r-5\delta) \d r$ is non-zero on $(3\delta,1)$, with further
\begin{align*}
    \left\Vert \beta a(a-a_*)\int_0^y \chi_\delta(r-5\delta) \d r \right\Vert_{C^{1,\alpha}}\lesssim \delta^{-\alpha}\beta a |a - a_*|.
\end{align*}
and we note that for all $0<\alpha< 1$,
\begin{align}
    \Vert \rho_{\delta,a} - \rho_0 \Vert_{C^{1,\alpha}} < \ep
\end{align}
for $\delta>0$ small enough, $a_*,\beta$ bounded and $|a-a_*|$ small enough. We next define 
\begin{align}
    \tilde\Theta_{\delta,a,\gamma}^+(s) = (\rho_{\delta,a}\circ \phi_{\delta,a,\gamma}^+)(s)
\end{align}
for $s\in[\psi_{\delta,a,\gamma}(1),0]$. It is such that $\rho_{\delta,a}(y) = \tilde{\Theta}_{\delta,a,\gamma}^+(\psi_{\delta,a,\gamma}(y))$ for all $y\in [0,1]$. For $s\leq 0$ sufficiently close to 0, we see that $|\phi_{\delta,a,\gamma}^+(s)|\leq \delta$ and thus $\tilde\Theta_{\delta,a,\gamma}^+(s) = \rho_0(0)$. Hence, let ${\Theta}_{\delta,a,\gamma}^+$ denote a $C^4$ extension of $\tilde\Theta_{\delta,a,\gamma}^+(s)$ on $\R$ such that $\Theta_{\delta,a,\gamma}^+(s)=\rho_0(0)$ for all $s\geq 0$. Similarly, we define
\begin{align}
    \tilde\Theta_{\delta,a,\gamma}^-(s) = (\rho_{\delta,a}\circ \phi_{\delta,a,\gamma}^-)(s)
\end{align}
for $s\in[\psi_{\delta,a,\gamma}(-1),0]$. As before, it is such that $\rho_{\delta,a}(y) = \tilde{\Theta}_{\delta,a,\gamma}^-(\psi_{\delta,a,\gamma}(y))$, for all $y\in [-1,0]$ and we let ${\Theta}_{\delta,a,\gamma}^-(\cdot)$ denote a $C^4$ extension of $\tilde\Theta_{\delta,a,\gamma}^-(\cdot)$ on $\R$ such that $\Theta_{\delta,a,\gamma}^-(s)=\rho(0)$ for all $s\geq 0$. In particular, we see that $\Theta_{\delta,a,\gamma}^-(s) = \Theta_{\delta,a,\gamma}^+(s)=\rho(0)$ for all $s\geq 0$ and all $s\in(s_0,0)$, where $s_0<0$ is such that $\max \lbrace |\phi_{\delta,a,\gamma}^+(s_0)|, |\phi_{\delta,a,\gamma}^-(s_0)| \rbrace \leq \delta$.

We next set $\eta\in C^\infty$ with $\eta\geq 0$, $\eta'\geq 0$ and such that 
\begin{align}
    \eta(u) = \begin{cases}
        0, & u \leq -\delta/4, \\
        1, & u \geq \delta/4.
    \end{cases}
\end{align}
and we define
\begin{align}
    \Theta_{\delta,a,\gamma}(y,s) \coloneqq  \eta(y) \Theta_{\delta,a,\gamma}^+(s) + (1-\eta(y))\Theta_{\delta,a,\gamma}^-(s).
\end{align}
Now, it holds
\begin{align}
    \rho_{\delta,a}(y) = \Theta_{\delta,a,\gamma}(y,\psi_{\delta,a,\gamma}(y))
\end{align}
for all $y\in [-1,1]$. Moreover,
\begin{align}
    \nabla \Theta_{\delta,a,\gamma}(y,\varphi) = e_2 \eta'(y) \left( \Theta_{\delta,a,\gamma}^+(\varphi) - \Theta_{\delta,a,\gamma}^-(\varphi) \right) + \left( \eta(y) (\Theta_{\delta,a,\gamma}^+)'(\varphi) + (1-\eta(y)) (\Theta_{\delta,a,\gamma}^-)'(\varphi) \right) \nabla \varphi,
\end{align}
for all $\varphi\in H_0^1$. With a slight abuse of notation, we define $\Theta_{\delta,a,\gamma}'(y,s) = \eta(y) (\Theta_{\delta,a,\gamma}^+) '(s) + (1-\eta(y)) (\Theta_{\delta,a,\gamma}^-) '(s)$. In particular, 
\begin{align}
    \nabla \Theta_{\delta,a,\gamma}(y,\varphi) = \Theta_{\delta,a,\gamma}'(y,\varphi) \nabla \varphi,
\end{align}
for all $\varphi\in H_0^1$ such that $\Vert \varphi \Vert_{L^\infty(\T \times (-\delta/4, \delta/4))}\leq |s_0|$ and 
\begin{align}
    \Theta_{\delta,a,\gamma}''(y,s) \coloneqq  \eta(y)(\Theta_{\delta,a,\gamma}^+)''(s) + (1-\eta(y)) (\Theta_{\delta,a,\gamma}^-)''(s)  
\end{align}
denotes the second derivative of $\Theta_{\delta,a,\gamma}(y,s)$ with respect to $s$.


The pressure $p_{\delta,a,\gamma}$ associated to the Boussinesq steady state $(u_{\delta,a,\gamma}(y),\rho_{\delta,a}(y))$ reads
\begin{align}
    p_{\delta,a,\gamma}(y) = p_0 - \int_0^y\rho_{\delta,a}(z) \d z,
\end{align}
for some constant $p_0\in \R$. Therefore, 
\begin{align}
    P(y) + y\rho_{\delta,a}(y) = p_0 + \int_0^y z\rho_{\delta,a}'(z) \d z + \frac12 |u_{\delta,a,\gamma}(y)|^2 
\end{align}
and we define
\begin{align}
    \tilde G_{\delta,a,\gamma}^+(s) \coloneqq  -p_0 - 
    \begin{cases} \int_0^{\phi_{\delta,a,\gamma}^+(s)} z\rho_{\delta,a}'(z) \d z + \frac12|u_{\delta,a,\gamma}(\phi_{\delta,a,\gamma}^+(s))|^2, & s\in [\psi_{\delta,a,\gamma}(1),\psi_{\delta,a,\gamma}(\delta/2)), \\
    (Q\circ H^{-1})(s), & s\in [\psi_{\delta,a,\gamma}(\delta/2), 0]
    \end{cases}
\end{align}
To show that $\tilde{G}_{\delta,a,\gamma}^+(s)$ is smooth across the interface $s=\psi_{\delta,a,\gamma}(\delta/2)$, we note that since $\psi_{\delta,a,\gamma}$ is monotone in $(0,1)$, it is injective there so that for all $s\in (\psi_{\delta,a,\gamma}(3\delta/4), \psi_{\delta,a,\gamma}(\delta/4))$, there exists a unique $y\in (\delta/4, 3\delta/4)$ such that $\phi_{\delta,a,\gamma}^+(s)=y$. 

We now claim that $\phi_{\delta,a,\gamma}^+(s) = \sqrt{2H^{-1}(s)}$ for all $s\in (\psi_{\delta,a,\gamma}(3\delta/4), \psi_{\delta,a,\gamma}(\delta/4))$. Since 
\begin{enumerate}
\setlength{\itemsep}{2pt}
\renewcommand{\labelenumi}{(\roman{enumi})}%
    \item $\psi_{\delta,a,\gamma}:(\delta/4, 3\delta/4)\rightarrow (\psi_{\delta,a,\gamma}(3\delta/4), \psi_{\delta,a,\gamma}(\delta/4))$ is strictly monotone,
    \item $H:[0,\delta^2/2]\rightarrow [\psi_{\delta,a,\gamma}(\delta),0]$ is invertible and
    \item $\phi_{\delta,a,\gamma}^+(s)$ and $ \sqrt{2H^{-1}(s)}$ are positive for all $s\in (\psi_{\delta,a,\gamma}(3\delta/4), \psi_{\delta,a,\gamma}(\delta/4))$,
\end{enumerate}
it is enough to show that $\psi_{\delta,a,\gamma}(\phi_{\delta,a,\gamma}^+(s)) = \psi_{\delta,a,\gamma}(\sqrt{2H^{-1}(s)})$. Indeed,
\begin{align}
    \psi_{\delta,a,\gamma}(\phi_{\delta,a,\gamma}^+(s)) = s = H(\frac12(\sqrt{2H^{-1}(s)})^2) = \psi_{\delta,a,\gamma}(\sqrt{2H^{-1}(s)}).
\end{align}
Hence, the smoothness of $\tilde{G}_{\delta,a,\gamma}^+(s)$ at the interface follows from noting that $\frac12 |u_{\delta,a,\gamma}(y)|^2 = Q(\frac12 y^2)$ and that $\rho'_{\delta,a}(z) = 0$ for all $|z|\leq \delta$. As usual, we denote by $G_{\delta,a,\gamma}^+(s)$ a $C^4$ extension of $\tilde G_{\delta,a,\gamma}^+(s)$ for all $s\in \R$.

Similarly, we define
\begin{align}
    \tilde G_{\delta,a,\gamma}^-(s) \coloneqq  -p_0 - 
    \begin{cases} \int_0^{\phi_{\delta,a,\gamma}^-(s)} z\rho_{\delta,a}'(z) \d z + \frac12|u_{\delta,a,\gamma}(\phi_{\delta,a,\gamma}^-(s))|^2, & s\in [\psi_{\delta,a,\gamma}(-1),\psi_{\delta,a,\gamma}(\delta/2)), \\
    (Q\circ H^{-1})(s), & s\in [\psi_{\delta,a,\gamma}(\delta/2), 0],
    \end{cases}
\end{align}
we note that it is smooth across the interface by the same reason as above and we denote by $G_{\delta,a,\gamma}^-(s)$ the $C^4$ extension of $\tilde{G}_{\delta,a,\gamma}^-(s)$ such that $G_{\delta,a,\gamma}^-(s) = G_{\delta,a,\gamma}^+(s)$ for all $s> \psi_{\delta,a,\gamma}(\delta/2)$. Finally, we define
\begin{align}
    G_{\delta,a,\gamma}(y,s) \coloneqq  \eta(y) G_{\delta,a,\gamma}^+(s) + (1-\eta(y)) G_{\delta,a,\gamma}^-(s)
\end{align}
for which there holds $P_{\delta,a,\gamma}(y) - y \rho_{\delta,a}(y) = -G_{\delta,a,\gamma}(y,\psi_{\delta,a,\gamma}(y))$ for all $y\in[-1,1]$. Moreover,
\begin{align}
    \nabla G_{\delta,a,\gamma}(y,\varphi) = e_2 \eta'(y) \left( G_{\delta,a,\gamma}^+(\varphi) - G_{\delta,a,\gamma}^-(\varphi) \right) +G_{\delta,a,\gamma}'(y,\varphi)\nabla\varphi
\end{align}
for all $\varphi\in H^1$, where we have defined $G_{\delta,a,\gamma}'(y,s) = \eta(y) (G_{\delta,a,\gamma}^+)'(s) + (1-\eta(y)) (G_{\delta,a,\gamma}^-)'(s)$. As before, there holds $\nabla G_{\delta,a,\gamma}(y,\varphi) = G_{\delta,a,\gamma}'(y,\varphi)\nabla\varphi$ for all $\varphi\in H^1$ such that $\Vert \varphi \Vert_{L^\infty(\T\times (-\delta/4, \delta/4))}\leq |\psi_{\delta,a,\gamma}(\delta/2)|$ and we further define
\begin{align}
    G_{\delta,a,\gamma}''(y,s) \coloneqq  \eta(y) (G_{\delta,a,\gamma}^+)''(s) + (1-\eta(y)) (G_{\delta,a,\gamma}^-)''(s). 
\end{align}


Fixed $\delta>0$ small enough, for the change of variables $\xi=kx$ we define the functional
\begin{align}
    \mathcal{F}(k, \varphi,a) = k^2\partial_\xi^2 \varphi + \partial_y^2 \varphi + y\Theta'_{\delta,a,\gamma}(y,\varphi) + G'_{\delta,a,\gamma}(y,\varphi)
\end{align}
and we aim to solve the functional equation
\begin{align}
    \sF(k,\varphi,a)=0, \quad \sF(k,\varphi,a) \coloneqq  \mathcal{F}(k, \psi_{\delta,a,\gamma}+\varphi,a) - \mathcal{F}(k,\psi_{\delta,a,\gamma},a).
\end{align}
Indeed, for any $\varphi\in H^4$ with $\Vert \varphi\Vert_{H^4}\leq \epsilon$ such that $\sF(k,\varphi,a)=0$, the triplet
\begin{align}
    \tilde\psi\coloneqq \psi_{\delta,a,\gamma}+\varphi, \quad \tilde\rho \coloneqq  \Theta_{\delta,a,\gamma}(y, \tilde\psi), \quad \tilde P \coloneqq  -y \tilde\rho - G_{\delta,a,\gamma}(y, \tilde\psi)
\end{align}
is a stationary solution of the Boussinesq equations because $\mathcal{F}(k,\psi_{\delta,a,\gamma},a)=0$. 

\subsection{Linearised operator and spectral properties}
To apply a Crandall-Rabinowitz bifurcation argument, we consider the linearised operator
\begin{align}
    \partial_\varphi \sF(k,0,a)\varphi = k^2\partial_\xi^2 \varphi - \cL_{\delta,a,\gamma}\varphi,  
\end{align}
with 
\begin{align}
    \cL_{\delta,a,\gamma}\varphi\coloneqq  - \partial_y^2\varphi - y \Theta''_{\delta,a,\gamma}(y,\psi_{\delta,a,\gamma})\varphi - G''_{\delta,a,\gamma}(y,\psi_{\delta,a,\gamma})\varphi
\end{align}
and where
\begin{align}
    - y \Theta''_{\delta,a,\gamma}(y,\psi_{\delta,a,\gamma}) - G''_{\delta,a,\gamma}(y,\psi_{\delta,a,\gamma}) =  \frac{u''_{\delta,a,\gamma}(y)}{u_{\delta,a,\gamma}(y)} + \frac{\rho'_{\delta,a}(y)}{u_{\delta,a,\gamma}^2(y)}.
\end{align}
Hence, we obtain the more suitable formulation
\begin{align}
    \cL_{\delta,a,\gamma}\varphi = -\partial_y^2 \varphi + \left( \frac{u''_{\delta,a,\gamma}(y)}{u_{\delta,a,\gamma}(y)} + \frac{\rho'_{\delta,a}(y)}{u_{\delta,a,\gamma}^2(y)} \right) \varphi
\end{align}
where we note that $\frac{\rho'_{\delta,a}(y)}{u_{\delta,a,\gamma}^2(y)}\in C([-1,1])$ since $\rho'_{\delta,a}(y)=0$ for all $|y|\leq \delta$. The first eigenvalue $\lambda_{a,\gamma}$ of $\cL_{\delta,a,\gamma}$ is characterized by the Rayleigh quotient,
\begin{align}\label{eq:Ray_quot_zero_vel}
    \lambda_{a,\gamma} = \inf_{\varphi\in H_0^1(-1,1),\, \varphi\neq 0}\frac{\int_{-1}^1 |\varphi'(y)|^2 \d y + \int_{-1}^{1} \left( \frac{u''_{\delta,a,\gamma}(y)}{u_{\delta,a,\gamma}(y)} + \frac{\rho'_{\delta,a}(y)}{u_{\delta,a,\gamma}^2(y)} \right) |\varphi(y)|^2 \d y }{\int_{-1}^1 |\varphi(y)|^2 \d y}.
\end{align}

Following the ideas of Lemmas \ref{lemma:properties_lambda_ep_a}, for all $\gamma>0$ fixed the minimal eigenvalue $\lambda_{a,\gamma}$ exists for all $a\geq 0$, the associated eigenfunction $\varphi_{a,\gamma}\in H_0^1(-1,1)$ is not identically zero on open subsets of $(-1,1)$ and the map $a\mapsto (\lambda_{a,\gamma}, \varphi_{a,\gamma})$ is continuous in the product space $\R\times [H^3(-1,1)\cap H^1_0(-1,1)]$.

We next observe that for $a=0$ the linearised operator $\cL_{0,\gamma}$ is independent of $\gamma>0$ and thus its minimal eigenvalue $\lambda_0\coloneqq  \lambda_{0,\gamma}$ is independent of $\gamma>0$ as well. As before, let $j'=\min \lbrace i\in \N : -i^2 \leq \lambda_0 \rbrace$ and let $j=j'+1$ so that $-j^2+\frac12 < \min \lbrace 0,\lambda_0\rbrace$. 

\begin{lemma}\label{lemma:control_eigenvalues_stag}
There exists $\delta_\ast\in(0,\delta_0)$, $\gamma>0$ small enough and intervals $\mathrm{J}\ni -j^2$ and  $\tilde\cA\subset \R$ such that:
\begin{enumerate}
\renewcommand{\labelenumi}{(\roman{enumi})}
    \item $a\mapsto \lambda_{a,\gamma}$ is strictly decreasing for all $a\in {\tilde\cA}$.
    \item For any $\lambda\in \mathrm{J}$ there exists $a\in {\tilde\cA}$ such that $\lambda_{a,\gamma}=\lambda$. In particular, there exists $a_\ast\in{\tilde\cA}$ such that $\lambda_{a_\ast,\gamma}=-j^2$.
    \item $0$ is not an eigenvalue of $\mathcal{L}_{a,\gamma}$ for any $a\in {\tilde\cA}$.
\end{enumerate}
\end{lemma}




\begin{proof}
    We first show there exists some $a_*>0$ such that $\lambda_{a_\ast,\gamma}=-j^2$. To do that, let $\varphi=1-|y|\in H_0^1((-1,1))$, for which $\varphi(0)=1$, $\Vert \varphi\Vert_{L^2}^2=\frac23$ and $\Vert \varphi' \Vert_{L^2}^2=2$. There holds
\begin{align}
    \lim_{\gamma\rightarrow 0} \int_{-1}^1 \frac{\rho'_{\delta,a}(y)}{u_{\delta,a,\gamma}^2(y)} |\varphi(y)|^2 \d y &= \int_{-1}^1 \frac{\tilde\rho'_\delta(y)}{u_{\delta}^2(y)}|\varphi(y)|^2 \d y - \beta a (a-a_*)\int_{-1}^1 \frac{\chi_\delta(y-5\delta)}{u^2_{\delta}(y)}|\varphi(y)|^2 \d y \\
    &=:C_{1,\delta} - \beta a (a-a_*) C_{3,\delta} < \infty
\end{align}
for all $a\in \R$. Moreover, the dominated convergence theorem yields
\begin{align}
    \lim_{\gamma\rightarrow 0}\int_{-1}^1 \frac{u''_{\delta,a,\gamma}(y)}{u_{\delta,a,\gamma}(y)} |\varphi(y)|^2 \d y &= \int_{-1}^1 \frac{u''_{\delta}(y)}{u_{\delta}(y)}|\varphi(y)|^2 \d y + a \lim_{\gamma\rightarrow 0} \int_{-\frac1\gamma}^{\frac1\gamma} \frac{\Gamma'(r)}{r}\frac{|\varphi(\gamma r)|^2}{\frac{u_\delta(\gamma r)}{\gamma r} +  \gamma \frac{a}{\gamma r}\int_0^{\gamma r}\Gamma(s/\gamma) \d s} \d r \\
    &=C_{2,\delta} - a.
\end{align}


For $\beta=0$ and $a_M:=2 + C_{1,\delta} + C_{2,\delta} + \frac{4}{3}(j^2 + \frac12)$ we note that 
\begin{align*}
    \lim_{\gamma\rightarrow 0} \frac{\langle \cL_{\delta,a_M,\gamma}\varphi, \varphi \rangle}{\Vert \varphi \Vert_{L^2}^2} = \frac{2 + C_{1,\delta} + C_{2,\delta} - a_M }{2/3} = -2(j^2 +\frac12)<0
\end{align*}
In particular, we have that
\begin{align*}
    \lambda_{a_M,\gamma} \leq -j^2 - \frac12
\end{align*}
for $\gamma>0$ small enough fixed, and since $\lambda_{0,\gamma}>-j^2 + \frac12$, by continuity of $a\mapsto \lambda_{a,\gamma}$ there exists some $a_0\in (0, a_M)$ such that $\lambda_{a_0,\gamma} = -j^2$. 

Set $a_*:=a_0$ and observe that for $\beta>0$, we still have $\lambda_{a_*,\gamma} = -j^2$. We next show that $a\mapsto \lambda_{a,\gamma}$ is strictly decreasing in a neighbourhood ${\tilde\cA}$ of $a_0$ for $\beta>0$ large enough. Since the first Dirichlet eigenvalue is simple, we have that
\begin{align*}
    \frac{\d}{\d a}\lambda_{a_*,\gamma} = \int_{-1}^1 \left( \frac{\d}{\d a} \left( \frac{u''_{\delta,a_*,\gamma}}{u_{\delta,a_*,\gamma}} + \frac{\tilde\rho'_\delta(y)}{u^2_{\delta,a_*,\gamma}}\right)  - \beta a_*\frac{\chi_{\delta}(y-5\delta)}{u^2_{\delta,a_*,\gamma}(y)} \right) |\varphi_{a_*,\gamma}|^2 \d y, 
\end{align*}
where $\varphi_{a_*,\gamma}$ denotes the $L^2$ normalized eigenfunction associated to $\lambda_{a_*,\gamma}$. In view of \eqref{eq:def_u_delta_a_gamma}, we have
\begin{align*}
    \partial_a u_{\delta,a,\gamma} = \gamma\int_0^y \Gamma(s/\gamma) \d s, \quad \partial_a u''_{\delta,a,\gamma} = \Gamma'(y/\gamma)
\end{align*}
so that a straightforward computation yields
\begin{align*}
    \partial_a \left( \frac{u''_{\delta,a,\gamma}}{u_{\delta,a,\gamma}} + \frac{\tilde\rho'_\delta}{u^2_{\delta,a,\gamma}}\right) = \frac{u_\delta(y) \Gamma'(y/\gamma)}{u^2_{\delta,a,\gamma}(y)} - u''_\delta(y)\frac{\gamma\int_0^y \Gamma(s/\gamma) \d s}{u^2_{\delta,a,\gamma}(y)} - 2\tilde\rho_\delta'(y)\frac{\gamma\int_0^y \Gamma(s/\gamma) \d s}{u^3_{\delta,a,\gamma}(y)}.
\end{align*}
 At $a=a_*$ and for $\gamma>0$ fixed, we have $\left\Vert \partial_a \left( \frac{u''_{\delta,a,\gamma}}{u_{\delta,a,\gamma}} + \frac{\tilde\rho'_\delta}{u^2_{\delta,a,\gamma}}\right)\right\Vert_{L^\infty}\leq C_{\delta,a_*,\gamma} <\infty$ and thus
\begin{align*}
     \frac{\d}{\d a}\lambda_{a_*,\gamma} \leq -\beta a_* \int_{-1}^1 \frac{\chi_\delta(y-5\delta)}{u^2_{\delta,a_*,\gamma}(y)}|\varphi_{a_*,\gamma}|^2 + C_{\delta,a_*,\gamma} < 0
\end{align*}
for $\beta > \beta_0 > 0$ sufficiently large. Hence, $\lambda_{a,\gamma}$ is strictly monotone decreasing in a small neighbourhood ${\tilde\cA}$ of $a_*$ and we define $\mathrm{J} := \lbrace \lambda_{a,\gamma} : a\in {\tilde\cA} \rbrace$. 


\paragraph{\textbf{Proof of (iii)} }
If $\lambda_{a_*,\gamma}^{(n)}\neq 0$, for all $n\geq 2$, since the eigenvalues form an increasing sequence and $a\mapsto\lambda_{a,\gamma}$ is continuous, we shrink ${\tilde\cA}\ni a_*$ if necessary to ensure  $\lambda_{a,\gamma}^{(n)}\neq 0$, for all $n\geq 2$ and all $a\in {\tilde\cA}$. Assume next that we have $\lambda_{a_*,\gamma}^{(n)}=0$, for some $n\geq 2$. We introduce $b_1,b_2\in \R$ and $g_1,g_2\in C^\infty_0(-1,1)$ with $\text{supp } g_i \subset (3\delta,6\delta)$ for $i=1,2$ and such that
\begin{align*}
\begin{array}{cc}
    \int_{-1}^1 g_1(y) |\varphi_{a_*,\gamma}(y)|^2 \d y = 1, & \int_{-1}^1 g_1(y) |\varphi_{a_*,\gamma}^{(n)}(y)|^2 \d y = 0, \\
     \int_{-1}^1 g_2(y) |\varphi_{a_*,\gamma}(y)|^2 \d y = 0, & \int_{-1}^1 g_2(y) |\varphi_{a_*,\gamma}^{(n)}(y)|^2 \d y = 1, 
\end{array}
\end{align*}
where $\varphi_{a_*,\gamma}$ and $\varphi_{a_*,\gamma}^{(n)}$ denote the normalized eigenfunctions associated to $\lambda_{a_*,\gamma}$ and $\lambda_{a_*,\gamma}^{(n)}$ respectively. We redefine the density profile as
\begin{align*}
    \rho_{\delta,a,\gamma,b_1,b_2}(y) = \tilde\rho_\delta(y) -\beta a (a-a_*)\int_0^y\chi_\delta(r-5\delta)\d r + \sum_{i=1}^2 b_i \int_0^y g_i(r) u^2_{\delta,a,\gamma}(r) \d r
\end{align*}
which is constant in $(-\delta,\delta)$ so that the functional relation is still well-defined and we introduce the linearised operator
\begin{align}
    \cL_{\delta,a,\gamma,b_1,b_2}\varphi := -\partial_y^2\varphi + \left( \frac{u''_{\delta,a,\gamma}(y)}{u_{\delta,a,\gamma}(y)} + \frac{\rho'_{\delta,a,\gamma,b_1,b_2}(y)}{u_{\delta,a,\gamma}^2(y)} \right) \varphi.
\end{align}
Clearly, $\cL_{\delta,a,\gamma} = \cL_{\delta,a,\gamma,0,0}$ so that $\lambda_{a_*,\gamma,0,0}=-j^2$ and $\mu_{a_*,\gamma,0,0}=0$ are eigenvalues of $\cL_{\delta,a,\gamma,0,0}$. Since
\begin{align}
    \frac{\d}{\d b_1}\lambda_{a_*,\gamma, 0,0} = \int_{-1}^1 g_1(y) |\varphi_{a_*,\gamma}(y)|^2 \d y = 1, \quad \frac{\d}{\d b_2}\lambda_{a_*,\gamma, 0,0} = \int_{-1}^1 g_2(y) |\varphi_{a_*,\gamma}(y)|^2 \d y = 0,
\end{align}
the implicit function theorem shows there exists a $C^1$ function $h:(-\nu,\nu)\rightarrow \R$ with $h(0)=0$ and $h'(0)=0$ such that $\lambda_{a_*,\gamma,h(b_2),b_2}=-j^2$ for all $b_2\in (-\nu,\nu)$. On the other hand,
\begin{align*}
    \left.\frac{\d}{\d b_2}\mu_{a_*,\gamma, h(b_2),b_2}\right|_{b_2=0} &= \frac{\d}{\d b_1}\mu_{a_*,\gamma, 0,0}h'(0) + \frac{\partial}{\partial b_2}\mu_{a_*,\gamma, 0,0} \\
    &=  \int_{-1}^1 g_2(y) |\varphi_{a_*,\gamma}^{(n)}(y)|^2 \d y = 1,
\end{align*}
from which we deduce that $\mu_{a_*,\gamma,h(b_2),b_2}>0$ for all $b_2\in (0,\nu_0)$ for some $0<\nu_0<\nu$. Fixing $b_2\in (0,\nu_0)$ small enough, we have that 
\begin{align*}
    \lambda_{a_*,\gamma,h(b_2),b_2}=-j^2, \quad \mu_{a_*,\gamma, h(b_2),b_2} >0
\end{align*}
and no other eigenvalue of $\cL_{\delta,a_*,\gamma,h(b_2),b_2}$ is zero since all other eigenvalues of $\cL_{\delta,a_*,\gamma,0,0}$ were a positive distance away from zero and they are continuous with respect to $b_1,b_2$. Finally, properties \textit{(i)} and \textit{(ii)} still hold if $b_2$ is chosen small enough. Indeed, it suffices to check that $\frac{\d}{\d a}\lambda_{a,\gamma,h(b_2),b_2}\big|_{a=a_\ast}<0$ for small $b_2$. The proof is the same as for \textit{(ii)}; we get extra terms than can be bounded in an analogous manner. The lemma then follows.
\end{proof}

Throughout the rest of the section, we assume that $0$ is not an eigenvalue of $\cL_{a,\gamma}$ for any $a\in {\tilde\cA}$, omitting the notation introduced in the proof of the above lemma.

\subsection{Proof of Theorem \ref{theorem:nonlaminar_flexibility}}
We consider the Banach spaces $X$ and $Y$ defined in \eqref{eq:defX} and \eqref{eq:defY} respectively. For $a\in \tilde\cA$ the function $j_\ast(a) \coloneqq  \sqrt{-\lambda_{a,\gamma}}$ is well-defined, strictly monotone and continuous for all $a\in \tilde\cA$. Let $\cA := \tilde\cA -a_*$ an open interval centered at $0$ and define the functional $\G(\mu,\varphi,a):\R\times X\times \cA\rightarrow Y$ given by
\begin{align}\label{eq:def_G_delta_a_gamma}
    \G(\mu,\varphi,a)&\coloneqq  (j_\ast(a_\ast + a)+\mu)^2\partial_\xi^2\varphi + \partial_y^2(\psi_{\delta,a_\ast + a,\gamma}+\varphi) \\
    &\quad + y \Theta_{\delta,a_\ast + a,\gamma}'(y,\psi_{\delta,a_\ast + a,\gamma}+\varphi) + G_{\delta,a_\ast + a,\gamma}'(y,\psi_{\delta,a_\ast + a,\gamma}+\varphi).
\end{align}

\begin{lemma}\label{lemma:conditions_CR_monotone}
    The functional $\G(\mu,\varphi,a)$ satisfies the conditions of the Crandall-Rabinowitz bifurcation Theorem \ref{thm:unifCR}.
\end{lemma}

\begin{proof}
    The Fr\'{e}chet partial derivatives of $\G(\mu,\varphi,a)$ are also given by \eqref{eq:Frechet_G}. Due to the smoothness of $\Theta_{\delta,a,\gamma}'(y, s)$ and $\G_{\delta,a,\gamma}'(y,s)$ with respect to both $y\in(-1,1)$ and $a\in \tilde\cA$, it is immediate to see that $\G, \partial_\mu \G, \in \mathcal{C}(\R\times X \times \cA; Y)$ and $\partial_\varphi\G, \partial_{\mu\varphi}\G\in \mathcal{C}(\R\times X \times \cA; \mathcal{L}(X,Y))$. Finally, in view of Lemma \ref{lemma:conditionsCR}, conditions $(1)-(3)$ follow mutatis mutandis.
\end{proof}

With Lemma \ref{lemma:conditions_CR_monotone} at hand, the proof of Theorem \ref{theorem:nonlaminar_flexibility} follows analogously to that of Theorem \ref{thm:nonzeroflexibility}. In particular, for $s>0$ small enough, $\u=\nabla^\perp\psi$ develops islands since 
\begin{align}
    \psi = \psi_{\delta,a_*+a(s),\gamma}(y) + s\cos(j_*x)\varphi_{a(s)}(y) + sT_{a(s)}^X\varpi(s,a(s))
\end{align}
holds in $H^4(M)\hookrightarrow C^{2,\alpha}(M)$ for all $0<\alpha<1$ and there holds $\psi_{\delta,a_*+a(s),\gamma}'(0)=0$, $\psi_{\delta,a_*+a(s),\gamma}''(0)\neq 0$ with $\varphi_{a(s)}(y)\neq 0$ and $\varpi(s,\cdot) = o(1)$.

\begin{remark}\label{rmk:extension_shear}
    The conclusion of Theorem \ref{theorem:nonlaminar_flexibility} holds for a broader class of shear flows and density stratifications. More precisely, the theorem holds for all smooth pairs $(u(y),\rho(y))$ such that for
    \begin{align}
        \mathcal{Z}_u \coloneqq  \lbrace z\in (-1,1) : u(z) = 0 \rbrace
    \end{align}
    non-empty, there holds $u'(z)\neq 0$ and $\rho'(z)=0$ for all $z\in \mathcal{Z}_u$. Indeed, linearizing $u(y)$ as in \eqref{eq:def_u_delta} for all stagnation points of $u(y)$, choosing $y_0\in \mathcal{Z}_u$, defining
    \begin{align*}
        u_{\delta,a,\gamma}=u_\delta(y) + a\gamma\chi_\delta(y-y_0) \int_{y_0}^y \Gamma(s/\gamma) \d s
    \end{align*}
    and flattening $\rho(y)$ as in \eqref{eq:def_rho_delta} also for all stagnation points of $u(y)$, one can define the functional relations $\Theta_{\delta,a,\gamma}$ and $G_{\delta,a,\gamma}$ using the sided inverse functions $\phi^\pm_{\delta,a,\gamma}(z)$ on each side of the stagnation point. By uniqueness of the implicit and inverse function theorem, any overlaps on regions with non-zero velocity between inverse functions from two distinct stagnation points are well defined. Hence, the functional $\G(\mu,\varphi,a)$ given by \eqref{eq:def_G_delta_a_gamma} is well-defined and satisfies the conditions of Theorem \ref{thm:unifCR}, thus yielding stationary non-trivial solutions to the Boussinesq equations close to $(u(y), \rho(y))$.
\end{remark}

\subsection{Proof of Theorem \ref{thm:nonlaminar_rigidity}}\label{sec:proof-thm-nonlaminar-rigidity}
We shall prove the following result.
\begin{theorem}\label{thm:rigidity_crit_shear_dens}
    Let $0<\alpha<1$, $u_0,\rho_0\in C^4$ satisfying the assumptions of Theorem \ref{thm:nonlaminar_rigidity}.  Let $((u,v),\rho)$ be a steady Boussinesq solution such that $\frac{\partial_y \rho}{u}\in C^{0,\alpha}$. There exists $\ep_0>0$ such that if
    \begin{align}
        \Vert u-u_0 \Vert_{C^{2,\alpha}} + \Vert v \Vert_{H^2} + \Vert \rho - \rho_0 \Vert_{C^{2,\alpha}} < \ep_0
    \end{align}
    then $v\equiv 0$, $u=u(y)$ and $\rho=\rho(y)$.
\end{theorem}
Once Theorem \ref{thm:rigidity_crit_shear_dens} is established, we combine it with Theorem \ref{thm:rigidity_noncrit_J} to show Theorem \ref{thm:nonlaminar_rigidity}. 

\begin{proof}[Proof of Theorem \ref{thm:nonlaminar_rigidity}]
   Any solution $((u,v),\rho)$ sufficiently close to $((u_0,0),\rho_0)$ has $\partial_y u \neq 0$ on $\Gamma = \lbrace u = 0 \rbrace$ a closed graph on $x\in \T$. In view of Lemma \ref{lemma:flat_stagnation_curve}, either $\partial_y\rho \equiv 0$ on $\Gamma$ or $\partial_y \rho \equiv \ep$ on $\Gamma$ with $0<|\ep|< \ep_0$. In the first case, Theorem \ref{thm:rigidity_crit_shear_dens} applies, while in the second case we appeal to Theorem \ref{thm:rigidity_noncrit_J}, see also Remark \ref{remark:rigidity_smooth_carleman}. In both settings, we conclude that $u=u(y)$, $v=0$ and $\rho=\rho(y)$. 
\end{proof}

We next prove Theorem \ref{thm:rigidity_crit_shear_dens}. Under the conditions of the theorem, we define 
\begin{align}
    q(x,y) = \frac{\partial_y\rho(x,y)}{u(x,y)}\in C^{0,\alpha}(M), \quad Q(x,y) = \Delta u+q(x,y) \in C^{0,\alpha}(M).
\end{align}
and we note that 
\begin{align}
    u\Delta v - Q v =0.
\end{align}

Since $u_0$ is strictly monotone and vanishes at $y=0$, we can describe the zero level set of any sufficiently close function.
\begin{lemma}
    Assume that $u_n\rightarrow u_0$ in $C^{2,\alpha}$ for $0<\alpha<1$ as $n\rightarrow \infty$. Then, $\Gamma_n\coloneqq  \lbrace u_n=0 \rbrace$ is a $C^1$ graph given by $\lbrace (x,h_n(x)) : x\in \T \rbrace$ for some $h_n:\T\to(-1,1)$ such that $\Vert h_n \Vert_{C^1}\rightarrow 0 $ as $n\rightarrow \infty$. Moreover, there exists some $\tilde u_n\in C^1(M)$ such that $u_n(x,y)=(y-h_n(x))\tilde u_n(x,y)$, with further $\tilde u_n \geq c_0/2$.
\end{lemma}

\begin{proof}
    Since $u_n\rightarrow u_0$ in $C^{2,\alpha}$ for $\alpha>0$, we have $u_n \rightarrow u_0$ in $C^1(-1,1)$ and thus $\partial_y u_n\geq c_0/2$ for all $n$ sufficiently large. For all $x\in \T$, the intermediate value theorem and strict monotonicity yields the existence of a unique $y_x\in (-1,1)$ such that $u_n(x,y_x)=0$. The uniform implicit function theorem then asserts the existence of $h_n:\T\rightarrow\R$ such that $u_n(x,h_n(x))=0$ for all $x\in \T$. Next,
    \begin{align}
        c_0|h_n(x)| \leq |u_0(h_n(x))| = |u_0(h_n(x)) - u_n(x,h_n(x))| \leq \Vert u_0 - u_n \Vert_{C(-1,1)}\leq C\ep
    \end{align}
    for all $x\in \T$ shows that $\Vert h_n \Vert_{C(-1,1)}\lesssim \ep$ for $n$ large and $0 = \partial_x u_n(x,h_n(x)) + h_n'(x) \partial_y u_n(x,h_n(x))$ yields $\Vert h_n' \Vert_{C(-1,1)}\lesssim \ep$ as well.

    Since $u_n(x,h_n(x))=0$, we have that
    \begin{align}
        u_n(x,y) = \int_0^1 \frac{\d }{\d t} u_n(x, ty + (1-t)h_n(x)) \d t &= (y-h_n(x))\int_0^1\partial_y u_n (x, ty + (1-t)h_n(x)) \d t \\
        &=: (y-h_n(x)) \tilde u_n(x,y) 
    \end{align}
    from which the last statement of the lemma follows.
\end{proof}

\begin{lemma}\label{lemma:bounds_qn}
    Let $0<\alpha<1$ and assume $\rho_n\rightarrow \rho_0$ in $C^{2,\alpha}$ and $u_n\rightarrow u_0$ in $C^{2,\alpha}$ are such that $\partial_y \rho_n=0$ on $\Gamma_n = \lbrace u_n = 0 \rbrace$. Then, for $q_n = \frac{\partial_y \rho_n}{u_n}$ we have that $q_n \rightarrow q_0$ in $C^{0,\alpha}$.
\end{lemma}

\begin{proof}
    We can write $\partial_y \rho_n = (y-h_n(x)) \varrho_n(x,y)$, for 
    \begin{align}
        \varrho_n(x,y) := \int_0^1 \partial_y^2 \rho_n(x, ty + (1-t) h_n(x)) \d t 
    \end{align}
    Hence, $q_n = \frac{\varrho_n}{\tilde u_n}$ and $\Vert q_n \Vert_{C^\alpha(M)} \leq C$ for some $C>0$. It is also immediate to see that 
    \begin{align}
        \varrho_n\rightarrow \int_0^1 \rho_0''(ty)\, \d t, \qquad\qquad \tilde u_n \rightarrow \int_0^1 u_0'(ty)\, \d t
    \end{align}
    and thus $q_n \rightarrow \frac{\partial_y\rho_0}{u_0}=q_0$ in $C^{0,\alpha}$ as $n\rightarrow \infty$.
\end{proof}




The following technical lemma gives bounds on a bilinear operator that will appear when obtaining a coercivity for $v$.

\begin{lemma}\label{lemma:bound_Bn}
Let $0 < \alpha< \sigma-\frac12$ and $\frac12 < \sigma < 1$. Let $\ep_n\rightarrow 0$, $u_n\rightarrow u_0$ in $C^{2,\alpha}$ and $\rho_n\rightarrow \rho_0$ in $C^{2,\alpha}$ such that $\partial_y\rho_n = 0$ on $\lbrace u_n = 0 \rbrace$. Then, for $\ep_n>0$ the map 
\begin{align}
    \cB_n(f,g):= \int_M \frac{\Delta u_n + q_n}{u_n + i \ep_n}f \overline{g} \d x \d y
\end{align}
is uniformly bounded, namely there exists some $C>0$ such that
\begin{align}
    |\cB_n(f,g)| \leq C \Vert f \Vert_{L^2_x H^\sigma_y} \Vert g \Vert_{L^2_x H^\sigma_y}
\end{align}
for all $n\in \N$ large enough and all $f,g\in L^2_x H^{\sigma}_y\cap H_0^1(M)$. Moreover,
\begin{align}
    \lim_{n\rightarrow \infty} B_n(f,g) = B_0(f,g):= p.v.\int_M \frac{Q_0(y)}{u_0(y)}f \overline{g} \d x \d y - i\pi \frac{Q_0(0)}{u_0'(0)}\int_\T f(x,0)\overline{g(x,0)} \d x
\end{align}
for all $f,g\in H^1_0(M)$.
\end{lemma}

\begin{proof}
    Let $\varphi\coloneqq f\overline{g}$ and $Q_n\coloneqq \Delta u_n+q_n$. Note that
    \begin{align}
    \varphi(x,y)= \varphi(x,y)- \varphi(x,h_n(x)) + \varphi(x,h_n(x)) - \varphi(x,-1)
    \end{align}
    with 
    \begin{align}
        |Q_n(x,y_1)\varphi(x,y_1) - Q_n(x,y_2)\varphi(x,y_2)|\leq |y_1-y_2|^{\eta}\Vert Q_n(x,\cdot)\Vert_{C^{0,\eta}_y} \Vert \varphi(x,\cdot)\Vert_{C^{0,\eta}_y} \lesssim |y_1 - y_2|^{\eta}\Vert \varphi(x,\cdot)\Vert_{H^\sigma_y}
    \end{align}
    for some $0< \eta < \alpha < \sigma-\frac12$, where we have used that $C^{0,\eta}(-1,1)$ is an algebra, the Sobolev embedding and Lemma \ref{lemma:bounds_qn}. 
    As
    \begin{align*}
        \int_{-1}^{1} \frac{1}{u_n + i\ep_n} \d y  = \frac{\log(u_n(x,y)+i\ep_n)}{\partial_y u_n(x,y)}\Big|_{y=-1}^{y=1}  - \int_{-1}^{1} \partial_y \left( \frac{1}{\partial_y u_n}\right)\log(u_n(x,y)+i\ep_n) \d y
    \end{align*}
    is uniformly bounded for all $x\in \T$, $n\geq 1$ large enough and $0<\ep_n \leq 1$, we deduce
    \begin{align}
        |\cB_n(f,g)| &\lesssim \int_\T \Vert \varphi(x,\cdot) \Vert_{H_y^\sigma} \int_{-1}^1 |y-h_n(x)|^{\eta-1} \d y \d x  + \int_\T \Vert \varphi(x,\cdot) \Vert_{H_y^\sigma} \left|\int_{-1}^1 \frac{1}{u_n + i\ep_n} \d y \right| \d x \\
        &\lesssim \int_\T \Vert f(x,\cdot)\Vert_{H^\sigma_y} \Vert g(x,\cdot) \Vert_{H^\sigma_y} \d x \\
        &\lesssim \Vert f \Vert_{L^2_xH^\sigma_y} \Vert g \Vert_{L^2_xH^\sigma_y}
    \end{align}
    since $H^\sigma_y(-1,1)$ is an algebra.
    
    To prove the limit formula for $\cB_n(f,g)$, with $f,g\in H_0^1(M)$, we first note that 
    \begin{align}
        \lim_{n\rightarrow \infty} \int_{M}\frac{Q_n - Q_0}{u_n + i\ep_n}f \overline{g} \d x \d y = 0
    \end{align}
    since $Q_n \rightarrow Q_0$ in $C^{0,\alpha}$ for $0<\alpha< 1$. The dominated convergence theorem on the usual Cauchy integral then gives 
    \begin{align}
        \lim_{n\rightarrow \infty} \cB_n(f,g) = \cB_0(f,g) 
    \end{align}
    for all $f,g\in H_0^1(M)$.
\end{proof}

With this, we are in position to prove rigidity.

\begin{proof}[Proof of Theorem~\ref{thm:rigidity_crit_shear_dens}]
    Assume by contradiction that for all $n\geq 1$ we can find $u_n\rightarrow u_0$ in $C^{2,\alpha}$, $v_n\rightarrow 0$ in $C^2$ and $\rho_n\rightarrow \rho_0$ in $C^{2,\alpha}$ such that  $q_n:=\frac{\partial_y\rho_n}{u_n}\in C^{0,\alpha}$ and $v_n\in C^2$ satisfies
    \begin{align*}
        u_n \Delta v_n - Q_n v_n = 0,
    \end{align*}
    with $v_n(x,-1)=v_n(x,1)=0$ and $v_n\not \equiv 0$. In particular, $\Vert \nabla v_n \Vert_{L^2}\neq 0$ and we normalize $w_n := \frac{v_n}{\Vert \nabla v_n \Vert}_{L^2}$, so that $\Vert \nabla w_n \Vert_{L^2}=1$. For any $\ep > 0$ we have
    \begin{align*}
        (u_n + i \ep) \Delta w_n - Q_n w_n = i\ep\Delta w_n
    \end{align*}
    and thus
    \begin{align}\label{eq:weakform_wn}
        \int_M \nabla w_n \cdot \nabla \psi + \cB_n(w_n,\psi) = - \int_M \frac{i\ep}{u_n + i \ep}\Delta w_n \overline{\psi}.
    \end{align}
    for all $\psi\in H_0^1(M)$. Since $|\Delta w_n|^2\in L^1(M)$ (it might not be uniformly bounded in $n\geq 1$) and $u_n$ vanishes on a set of measure zero, we have with the dominated convergence theorem that
    \begin{align*}
        \lim_{\ep\rightarrow 0}\int_M \frac{\ep^2}{|u_n+i\ep|^2}|\Delta w_n|^2 = 0.
    \end{align*}
    Therefore, for all $n\geq 1$ there exists $0<\ep_n<\frac1n$ such that 
    \begin{align*}
        \left| \int_M \frac{i\ep_n}{|u_n+i\ep_n|}\Delta w_n \overline{\psi} \right| \leq \frac{\Vert \psi \Vert_{L^2}}{n}.
    \end{align*}
    Since $w_n\in H_0^1(M)$ and $\Vert \nabla w_n \Vert_{L^2}=1$, there exists some $C>0$ such that $\Vert w_n \Vert_{L^2}\leq C$, for all $n\geq 1$. Moreover, after passing to a subsequence we have that $w_n\rightharpoonup w_0$ in $H^1$ for some $w_0\in H^1$. In addition, we have $w_n \rightarrow w_0$ in $L^2_xH^\sigma_y$ due to compactness of the embedding $ H_0^1(M)\hookrightarrow L^2_xH^\sigma_y$ for $\frac12 < \sigma < 1$. Hence,
    \begin{align}
        \left| \cB_n(w_n,w_n) - \cB_0(w_0,w_0) \right| \leq \left| \cB_n(w_n,w_n) - \cB_n(w_0,w_0) \right| + \left| \cB_n(w_0,w_0) - \cB_0(w_0,w_0) \right| 
    \end{align}
    and due to linearity of $\cB_n$ and Lemma \ref{lemma:bound_Bn} we have
    \begin{align}
        \left| \cB_n(w_n,w_n) - \cB_n(w_0,w_0) \right| \leq C \Vert w_n - w_0 \Vert_{L^2_xH^\sigma_y}\left( \Vert w_n \Vert_{L^2_x H^\sigma_y} + \Vert w_0 \Vert_{L^2_xH^\sigma_y}\right) \rightarrow 0
    \end{align}
    due to strong convergence. Therefore, together with Lemma \ref{lemma:bound_Bn}, we conclude that
    \begin{align}
        \lim_{n\rightarrow\infty} \cB_n(w_n, w_n) = \cB_0(w_0,w_0)
    \end{align}
    Now, testing with $\psi= w_n\in H_0^1(M)$, noting that $\Vert \psi \Vert_{\dot{H}^1}=1$ and taking the limit we obtain
    \begin{align}
        1 + \cB_0(w_0,w_0) = 0.
    \end{align}
    On the other hand, testing with $\psi=w_0\in H_0^1(M)$, weak convergence in $H^1$ yields
    \begin{align}
        \Vert \nabla w_0 \Vert_{L^2}^2 + \cB_0(w_0,w_0) = 0.
    \end{align}
    Therefore, $\Vert \nabla w_0 \Vert_{L^2} =1$. Thus, $w_n$ both weakly converges and norm converges to $w_0\in H^1$, so that $w_n\rightarrow w\in H^1(M)$. Moreover, passing to the limit in \eqref{eq:weakform_wn} we reach
    \begin{align}
        \int_M \nabla w_0 \cdot \nabla \psi + \cB_0(w_0,\psi) = 0
    \end{align}
    for all $\psi\in H^1_0(M)$. In particular, we have that
    \begin{align}
        \Vert \nabla w_0 \Vert_{L^2}^2 + p.v.\int_M \frac{Q_0(y)}{u_0(y)}|w_0|^2  -i\pi \frac{Q_0(0)}{u_0'(0)}\int_\T |w_0(x,0)|^2 \d x = 0,
    \end{align}
    from which we deduce that $\frac{Q_0(0)}{u_0'(0)}\int_\T |w_0(x,0)|^2 \d x=0$ after taking the imaginary part. Thus, $Q_0(0)w_0(x,0)=0$ for a.e $x\in \T$ and since $Q_0w_0\in H^1$, we have that $\frac{Q_0}{u_0}w\in L^2$ and hence $w_0\in H^1_0(M)$ solves
    \begin{align}
        \Delta w_0 - \frac{Q_0}{u_0}w_0 = 0
    \end{align}
    in distributions. In particular, $w_0\in H^2(M)\cap H^1_0(M)$ is a non-trivial solution to $\cR_0w_0=0$, reaching a contradiction with $\cR_0$ having trivial kernel.
\end{proof}

\appendix

\section{Rigidity results for non-vanishing velocity or density gradient}\label{app:rigidity}
In this section we present two simplifications of Theorem \ref{informal_rigidity_1}. They concern steady states with non-vanishing velocity field and suitably oriented density, Theorem \ref{thm:mainnonzerovelocityrigidity}, which is an extension of \cite[Theorem 1.4]{CDG21}, and unstably strictly stratified density, Theorem \ref{thm:maindensityrigidity}.

\begin{theorem}\label{thm:mainnonzerovelocityrigidity}
Let $(\u, \rho, p)$ be a steady Boussinesq solution such that $\u \neq 0$ on $M$. Assume 
\begin{align}\label{eq:rho_orinetation}
    (\u^\perp\cdot\nabla \rho)\int_M u\, \d x \d y \geq 0.
\end{align}
Then, $\u$ is a shear flow and $\rho = \rho(y)$.
\end{theorem}


\begin{proof}
Our argument follows that of \cite{CDG21}, we refer the interested reader to the original proof for the precise details. We denote $\u=(u,v)$. Since $\u\neq 0$, we have $\int_M u \,\d x \d y\neq 0$. Otherwise, $$0= \int_M u\, \d x \d y =-\int_M \partial_y\psi\, \d x \d y = -2\pi(\psi(x,1) - \psi(x,-1)),$$ since $\psi$ is constant along $\partial M$. Hence, $\psi(x,1)=\psi(x,-1)$, which produces an interior critical point $p\in M$, a contradiction with $\u =\nabla^\perp\psi\neq 0$.

Assume without loss of generality that $\int_M u\, \d x \d y >0$ (we consider otherwise $\tilde \u = -\u$, noting that $(\tilde \u, \rho,p)$ is a steady Boussinesq solution and \eqref{eq:rho_orinetation} holds for $\tilde \u$ as well). Since $\u\neq 0$, there exists $C^k$ differentiable functions $\Theta$ and $G$ such that
    \begin{align}
        \rho = \Theta(\psi), \quad -P-\rho  y = G(\psi).
    \end{align}
Moreover, $\psi(x,1) < \psi(x,-1)$ and $0\leq \u^\perp\cdot\nabla \rho = -\nabla\psi\cdot\nabla\rho = - |\nabla\psi|^2 \Theta'(\psi)$, which yields $\Theta'(\psi)\leq 0$. Then, $\psi$ satisfies 
    \begin{equation}
        \begin{cases}
            \Delta\psi + y \Theta'(\psi) + G'(\psi)=0 & \\
            \psi(x,1) < \psi(x,-1). & 
        \end{cases}
    \end{equation}
    Consider next $\xi\in \R^2$ with $\xi=(\xi_1,\xi_2)$ such that $\xi_2>0$. Set
    \begin{align}
        w^\tau(x,y) = \psi(x+\tau\xi_1,y+\tau\xi_2) - \psi(x,y)
    \end{align}
    for all $x\in \overline{M}_\tau$, where
    \begin{align}
        M_\tau = \T \times (-1,1-\tau\xi_2)
    \end{align}
    We claim that 
    \begin{align}
        w^\tau(x,y) < 0, \quad \text{in } \overline{M}_\tau, \text{ for all } \tau\in (0, 2/\xi_2). 
    \end{align}
     and as $\xi_2\rightarrow 0$ we obtain 
    \begin{align}
        \psi(x+\tau\xi_1, y ) \leq \psi(x,y)
    \end{align}
    for all $\tau>0$ and all $\xi_1\in \R$. Hence, $\psi=\psi(y)$. To prove the claim, we define
    \begin{align}
        \tau_* = \inf \lbrace \tau\in(0,2/\xi_2) \text{ such that } w^{\tau'}<0 \text{ in } \overline{M}_{\tau'} \text{ for all }\tau'\in(\tau, 2/\xi_2) \rbrace.
    \end{align}
    Since $\nabla\psi\neq 0$ we have that $\psi_b > \psi(x,y) > \psi_{t}$ for all $(x,y)\in M$, where $\psi_{t}:=\psi(x,1)$ and $\psi_{b}:=\psi(x,-1)$. As $\psi$ is continuous, we have that $\tau_*< 2/\xi_2$, we shall prove that $\tau_*=0$. On the contrary, if $\tau_*>0$, then (since the space is compact) there exists $(x_0,y_0)\in \overline{M}_{\tau_*}$ such that $w^{\tau_*}(x_0,y_0) = 0$. We further observe that $(x_0,y_0)\in {M}_{\tau_*}$ since $\psi_{b} > \psi(x,y) > \psi_{t}$ for all $(x,y)\in M$. Then,
    \begin{align}
        W(x,y) \coloneqq  \psi(x+\tau_*\xi_1, y+\tau_*\xi_2) - \psi(x,y)
    \end{align}
    is such that $W < 0$ on $\partial\overline{M_{\tau_*}}$, $W \leq 0$ on ${M_{\tau_*}}$ by definition of $\tau_*$ and 
    \begin{align}
        \Delta W(x,y) + F(x,y)W(x,y) = -\tau_*\xi_2  \Theta'(\psi(x+\tau_*\xi_1, y+\tau_*\xi_2)) \geq 0,
    \end{align}
    in $M_{\tau_*}$ for some $F\in L^\infty$. Here we use the assumption that $\Theta'(\psi)\leq 0$. Hence, the strong maximum principle for non-positive functions shows that $W(x,y) < 0$ for all $(x,y)\in M_{\tau_*}$, a contradiction with $W(x_0,y_0)=0$ for some $(x_0,y_0)\in {M_{\tau_*}}$. Hence, $\tau_*=0$, the claim is proved and the theorem follows.
\end{proof}

In a different direction, the next theorem ensures that shear flows are the only admissible stationary stratified flows with positive density gradient.
\begin{theorem}\label{thm:maindensityrigidity}
Let $(\u, \rho, p)$ be a steady Boussinesq solution such that $\nabla\rho \neq 0$ on $\overline{M}$ and $\sup_{x\in\T} \rho(x,-1)\leq \inf_{x\in\T} \rho(x,1)$. Then, $\u$ is a shear flow and $\rho = \rho(y)$.
\end{theorem}

Towards the proof of the theorem, we first recall from Remark \ref{rmk:rhoctnpartialM} that steady Boussinesq configurations have constant density on each connected component of the domain boundary. The next lemma, whose proof can be found in \cite[Lemma 2.1]{CDG21} establishes functional relations for any quantity steadily transported by $\nabla^\perp\rho$.

\begin{lemma}\label{lemma:TransportImplicitFunction}
    Let $k\geq 3$. Assume that $\rho\in C^k(\overline{M})$, with $\nabla\rho\neq 0$ in $\overline{M}$ and $\rho(x,-1)=\rho_{b}$ and $\rho(x,1)=\rho_t$ for all $x\in \T$ for some $\rho_{b},\rho_t\in \R$. Let $h\in C^{k-2}(\overline{M})$ be such that $\nabla^\perp\rho \cdot \nabla h = 0$ in $\overline{M}$. Then, there exists $H\in C^{k-2}(\R)$ such that $h(p) = H(\rho(p))$ for all $p\in \overline{M}$.
\end{lemma}

The following result shows any stagnation of $\u$ must take place on closed horizontal lines.

\begin{lemma}
Let $(\u, \rho)\in C^2\times C^3$ be a stationary solution on the Boussinesq equations in $\overline{M}$ and assume that $\nabla\rho\neq 0$. Then, either $|\u|\neq 0$ or there exists $J\subseteq [-1,1]$ such that   
\begin{align}
    \lbrace |\u| = 0 \rbrace = \T\times J.
\end{align}
\end{lemma}

\begin{proof}
Let us denote $[\rho_-, \rho_+]\coloneqq \rho(\overline{M})$. According to Lemma \ref{lemma:TransportImplicitFunction}, there exists $\Psi\in C^1([\rho_{-},\rho_{+}])$ such that $\psi(p) = \Psi(\rho(p))$, for all $p\in \overline{M}$. In particular, $\u(p) = \Psi'(\rho(p))\nabla^\perp \rho(p)$ and since $\nabla\rho\neq 0$ we deduce that in $\overline{M}$ the points where $\u=0$ correspond to is the union of the density level sets $\{\rho=\rho_\ast\}$ for which $\Psi'(\rho_\ast)=0$. 

Moreover, from the equations we see that 
\begin{equation}
\Delta\psi \nabla\psi = \nabla P + e_2\rho = \nabla (P + y\rho) - y \nabla\rho
\end{equation}
and since $P+y\rho = G(\rho)$, for some $G\in C^1([\rho_-, \rho_+])$ we obtain
\begin{equation}
\Delta\psi \nabla \psi = G'(\rho)\nabla\rho - y \nabla\rho.
\end{equation}
Assume next that $\Psi'(\rho_*)=0$, for some $\rho_- \leq \rho_* \leq \rho_+$. Then, $\u = 0$ on $\Gamma_{\rho_*}$, which denotes the level set $\lbrace \rho = \rho_* \rbrace$, and there holds $0 = (G'(\rho_*) - y)\nabla\rho$. Since $\nabla\rho\neq 0$, we conclude that $y=G'(\rho_*)$. Namely, all fluid stagnation lines contained in $\overline{M}$ must be horizontal and we define 
\begin{align}
    J=\lbrace G'(\varrho)\in [-1,1] : \Psi'(\varrho)=0 \text{ for some }\varrho\in [\rho_{-}, \rho_{+}]\rbrace.
\end{align}
With this, the proof is concluded.
\end{proof}

With the above geometric characterization, we are ready to prove the theorem.

\begin{proof}[Proof of Theorem \ref{thm:maindensityrigidity}]
    We note that on $\T\times J$ we have $\u=0$ so the vorticity equation $\u\cdot\nabla \omega = -\partial_x\rho$ implies that $\rho=\rho(y)$ in $\T\times J$. Since $\nabla\rho \neq 0$, we further have that for any $y_1,y_2\in J$ with $y_1<y_2$ we have $\rho(x,y_1) < \rho(x,y_2)$, for all $x\in \T$. Moreover, we have $|\u|>0$ on $\T\times J^c$. We next argue as in Theorem \ref{thm:mainnonzerovelocityrigidity} on each connected component of $\T\times J^c$ . We note that no $C^1$ regularity of $\Theta$ is required on $\T\times \partial J^c$ to use the sliding argument, but rather continuity of $\Theta$ is enough, which is in turn ensured since $|\u|\neq 0$ in $\T\times J^c$.
\end{proof}

A direct consequence of Theorem \ref{thm:maindensityrigidity} is that shear steady states with strictly increasing stratified densities are isolated in $W^{1,\infty}$ from non-shear steady states.

\begin{corollary}\label{cor:rigidity_stratified_dens}
Let $(\u_0(y),\rho_0(y))$ be such that $\min_{y\in[-1,1]}\rho_0'(y) = \rho_{\min} > 0$. Then, all steady states $(\u,\rho)$  such that $\Vert \rho - \rho_0 \Vert_{W^{1,\infty}}< \rho_{min}$ must necessarily be shear, that is, $(\u,\rho) = (\u(y),\rho(y))$.
\end{corollary}

\begin{proof}
    The bound $\Vert \rho - \rho_0 \Vert_{W^{1,\infty}}< \rho_{min}$ ensures that $\min_{\T\times[-1,1]} \partial_y\rho>0$, so that $\rho(\cdot,-1) < \rho(\cdot, 1)$ and  $\nabla\rho\neq 0$ so that Theorem \ref{thm:maindensityrigidity} applies to $(\u,\rho)$.  
\end{proof}

\section{A continuous Crandall-Rabinowitz local bifurcation theorem}\label{CR_conti_appendix}
Any triplet $(k,\varphi,a)$ that solves $\sF(k,\varphi,a)=0$ with $\sF$ as in \eqref{eq:defsF} gives rise to $\frac{2\pi}{k}$-periodic stationary Boussinesq states with velocity $\nabla^\perp(\psi_\ep+\varphi)$ and density $\Theta_{\ep,a}(\psi_\ep+\varphi)$ that are close to the shear pair $(u_0,\rho_0)$ if $\varphi$ is small. We shall obtain solutions to $\sF(k,\varphi,a)=0$ using a Crandall-Rabinowitz bifurcation scheme, where $k$ acts as the bifurcating parameter, since we note that $\sF(k,0,a)=0$ for all $(k,a)\in \R_+\times \R$ and $\partial_\varphi \sF(k,0,a)$ has non-trivial one-dimensional kernel for $k=\sqrt{-\lambda_{\ep,a}}$ and all $a$ in an open subset of $(-1,1)$. Since we are looking for solutions that are $2\pi$ periodic, we must ensure that the bifurcating curve $k(s,a)$ attains the value $j\geq 1$ for some $(s,a)$, the parameters of the bifurcating curve of solutions. This is achieved with a uniformly continuous-in-$a$ Crandall-Rabinowitz bifurcation-type result, Theorem \ref{thm:unifCR} below.

\begin{theorem}[Uniform Implicit Function Theorem]\label{thm:unifIFT}
    Let $\cA\subset\R$ be an open subset, $X,Y,Z$ be Banach spaces and let $F, \partial_y F\in C(X\times Y\times \cA;Z)$. Assume that $F(0,0,a)=0$ and $\partial_y F(0,0,a):Y\rightarrow Z$ is an isomorphism for all $a\in \cA$. Then, for any compact subset $K\subset \cA$ there exists $0  \in U_K\subset X$ and $ 0  \in V_K\subset Y$ open subsets and a continuous map
    \begin{align}
        f:U_K\times K \rightarrow V_K
    \end{align}
    such that 
    \begin{align}
        \lbrace (x,y,a)\in U_K\times V_K\times K : F(x,y,a)=0 \rbrace = \lbrace (x,f(x,a),a) : (x,a)\in U_K\times K \rbrace.
    \end{align}
    Moreover, there holds $f(0,a)=0$ for all $a\in K$.
\end{theorem}

\begin{proof}
    Let $K\subset \cA$ be compact and let $a_0\in K\subset\cA$. Then, $F(x,y,a)$ satisfies the conditions of the implicit function theorem at $(0,0, a_0)$, for which there exist $0\in U_{a_0}\subset X$, $0\in V_{a_0}\subset Y$, $a_0\in I_{a_0}\subset \cA$ open sets and a continuous function $f_{a_0}:U_{a_0}\times I_{a_0} \rightarrow V_{a_0}$ such that $F(x,y,a)=0$ for $(x,y,a)\in U_{a_0}\times V_{a_0}\times I_{a_0}$ if and only if $y=f_{a_0}(x,a)$. In particular, since $F(0,0,a_0)=0$, there holds $f_{a_0}(0,a_0)=0$.

    We next observe that $K\subset \bigcup_{a_0\in K}I_{a_0}$ and since $K$ is compact, there exists a finite sub-cover $K \subset \bigcup_{j=1}^n I_{a_j}$, for some $\lbrace a_j \rbrace_{j=1,...,n}\subset K$. 

    For every $a_0\in K$ and every $a\in I_{a_0}$, $F(0,0,a)=0$ and the local uniqueness yields $f_{a_0}(0,a)=0$. Let $J_{a_0}$ open intervals centered at each $a_0\in K$ such that $\overline{J_{a_0}}\subset I_{a_0}$. Since $K\subset \bigcup_{a_0\in K}J_{a_0}$ is compact, there exists a finite sub-cover $K \subset \bigcup_{j=1}^n J_{a_j}$, for some $\lbrace a_j \rbrace_{j=1,...,n}\subset K$. Since each $V_{a_j}$ is an open neighborhood of the origin and we have finitely many of them, there exists $r>0$ such that
    \begin{align}
        V_K:=B_r(0)\subset \bigcap_{j=1}^n V_{a_j}.
    \end{align}
    Next, note that since $f_{a_j}$ is jointly continuous, and $V_K$ is open, $f_{a_j}^{-1}(V_K)\subset U_{a_j}\times I_j$ is open, with $(0,a)\in f_{a_j}^{-1}(V_K)$ for all $a\in \overline{J}_j$, since $f_{a_j}(0,a)=0\in V_K$ for all $a\in K$. Since $\overline{J}_{a_j}$ is compact, there exists $0\in \tilde U_{a_j}\subset U_{a_j}$ open set such that $f_{a_j}(x,a)\in V_k$, for all $x\in \tilde U_j$ and all $a\in \overline{J}_j$. Setting further
    \begin{align}
        U_K := \bigcap_{j=1}^N \tilde U_{a_j}
    \end{align}
     we conclude that $f:U_K\times K\rightarrow V_K$ given by $f(x,a)=f_{a_i}(x,a)$ for $a\in J_{a_i}$ is well-defined and jointly continuous. Indeed, for $a\in K\cap J_{a_i}\cap J_{a_j}$ and $x\in U_K$, both $f_{a_i}(x,a), f_{a_j}(x,a)\in V_K \subset V_{a_i}\cap V_{a_j}$ where local uniqueness yields $f_{a_i}(x,a)=f_{a_j}(x,a)$. It satisfies $F(x,f(x,a),a)=0$ for all $(x,a)\in U_K \times K$, with $f(0,a)=f_a(0,a)=0$, for all $a\in K$. Finally, if $(x,y,a)\in U_K\times V_K\times K$ satisfies $F(x,y,a)=0$, let $j\in \lbrace 1, ... n \rbrace$ such that $a\in J_j$ and note that $y\in V_K\subset V_{a_j}$, where local uniqueness holds and thus $y= f_{a_j}(x,a)=f(x,a)$.
\end{proof}

\begin{theorem}[Uniform CR Theorem]\label{thm:unifCR}
    Let $0\in\cA\subset\R$ be an open bounded interval and $X,Y$ be real Banach spaces. Let $0\in \mathcal{I}^\mu\subset \R$, $0\in U\subset X$ be open sets. Let $\G(\mu,\varphi,a):\mathcal{I}^\mu\times  U\times\cA\longrightarrow Y$ such that $\G, \partial_\mu\G, \partial_\varphi\G, \partial^2_{\mu,\varphi}\G\in \mathcal{C}(\cI^\mu\times U\times \cA)$ and assume that 
    \begin{align}
        \G(\mu,0,a)=0
    \end{align}
    for all $(\mu,a)\in \cI^\mu\times \cA$. For each $a\in \cA$, let 
    \begin{align}
        L_a\coloneqq  \partial_\varphi \G(0,0,a):X\rightarrow Y
    \end{align}
    be the Fr\'{e}chet derivative of $\G$ with respect to $\varphi\in X$ and further assume that
    \begin{enumerate}[label=(\arabic*)]
        \item\label{CR_condi_1} For every $a\in \cA$, there exist some non-zero $\varphi_a\in X$ such that 
        \begin{align}
            \ker(L_a) = \vspan  \lbrace \varphi_a \rbrace, 
        \end{align}
        the map $L_a$ is Fredholm of index zero and the map $a\mapsto \varphi_a$ is continuous in $X$ for all $a\in \cA$.
        \item\label{CR_condi_2} There holds the transversality condition $\partial_\mu\partial_\varphi\G(0,0,a)\varphi_a\not \in R(L_a)$ for all $a\in \cA$.
        \item\label{CR_condi_3} For all $a\in \cA$, let $\cW_a$ be a closed complement of $\cV_a\coloneqq \ker(L_a)$ in $X$. There exists a family of linear bounded isomorphisms
        \begin{align}
            T_a^X:\cW_0\subset X\rightarrow \cW_a\subset X
        \end{align}
        such that $a\mapsto T_a^X$ is continuous in $\mathcal{L}(\cW_0,X)$ in operator norm for all $a\in \cA$.
    \end{enumerate}

    Then, for any compact interval $\cA'$ containing the origin in its interior and such that ${\cA'}\subset \cA$, there exists open neighbourhoods $(0,0)\in \mathcal{U}_{\cA'}\subset \cI^\mu \times \cW_0$, $0\in \cI_{\cA'}^s\subset\R$ and continuous functions 
    \begin{align}
        (\kappa(s,a),\varpi(s,a)):\cI_{\cA'}^s\times \cA'\rightarrow\mathcal{U}_{\cA'}
    \end{align}
    with $(\kappa(0,a),\varpi(0,a))=0$ for all $a\in \cA'$ and such that

        \begin{align}
        \G(\kappa(s,a),s\varphi_a + sT_a^X\varpi(s,a),a) =0
    \end{align}
    for all $(s,a)\in \cI_{\cA'}^s\times \cA'$. Moreover, for all $a\in {\cA'}$,
\begin{align}
        \lbrace (s,\mu,w)\in (\cI_{\cA'}^s \setminus \lbrace 0 \rbrace)\times\mathcal{U}_{\cA'} : \G(\mu, s\varphi_a+sT_a^X w, a) = 0,  \rbrace &=  \lbrace (s,\kappa(s,a), \varpi(s,a)) : s\in \cI^s_{\cA'}\setminus \lbrace 0 \rbrace \rbrace.
\end{align}
    
\end{theorem}

\begin{proof}
    We work according to the usual Lyapunov-Schmidt reduction argument for the standard Crandall-Rabinowitz bifurcation theorem, see \cite{crandall1971bifurcation, kielhofer2012bifurcation}. Since $L_a$ is Fredholm of index zero and $\ker (L_a)$ is one-dimensional, we set
    \begin{align}
        X= \cV_a \oplus \cW_a, \quad \cV_a=\ker (L_a) = \vspan   \lbrace \varphi_a \rbrace, \quad \cR_a=R(L_a), 
    \end{align}
    for $\varphi_a\in X$. 
    Let $\cA'\subset \cA$ compact with $0\in \text{int}(\cA')$. Let $\cA_1$ open with $\cA'\subset \cA_1\subset \overline {\cA_1}\subset \cA$. Let $0\in B_r \subset U$ for some $r>0$. There exists $0\in\cI^s\subset \R$ and $0\in\cW'\subset \cW_0$ open sets such that $s(\varphi_a+T_a^Xw)\in B_r\subset U$ for all $(s,w,a)\in \cI^s\times \cW'\times \overline{\cA_1}$. We define the map $\cH:\mathcal{I}^s\times \mathcal{I}^\mu\times \cW'\times \cA_1 \rightarrow Y$ given by
    \begin{align}
         \cH(s,\mu,w,a) \coloneqq  \int_0^1 \partial_\varphi \G(\mu, ts\hat\varphi(a,w),a)\hat\varphi(a,w) \d t, \quad \hat\varphi(a,w) \coloneqq \varphi_a + T_a^X w.
    \end{align}
    which is well-defined because $s\hat\varphi(a,w)\in B_r\subset U$ for all $(s,w,a)\in \cI^s\times \cW'\times \cA_1$. It is immediate to note that 
    \begin{align}
        \cH(s,\mu,w,a) = \begin{cases}
            s^{-1} \G(\mu, s\hat\varphi(a,w),a), & s\neq 0, \\
            \partial_\varphi \G(\mu, 0,a)\hat\varphi(a,w), & s = 0.
        \end{cases}
    \end{align}
    Moreover, $\cH(0,0,0,a)=0$ and 
    \begin{align}
        D_{(\mu,w)} \cH(0,0,0,a)(\tau,\varpi) = \tau (\partial_\mu\partial_\varphi)\G(0,0,a)\varphi_a + \partial_\varphi\G(0,0,a)T_a^X\varpi
    \end{align}
    for $(\tau,\varpi)\in \R\times\cW_0$. Note that   $D_{(\mu,w)} \cH(0,0,0,a):\R\times \cW_0 \rightarrow Y$ is an isomorphism between $\R\times \cW_0$ and $Y$. Indeed, $ \partial_\varphi\G(0,0,a)T_a^X:\cW_0\rightarrow \cR_a$ is a bounded isomorphism with bounded inverse between closed Banach spaces and the transversality condition  $(\partial_\mu\partial_\varphi)\G(0,0,a)\varphi_a \not \in \cR_a$ supplies the missing one-dimensional direction. 

    Since $\cH$, $D_{(\mu,w)}\cH\in C(\cI_s\times \cI_\mu\times \cW'\times \cA_1)$ satisfies the conditions of the Uniform Implicit Function Theorem \ref{thm:unifIFT} for $x=s$ and $y=(\mu,w)$, for $\cA'\subset \cA_1$ compact with $0\in \text{int}(\cA')$ there exist $0\in \cI_{\cA'}^s\subset \cI^s$, $(0,0)\in \mathcal{U}_{\cA'}\subset \cI^\mu \times \cW'$ and continuous functions
    \begin{align}
        (\kappa(s,a),\varpi(s,a)):\cI_{\cA'}^s\times {\cA'}\rightarrow\mathcal{U}_{\cA'}
    \end{align}
    with $(\kappa(0,a),\varpi(0,a))=0$ for all $a\in {\cA'}$ and such that
    \begin{align}
        \cH(s,\kappa(s,a), \varpi(s,a),a) =0,
    \end{align}
    for all $(s,a)\in \cI_{\cA'}^s\times {\cA'}$. Moreover, 
    \begin{align}
        \lbrace (s,\mu,w,a)\in \cI_{\cA'}^s\times \mathcal{U}_{\cA'}\times {\cA'} : \cH(s,\mu,w,a)=0 \rbrace = \lbrace (s,\kappa(s,a), \varpi(s,a),a) : (s,a)\in \cI_{\cA'}^s\times {\cA'} \rbrace
    \end{align}
    which completes the proof.
    \end{proof}

\bibliographystyle{abbrv}
\bibliography{refs}
\end{document}